\documentclass[11pt]{article}

\usepackage{verbatim,latexsym,amsfonts,amsmath,amssymb,graphicx,fancyhdr,hyperref,asymptote,enumitem,bbm}
\usepackage{appendix,latexsym,amsfonts,amsmath,amssymb,graphicx,hyperref,amsthm,soul,verbatim,authblk}
\usepackage[framemethod=tikz]{mdframed}

\newcommand{\la}{\langle}
\newcommand{\ra}{\rangle}
\newcommand{\EE}{\mathbb{E}}

\newcommand{\PP}{\mathbb{P}}

\newcommand{\R}{\mathbb{R}}
\newcommand{\C}{\mathbb{C}}
\newcommand{\Q}{\mathbb{Q}}

\newcommand{\HH}{\mathbb{H}}
\newcommand{\N}{\mathbb{N}}
\newcommand{\D}{\mathbb{D}}
\newcommand{\Z}{\mathbb{Z}}

\newcommand{\pa}{\partial}

\newcommand{\K}{{\cal K}}

\newcommand{\F}{{\cal F}}

\newcommand{\Lo}{{\cal L}}
\newcommand{\lotimes}{{\overleftarrow{\otimes}}}
\newcommand{\rotimes}{{\overrightarrow{\otimes}}}
\newcommand{\ulin}{\underline}
\newcommand{\lr}{\leftrightarrows}
\newcommand{\LR}{\Leftrightarrow}
\newcommand{\rl}{\rightleftarrows}
\newcommand{\ccw}{\circlearrowleft}
\newcommand{\cw}{\circlearrowright}
\newcommand{\ind}{\mathbbm{1}}

\def\eps{\varepsilon}
\def\til{\widetilde}
\def\ha{\widehat}
\def\sem{\setminus}
\def\lin{\overline}

\def\up{\upsilon}
\def\Up{\Upsilon}

\def\wto{\stackrel{\rm w}{\to}}
\def\tl{\vartriangleleft}

 \DeclareMathOperator{\Cont}{Cont}
 
\DeclareMathOperator{\rad}{rad} \DeclareMathOperator{\doub}{doub}
 \DeclareMathOperator{\diam}{diam}
\DeclareMathOperator{\dist}{dist} 
\DeclareMathOperator{\hcap}{hcap} 
\DeclareMathOperator{\Imm}{Im }

\DeclareMathOperator{\mA}{m}

\theoremstyle{plain}
\newtheorem{Theorem}{Theorem}[section]
\newtheorem{Lemma}[Theorem]{Lemma}
\newtheorem{Corollary}[Theorem]{Corollary}
\newtheorem{Proposition}[Theorem]{Proposition}

\theoremstyle{definition}
\newtheorem{Definition}[Theorem]{Definition}
\newtheorem{Remark}[Theorem]{Remark}
\newtheorem{Example}[Theorem]{Example}
\numberwithin{equation}{section}
\numberwithin{equation}{section}
\newcommand{\BGE}{\begin{equation}}
\newcommand{\BGEN}{\begin{equation*}}
\newcommand{\EDE}{\end{equation}}
\newcommand{\EDEN}{\end{equation*}}

\begin{document}
\title{SLE$_\kappa(\rho)$ bubble measures
}
\author{Dapeng Zhan\thanks{zhan@msu.edu}
}
\affil{Michigan State University}
\date{\today}
\maketitle

\begin{abstract}
For $\kappa>0$ and $\rho>-2$, we construct a $\sigma$-finite measure,  called a rooted SLE$_\kappa(\rho)$ bubble measure, on the space of curves in the upper half plane $\mathbb H$ started and ended at the same boundary point, which satisfies some SLE$_\kappa(\rho)$-related domain Markov property, and is the weak limit of SLE$_\kappa(\rho)$ curves in $\mathbb H$ with the two endpoints both tending to the root.  For $\kappa\in(0,8)$ and $\rho\in ((-2)\vee(\frac\kappa 2-4),\frac\kappa 2-2)$, we derive decomposition theorems for the rooted  SLE$_\kappa(\rho)$ bubble with respect to the Minkowski content measure of the intersection of the rooted SLE$_\kappa(\rho)$ bubble with $\mathbb R$, and   construct unrooted SLE$_\kappa(\rho)$ bubble measures.
\end{abstract}
\tableofcontents
\section{Introduction}
The Schramm-Loewner evolution (SLE)  introduced by Oded Schramm in 1999 (\cite{S-SLE})   is a
one-parameter ($\kappa\in(0,\infty)$) family of random fractal curves, which has close relation with  two-dimensional statistical lattice models, Brownian motion, Gaussian free field and Liouville quantum gravity. The reader is referred to  \cite{Law1} and \cite{RS} for basic properties of SLE.  

SLE are characterized by conformal invariance and domain Markov property (DMP). There are several versions of SLE, one of which is chordal SLE.  A chordal SLE$_\kappa$ curve $\eta$ grows in a simply connected domain $D$ from one prime end $a$ towards another prime end $b$, and satisfies the following DMP: if $\tau$ is a stopping time for $\eta$, then conditionally on the part of $\eta$ before $\tau$ and the event that $\eta$ is not complete at the time $\tau$, the part of $\eta$ after $\tau$ is a chordal SLE$_\kappa$ curve from $\gamma(\tau)$ (as a prime end) to $b$ in a connected component of $D\sem \eta[0,\tau]$.

Inspired by the Brownian loop measure constructed in \cite{LW-loop}, people have been working on the
construction of SLE loops, which locally look like SLE curves. There are rooted and unrooted SLE loops. A rooted SLE loop starts and ends at a marked point called the root, while an unrooted SLE loop does not pass through a marked point.  Unrooted SLE$_{8/3}$ and SLE$_2$ loops were respectively constructed in \cite{Wer-loop} and \cite{SLE2loop}. The SLE$_\kappa$ bubble measure (around some $z\in\HH$) defined in \cite{CLE} for $\kappa\in (\frac 83,4]$  and the whole-plane space-filling SLE$_\kappa$ from $\infty$ to $\infty$ constructed in \cite{mating} for $\kappa\ge 8$ are rooted SLE$_\kappa$ loops. The whole-plane space-filling SLE$_\kappa$ curves were also constructed for $\kappa\in (4,8)$, in which case they do not locally look like an ordinary SLE$_\kappa$.

For $\kappa\in (0,8)$, several types of SLE$_\kappa$ loops were constructed in \cite{loop}, which include rooted and unrooted SLE$_\kappa$ loops in the Riemann sphere $\ha\C$, rooted and unrooted SLE$_\kappa$ bubbles in simply connected domains, and unrooted SLE$_\kappa$ loops in general Riemann surfaces.
A rooted  SLE$_\kappa$ loop $\eta$ in $\ha\C$ satisfies the following DMP.  If $\tau$ is a positive stopping time for $\eta$, then conditionally on the part of $\eta$ before $\tau$ and the event that $\eta$ is not complete at the time $\tau$, the part of $\eta$ after $\tau$ is a chordal SLE$_\kappa$ curve from $\eta(\tau)$ to the root of $\eta$ in a connected component of $\ha\C\sem \eta[0,\tau]$. A rooted SLE$_\kappa$ bubble grows in the prime-end closure of a simply connected domain $D\subsetneqq\C$ with the root being a prime end, and satisfies almost the same DMP as above except that $\ha\C\sem \eta[0,\tau]$ is replaced by $D\sem \eta[0,\tau]$.

Except for SLE bubble in \cite{CLE} and the whole-plane space-filling SLE$_\kappa$ in \cite{mating}, the SLE loops described as above are not probability measures, but $\sigma$-finite infinite measures. If one  restricts an SLE$_\kappa$ bubble measure in $\HH$ rooted at $0$ in \cite{loop} to the event that the bubble disconnects a marked point $z\in\HH$ from $\infty$ and then normalizes the restricted measure to get a probability measure, then the resulted measure is the SLE$_\kappa$ bubble around $z$ in \cite{CLE}.

The construction of rooted SLE$_\kappa$ loops and bubbles in \cite{loop} was based on the decomposition of chordal SLE$_\kappa$ in \cite{decomposition}  with respect to its Minkowski content (cf.\ \cite{LR}), which asserts that the following two sampling procedures result in the same measure on curve-point pairs: (i) first sampling a chordal SLE$_\kappa$ curve in $D$ from $a$ to $b$ and then sampling a point on the curve according to a measure closely related to the Minkowski content of the curve of dimension $1+\frac\kappa 8$, and (ii) first sampling a point in $D$ according to the SLE$_\kappa$ Green's function and then sampling a two-sided radial SLE$_\kappa$ in $D$ from $a$ to $b$ passing through that point. 

The unrooted SLE$_\kappa$ loops in $\ha\C$ were constructed by integrating the laws of the SLE$_\kappa$ loops rooted at different $z\in\C$ against the Lebesgue measure and then unweighting the resulted measure by the Minkowski content of the loop. The unrooted SLE$_\kappa$ bubbles in the upper half plane $\HH$ were constructed for $\kappa\in(4,8)$ by integrating the laws of the SLE$_\kappa$ loops in $\HH$ rooted at different $x\in\R$ against the Lebesgue measure and then unweighting the resulted measure by the Minkowski content of the intersection of the loop with $\R$ (cf.\ \cite{Mink-real}). 

As the title suggests, we are going to construct a loop version of SLE, which extends the SLE bubbles in \cite{loop}, and is closely related to the SLE$_\kappa(\rho)$ curves (cf.\ \cite{LSW-8/3}). We will construct rooted SLE$_\kappa(\rho)$ bubbles in $\HH$ for all $\kappa>0$ and $\rho>-2$ in Section \ref{Section-rooted-loop}.   In the case $\kappa>4$ and $\rho\in (-2,\frac\kappa 2-4]$, the law of the bubble is a probability measure and satisfies conformal invariance (Theorem \ref{Thm-loop-probability}). In the case $\rho>(-2)\vee (\frac\kappa 2-4)$, the law of the bubble is a $\sigma$-finite infinite measure and satisfies conformal covariance (Theorem \ref{Thm-loop-sigma-finite}). In both cases, an SLE$_\kappa(\rho)$ bubble $\eta$ is characterized by the following DMP: if $\tau$ is a positive stopping time for $\eta$, then conditionally on the part of $\eta$ before $\tau$ and the event that $\eta$ is not complete at the time $\tau$, the part of $\eta$ after $\tau$ is an SLE$_\kappa(\rho)$ curve from $\gamma(\tau)$ to the root of $\eta$ in a connected component of $\HH\sem \eta[0,\tau]$.
In addition, we prove that the SLE$_\kappa(\rho)$ measure in $\HH$ started from $r>0$, aimed at $-r$, and with force point $r^+$, suitably  rescaled, tends to the counterclockwise SLE$_\kappa(\rho)$ bubble measure in $\HH$ rooted at $0$ ``weakly'' as $r\to 0^+$ (Theorem  \ref{weak-1}).


It has been demonstrated in \cite{ARS21} that some SLE-type bubble measure can be coupled with Liouville CFT naturally via  quantum zipper.  The particular bubble zipper result proved in \cite{ARS21} is used to prove the FZZ formula in Liouville CFT. Although the law of the bubble measure is not identified in \cite{ARS21}, it was shown that it is the limit of a chordal SLE$_\kappa(\gamma^2/2-2)$ whose two endpoints are  quantum length $\eps$ away from each other. Therefore, it  corresponds to SLE$_\kappa(\rho)$ bubble measure with $\rho=\gamma^2/2-2$. In a forthcoming work \cite{DW}, Da Wu will extend the bubble zipper result to general $\rho$.

The construction  in this paper is more straightforward than that in \cite{loop} as it does not rely on the decomposition of SLE$_\kappa(\rho)$. We use only the transition density of radial Bessel processes developed in \cite{tip}.
On the other hand, for $\kappa\in (0,8)$ and $\rho\in ((-2)\vee (\frac\kappa 2-4),\frac\kappa 2-2)$, we  derive a decomposition theorem for SLE$_\kappa(\rho)$ according to the Minkowski content of its intersection with $\R$ in Section \ref{section-decomp} (Theorem \ref{Thm-decomposition}), and then use it to  derive a decomposition theorems for the  rooted SLE$_\kappa(\rho)$ bubbles in Section \ref{section-proof} (Theorems \ref{Thm-decomposition-loop} and \ref{Thm-decomposition-loop-infty}).
Finally, in Section \ref{section-unrooted} we define the unrooted SLE$_\kappa(\rho)$ bubbles in $\HH$ by integrating the rooted SLE$_\kappa(\rho)$ bubble measures in $\HH$ with different roots $w\in\R$ against the Lebesgue measure and then unweighting the resulted measure by the Minkowski content of the intersection.


In the next section we recall/introduce some important notion, symbols and results including kernels, prime ends, Minkowski contents, Minkowski content measures, operations on continuous processes, locally absolutely continuity, modulo time-change (MTC) functions/curves, time-change invariant (TCI) times, rooted and unrooted MTC loops, Loewner curves and related subdomains and prime ends, SLE$_\kappa(\ulin\rho)$ curves, and radial Bessel processes. 

The new notion and symbols  in Section \ref{section-preli}  play crucial  roles in making the statements of the paper completely rigorous. For example, to rigorously describe the DMP for rooted SLE$_\kappa(\rho)$ bubbles, we encounter the following subtlety. First, the non-boundary-filling SLE$_\kappa(\rho)$ bubble measures are infinite measures. Second, an SLE$_\kappa(\rho)$ curve in a general simply connected domain is, strictly speaking, a continuous random curve {\it modulo time-change} in the {\it prime-end closure} of the domain instead of a parametrized random curve in the usual closure of the domain in $\C$ or $\ha\C$. Third,  stopping times for the DMP are defined for parametrized processes instead of processes modulo time-change. Forth, the DMP for SLE$_\kappa(\rho)$ involves a force point in addition to the tip and target of the curve, whose rigorous description requires prime ends. Moreover, we may not be able to parametrize the curve using the Minkowski content of the curve itself because the intersection of the curve with the domain boundary may blow up the Minkowski content. Because of these facts, we have to introduce some non-standard notion and symbols. 
For the reader's convenience, we make a list of global symbols/notion used in the paper.

\begin{itemize}
  \item Section \ref{Section-MK}: $f(\mu)$, $w\cdot \mu$, $\wto$, $\mu\rotimes \nu$, $\nu\lotimes \mu$.
  \item Section \ref{Section-PE}: $\ha D$, $\ha\pa D$, $w^+$, $w^-$, $\lin D^\#$,  $\pa^\# D$, $(a,b)_{\ha\pa D}$, $[a^+,b]_{\ha\pa D}$, $[a^+,b^-]_{\ha\pa D}$, $\ha\pa_w D$.
  \item Section \ref{Section-MC}: $\mA$, $S^r$, $\Cont_d$, $\lin\Cont_d$, $\ulin \Cont_d$, ${\cal M}^d_S$, ${\cal M}_S$.
  \item Section \ref{section-lifetime}: $\Sigma$, $\ha T_f$, $\K_t$, $\K_\tau$, $\Sigma_{\ha T}$, $\Sigma_{\ha T}^{<\infty}$, $\Sigma_{\oplus}$, $\Sigma_{\oplus}^{<\infty}$, $\oplus$, $\ha\oplus$, $\Sigma_{(0)}$, $\oplus_+$, $\ha\oplus_+$, $\Sigma_{\oplus,\tau}$, $\oplus_\tau$, $\Sigma_t$, $\pi_t$, $\F^*=(\F^*_t)_{t\ge 0}$, $\F=(\F_t)_{t\ge 0}$, $\F_\infty$, $\F^\mu$,  $\tl$, RN process, time-change, $\til\Sigma$, $\til\Sigma^2$, $\pi$, $\pi^2$,   MTC curve $[f]$, $\til\Sigma_{\ha T}$, $\til\Sigma_{(0)}$, $\til\Sigma_{\oplus}$, time-reversal map $\cal R$, TCI time/set, $\til\F_\tau$, $\til\F_{\ha T}$, $\Sigma^E$, $\til\Sigma^E$, $\dist_{\til\Sigma^E}$, $\tau_{g;t}$ ((\ref{tau-g-formula'})), $\tau_{[g];\sigma}$ ((\ref{tau-g-formula''})), rooted loop, rooted/unrooted MTC loop,   $[[f]]$, $\til\pi$.
  \item Section \ref{section-Loewner}: $\HH$-hull, $g_K$, $\hcap$, $\rad$, $a_K$, $b_K$, $\ha a_K$, $\ha b_K$, $C_K$, $D_K$, $C_K^a$, $D_K^b$, CP/AP/MTC Loewner curve, $W^*_t$, $g^*_t$, $\tau^*_z$, $K^*_t$, $H^*_t$, $a^*_t$, $b^*_t$, $\ha a^*_t$, $\ha b^*_t$, $C^*_t$, $D^*_t$, $I^*_t$, $\eta^*(t)$ (``$*$'' could be empty, a driving function $W$, or a CP/AP Loewner curve $\eta$ in $\HH$), $\Sigma^{\cal L}$, $\cal L$, $\ha{\cal L}$, $D(S;T)$, $D(S;w)$, $D(\eta|t_0^+)$,  $\eta(t_0^+)$, $D([\eta]|\tau^+)$, $I^\eta_{t^+;D}$, $a^\eta_{t^+;D}$, $b^\eta_{t^+;D}$, $\ha a^\eta_{t^+;D}$, $\ha b^\eta_{t^+;D}$, $I^{[\eta]}_{\tau^+;D}$, $a^{[\eta]}_{\tau^+;D}$, $b^{[\eta]}_{\tau^+;D}$, $\ha a^{[\eta]}_{\tau^+;D}$, $\ha b^{[\eta]}_{\tau^+;D}$, ${\cal P}$, $[v]^\eta_{t^+;D}$, $[v]^\eta_t$, $[v]^{[\eta]}_{\tau^+;D}$ ($\eta$/$[\eta]$ is an AP/MTC Loewner curve in $D$).
  \item Section \ref{section-SLE}: CP SLE$_\kappa(\ulin\rho)$ at $(w;\ulin v)$, $\PP^{w;\ulin v}_{\kappa;\ulin\rho}$, $\la v\ra^W_t$, (MTC) SLE$_\kappa(\ulin\rho)$ at $D:(w_0;\ulin v)\to w_\infty$.
  \item Section \ref{Section-Bessel}:  $\nu_x^{\delta_+,\delta_-}$, $\alpha_+$, $\alpha_-$, $w_{\alpha_+,\alpha_-}$,  $p_t(x,y)$, $p_\infty(x)$, $\mu_{\R}^{\delta_+,\delta_-}$, $\Sigma_{\R}$, $\F^{\R}$, $\mu_x^{\delta_+,\delta_-}$,  $\alpha_0$, $M^Z_t$, $\Sigma_\tau$, $\til\Sigma_\tau$, $\pi_{t',t}$, $\Sigma^{\R}_t$, $\Sigma^{\R}_\tau$, $\pi^{\R}_t$, $\pi^{\R}_{t',t}$.
  \item Section \ref{Section-rooted-loop}: $\til\mu^D_{(w;v)\to x}$, $\til\mu^D_{\gamma|s_0^+}$ ((\ref{til-mu-gamma})), $\til\mu^D_{[\gamma]|\sigma^+}$ ((\ref{til-mu-gamma-[]})), $\til\mu^{\HH}_{0\ccw}$, $\til\mu^{\HH}_{0\cw}$, $\til\mu^{\HH}_{w\ccw}$, $\til\mu^{\HH}_{w\cw}$, $Q^W_t$ ((\ref{QWt})), $\nu^*$, $Q^\eta_t$ ((\ref{Q-eta-t})), $\alpha$ ((\ref{alpha})).
  \item Section \ref{section-decomposition}: $M^W_t$ ((\ref{MWt})),  $\rho_\Sigma$, $I_\infty$, $ \PP_{\kappa;\ulin\rho,\oplus}^{w;\ulin v,u}$ ((\ref{dP-kappa-rho'-U})), conditional CP/MTC  SLE$_\kappa(\ulin\rho)$, $G_C(w,\ulin v;u)$, $C_{\kappa;\rho}$, $G(w,\ulin v;u)$ ((\ref{C-kappa-rho},\ref{Gwvu})), $M^{W;u}_t$ ((\ref{Mu-express})),   $M^{W;S}_t$ ((\ref{MWS1},\ref{MWS2})),  $\til\mu^D_{(w_0;v)\to w_\infty}$, $\til\nu^D_{(w_0;v)\to u\to w_\infty}$,   $\til\nu^D_{(w_0;v_1,v_2)\to w_\infty}$,   $\til\mu^D_{w_0\to w_\infty}$, $\til\nu^D_{w_0\to u\to w_\infty}$,   $\til\nu^D_{w_0\to w_\infty}$, $\til\nu^{D}_{w_0\lr w_\infty}$, $\til\nu^{D}_{w_0\rl w_\infty}$, $\til\nu^{D;u}_{\gamma|s_0^+}$, $\til\nu^{D;u}_{[\gamma]|\sigma^+}$ ((\ref{til-mu-gamma-u})),   $\Gamma$, $\Gamma_x$, $\ha\nu^{\HH}_{w\lr u}$, $\ha\mu^{\HH}_{\ccw}$, $\ha\mu^{\HH}_{\cw}$, $\ha\mu^D_{\ccw}$, $\ha\mu^D_{\cw}$. 
\end{itemize}

\section*{Acknowledgement}
The author thanks Xin Sun for valuable suggestions on the paper.

\section{Preliminary} \label{section-preli}
\subsection{Measures and kernels} \label{Section-MK}
For a measure $\mu$, we use $f(\mu)$ (rather than $f_*(\mu)$) to denote the pushforward of   $\mu$ under $f$, and use $w\cdot \mu$ or $w(x)\cdot \mu(dx)$ to denote  the measure obtained by weighting $\mu$ by $w$.

We use    ``$\wto$'' to denote weak convergence.   Recall that for bounded measures $\mu_n$, $n\in\N$, and $\mu$ defined on some metric space $E$, $\mu_n\wto \mu$ iff for any $f\in C_b(E,\R)$, $\mu_n(f)\to \mu(f)$.
We will use the following facts.
 \begin{itemize}
   \item [(F1)]  When $\mu_n$ and $\mu$ are all probability measures, $\mu_n\wto \mu$ iff  $\forall\eps>0,\exists N\in\N$, such that $\forall n>N,\exists$ a coupling $X_n$ and $X$ of $\mu_n$ and $\mu$   such that $\PP[\dist_E(X_n,X)>\eps]<\eps$.
   \item [(F2)] When all $\mu_n$ and $\mu$ are nonzero measures, $\mu_n\wto \mu$ iff $|\mu_n|\to |\mu|$ and $\mu_n/|\mu_n|\wto \mu/|\mu|$. 
\item [(F3)] If $\mu_n\wto\mu$, and $G\in  {\cal B}(E)$ satisfies $\mu(\pa_E G)=0$, then $\ind_G \mu_n\wto \ind_G \mu$.
 \end{itemize}

We now review the concept of kernels in \cite{Foundation}.
Suppose $(U,{\cal U})$ and $(V,{\cal V})$ are two measurable spaces. A kernel from $(U,{\cal U})$ to $(V,\cal V)$ is a map $\nu:U\times {\cal V}\to[0,\infty]$ such that (i) for every $u\in U$, $\nu(u,\cdot)$ is a measure on $(V,\cal V)$, and (ii) for every $F\in\cal V$, $\nu(\cdot,F)$ is $\cal U$-measurable.
Let $\mu$ be a $\sigma$-finite measure on $(U,\cal U)$. Let ${\cal U}^\mu$ be the $\mu$-completion of $\cal U$. A $\mu$-kernel from $(U,{\cal U})$ to $(V,\cal V)$ is a kernel from $(U^\mu,{\cal U}^\mu\cap U^\mu)$ to $(V,\cal V)$, where $ U^\mu\subset U$ is such that $U\sem U^\mu$ is a $\mu$-null set. The $\mu$-kernel is said to be   $\sigma$-finite if there is a sequence $F_n\in\cal V$, $n\in\N$, with $V=\bigcup F_n$ such that for any $n\in\N$, and $\mu$-a.s.\  every $u\in U$, $\nu(u,F_n)<\infty$.

If $\mu$ is a $\sigma$-finite measure on $(U,\cal U)$, and $\nu$ is a $\sigma$-finite $\mu$-kernel from $(U,{\cal U})$ to $(V,\cal V)$, then $\mu\rotimes \nu$ is the unique measure on ${\cal U}\times{\cal V}$ such that $\mu\rotimes \nu(E\times F)=\int_E \nu(u,F)d\mu(u)$ for all $E\in\cal U$ and $F\in \cal V$.
We sometimes write $\mu\rotimes \nu$ as $\mu(du) \rotimes \nu(u,\cdot)$ to emphasize how the kernel $\nu$ depends on $u$. We use $\nu\lotimes\mu$ or $\nu(u,\cdot)\lotimes \mu(du)$ to denote the pushforward measure on ${\cal U}\times{\cal V}$ of $\mu\rotimes \nu$ under the map $(v,u)\mapsto (u,v)$.

\subsection{Prime ends} \label{Section-PE}

The concept of prime ends describes the boundary correspondence between simply connected subdomains of  the Riemann sphere $\ha\C=\C\cup\{\infty\}$ under conformal maps.
In this paper, a simply connected domain is always a connected open subset of $\ha\C$ that is conformally equivalent to the unit disc $\D=\{z\in\C:|z|<1\}$. Under this definition, $\ha\C$ and $\C$ are not simply connected domains. The prime ends in \cite{Ahl} is defined using crosscuts and extremal length. For the purpose of this paper, we define prime ends in the following equivalent way.

Let $D$ be a simply connected domain. The prime-end closure of $D$, denoted by $\ha D$, is the family of couples $(f,z)$, where $f$ is a conformal map from $\D$ onto $D$ and $z$ is a point in $\lin\D$, modulo the equivalence relation: $(f_1,z_1)\sim (f_2,z_2)$ when the continuation of $f_1^{-1}\circ f_2$ from $\D$ to $\lin\D$ maps $z_2$ to $z_1$. The $\ha D$ is equipped with the topology such that for any fixed conformal map $f_0$, the map $z\mapsto [(f_0,z)]$ is a homeomorphism between $\lin\D$ and $\ha D$. Such topology does not depend on the choice of $f_0$. We understand $D$ as a subset of $\ha D$ by identifying every $z\in D$ with the point $[(f,f^{-1}(z))]\in\ha D$.
The set $\ha D\sem D=: \ha\pa D$ is called the prime-end boundary of $D$, and every element in $\ha\pa D$ is called a prime end of $D$. If $h$ maps one simply connected domain $D_1$ conformally onto another $D_2$, then $h$ induces a homeomorphism $h$ between $\ha D_1$ and $\ha D_2$  by $h([(f,z)])=[(h\circ f,z)]$, which sends every prime end of $D_1$  to a prime end of $D_2$.

For every prime end $w$ of $D$, we define two pseudo prime ends $w^+$ and $w^-$ of $D$, which  lie immediately next to the genuine prime end $w$ in the counterclockwise (ccw) and clockwise (cw) direction, respectively. If $w$ is a prime end of $D_1$, and $f$ maps $D_1$ conformally onto $D_2$, then $f(w^\pm)$ is understood as the pseudo prime end $f(w)^\pm$ of $D_2$.


We write $\lin D^\#$ and $\pa^\# D$ for the closure and boundary of $D$, respectively, in $\ha\C$. 
Let $w\in \ha\pa D$ and $w'\in\pa^\# D$. If  $D\ni z_n\to w$ in $\ha D$ implies that $D\ni z_n\to w'$ in $\ha \C$, then we say that $w$ determines $w'$. If the converse is also true, then we identify the prime end $w$ with the point $w'$.

\begin{Example}
  If $\pa^\# D$ is locally connected, then every   $w\in\ha\pa D$ determines a point on $\pa^\# D$. If $D$ is a Jordan domain in $\ha\C$, then every $z\in \pa^\# D$ is identified with a prime end of $D$. In the case $D=\HH:=\{z\in\C:\Imm z>0\}$, $\pa^\# \HH=\ha\pa \HH=\ha\R:=\R\cup\{\infty\}$.
\end{Example}

Let $D_1\subset D_2$ be two simply connected domains in $\ha\C$. Let $w_1\in \ha\pa D_1$ and $w_2\in \ha D_2$. If $D_1\ni z_n\to w_1$ in $\ha D_1$ implies that $D_1\ni z_n\to w_2$ in $\ha D_2$, then we say that $w_1$ determines $w_2$. If the converse is also true, then we identify $w_1$ with $w_2$.
Let $D_1, D_2$ be two simply connected domains in $\ha\C$, and $w_j\in\pa\ha D_j$, $j=1,2$. If for $j=1,2$, the intersection of any neighborhood of $w_j$ in $D_j$ with $D_{3-j}$ is a neighborhood of $w_{3-j}$ in $D_{3-j}$, then we identify $w_1$ with $w_2$ and say that $D_1$ and $D_2$ share the prime end $w_1=w_2$.

If $D$ is a simply connected domain with two distinct prime ends $a,b$, the  prime end interval $(a,b)_{\ha \pa D}$ is the open interval on $\ha\pa D$ from $a$ to $b$ in the ccw direction. Specifically, if $f$ maps $\HH$ conformally onto $D$ such that $f(\infty)=b$, then $(a,b)_{\ha \pa D}\subset \ha \pa D$ is the $f$-image of the real interval $(f^{-1}(a),\infty)$.  We call $a$ and $b$ respectively the left and right end points of $(a,b)_{\ha \pa D}$. We also define  $[a^+,b]_{\ha\pa D}=\{a^+\}\cup (a,b)_{\pa D}\cup \{b\}$ and  $[a,b^-]_{\ha\pa D}=\{a\}\cup (a,b)_{\pa D}\cup \{b^-\}$. For $w\in\ha\pa D$, we define $\ha\pa_w D=(\ha \pa D\sem\{w\})\cup \{w^+,w^-\}$.

\subsection{Minkowski content} \label{Section-MC}
Let $\mA$  denote the Lebesgue measures on $\R$. Let $n\in\N$. For $S\subset\R^n$ and $r>0$, we use $S^r$ to denote the $r$-neighborhood of $S$, i.e., $\{x\in \R^n:\dist(x,S)<r\}$. Let $d\in (0,n)$. The $d$-dimensional Minkowski content of $S$ is defined by
\BGE \Cont_d(S)=\lim_{r\to 0^+} r^{d-n} \mA^n(S^r),\label{Cont=}\EDE
provided that the limit exists. For simplicity we omit the extra multiplicative constant in the literature. The Minkowski content of a set may not exist in general. But  we may always define the upper and lower Minkowski contents $\lin \Cont_d(S)$ and $\ulin\Cont_d(S)$   using (\ref{Cont=}) with $\limsup$ and $\liminf$, respectively, in place of $\lim$.

It is straightforward to show that
\BGE \lin\Cont_d(S_1\cup S_2)\le \lin\Cont_d(S_1)+\lin\Cont_d(S_2).\label{Cont+}\EDE
By \cite[Lemma 6.2]{Green-kappa-rho},
\BGE \ulin\Cont_d(S_1\cup S_2)\ge \ulin\Cont_d(S_1)+\ulin\Cont_d(S_2)-\lim_{\eps\to 0^+} \lin\Cont_d(S_1\cap (S_1\cap S_2)^\eps).\label{Cont-}\EDE

The following is the definition of Minkowski content measure in \cite[Definition 6.3]{Green-kappa-rho}.

\begin{Definition}
  Let $S\subset\R^n$ be the intersection of an open set and a closed set. A $d$-dimensional Minkowski content measure on $S$ is a Borel measure $\mu$ on $S$, which satisfies that, for any compact set $K\subset S$, (i) $\mu(K)<\infty$; and (ii) whenever $\mu(\pa_S K)=0$,
 $\Cont_d(K) $ exists and equals $\mu(K)$,
where $\pa_S K$ is the relative boundary of $K$ in $S$.
\end{Definition}

The following proposition is a summary of some results in  \cite[Section 6.1]{Green-kappa-rho}.

\begin{Proposition}
   Suppose $S$ has a $d$-dimensional  Minkowski content measure $\mu$. Then
  \begin{itemize}
    \item[(i)] $\mu$ is unique and $\sigma$-finite.
    \item [(ii)] For any compact set $K\subset S$, $\lin \Cont_d(K)\le \mu(K)$.
    \item [(iii)] For any closed set $F\subset\R$ and open set $G\subset\R$, if $\mu(\pa_S(S\cap F))=0$, then $\mu|_{F\cap G}$ is the $d$-dimensional Minkowski content measure on $S\cap F\cap G$.
    \item [(iv)] If $G$ is an open set containing $S$ such that $G\sem S$ is open, and $\phi:G\to \R^n$ is a conformal map, then $\nu:=\phi(|\phi'|^d \cdot \mu)$ is the $d$-dimensional Minkowski content measure on $\phi(S)$.
  \end{itemize}
  \label{Prop-Minkowski}
\end{Proposition}


We use ${\cal M}^d_S$ to denote the $d$-dimensional Minkowski content measure on $S$ when it exists and omit the superscript $d$ when it is fixed in the context.

\begin{Lemma}
 Let $S\subset\R^n$ be the intersection of an open set and a closed set.  Let $(S_\lambda)_{\lambda\in\Lambda}$ be a   family of relatively open subset of $S$ such that $S=\bigcup_{\lambda\in\Lambda} S_\lambda$. Suppose ${\cal M}^d_{S_\lambda}$ exists for each $\lambda$.   Then ${\cal M}^d_S$ also exists. \label{Mink-exist-S12}
\end{Lemma}
\begin{proof}
Since $S$  could be written as a countable union of compact sets, we may assume that $\Lambda$ is countable. Let $\mu_j ={\cal M}^d_{S_j}$. By Proposition \ref{Prop-Minkowski} (i,iii), $\mu_j|_{S_j\cap S_k}=\mu_k|_{S_j\cap S_k}$ for any $j\ne k\in\Lambda $. So there is a $\sigma$-finite measure $\mu$ on $S=\bigcup_{\lambda\in\Lambda} S_\lambda$ such that $\mu|_{S_\lambda}=\mu_\lambda$ for all $\lambda\in\Lambda$. We will show that $\mu={\cal M}^d_S$. It suffices to show that, for any compact set $K\subset S$, (i) $\mu(K)<\infty$ and (ii) when $\mu(\pa_S K)=0$, $\Cont_d(K)=\mu(K)$.

Let $K$  be a compact subset of $S$.
Since  $(S_j)$ is a relatively open covering of $K$ in $S$, there is $r_0>0$ such that any subset of $K$ with diameter less than $r_0$ is contained in at least one $S_j$. Since $\mu$ is $\sigma$-finite, after some translation we may assume that, for any $1\le j\le n$ and $q\in\Q$, the hyperplane $\{\ulin x\in\R^n: x_j=q\}$ has no $\mu$-mass. Pick  $r\in\Q\cap (0,r_0/\sqrt n)$. Let ${\cal Q}$ denote the space of $n$-dimensional cubes $\prod_{j=1}^n [ m_j r,(m_j+1)r]$, where $m_1,\dots,m_n\in\Z$. Then any $Q\in\cal Q$ has diameter $\sqrt n r<r_0$, and so $Q\cap K$ is contained in at least one $S_j$.

Let ${\cal Q}^K$ denote the set of $Q\in\cal Q$ such that $Q\cap K\ne \emptyset$. Then ${\cal Q}^K$ is finite since $K$ is bounded.
For each $j\in \Lambda$, let ${\cal Q}^K_j$ denote the set of $Q\in{\cal Q}^K$ such that $ Q\cap K\subset S_j$. Then ${\cal Q}^K=\bigcup_{j\in\Lambda} {\cal Q}^K_j$. Since ${\cal Q}^K$ is finite, we may find a finite subset $\Lambda'\subset \Lambda$ and a partition $\til{\cal Q}^K_j$, $j\in\Lambda'$, of ${\cal Q}^K$, such that $\til {\cal Q}^K_j\subset {\cal Q}^K_j$ for each $j\in\Lambda'$.  We may assume $\Lambda'=\N_m:=\{j\in\N:j\le m\}$, where $m \in\N$. Let $K_j=K\cap \bigcup_{Q\in \til {\cal Q}^K_j} Q$, $1\le j\le m$. Then $K=\bigcup_{j=1}^m K_j$, and each $K_j$ is a compact subset of $S_j$. For any $j\ne k$, we have $K_j\cap K_k\subset \bigcup_{s=1}^n \bigcup_{q\in\Q} \{\ulin x\in\R^n:x_s=q\}$, and so  $\mu(K_j\cap K_k)=0$.   Since $\mu|_{S_j}=\mu_j$ and $\mu_j={\cal M}^d_{S_j}$, we get $\mu(K_j)=\mu_j(K_j)<\infty$. Thus, $\mu(K)\le\sum_{j=1}^n \mu(K_j)<\infty$.

Now assume $\mu(\pa_S K)=0$. For each $j\in\N_m$, from $\pa_{S_j} K_j\subset \pa_S K\cup (K_j\cap \bigcup_{k\ne j} K_k)$ we get $\mu_j(\pa_{S_j} K_j)=0$. Since  $\mu_j={\cal M}^d_{S_j}$ and $\mu|_{S_j}=\mu_j$, we get $\Cont_d(K_j)=\mu_j(K_j)=\mu(K_j)$.  If $m=1$, then   $\Cont_d(K)=\mu(K)$ since $K=K_1$. Suppose $m\ge 2$. For  $1\le k\le m$, let $K^\cup_k=\bigcup_{j=1}^k K_j$. Since $\mu(K_j\cap K_k)=0$ for $k\ne j$, we have $\mu(K^\cup_k)=\sum_{j=1}^k \mu(K_j)$.
  We will prove by induction that $\Cont_d(K^\cup_k)=\mu(K^\cup_k)$ for each $k\in\N_m$. Then  we get the desired equality $\Cont_d(K)=\mu(K)$ since $K=K^\cup_m$.

  We have known that $\Cont_d(K^\cup_k)=\mu(K^\cup_k)$ for $k=1$ since  $K^\cup_1=K_1$. Suppose we have proved the equality for some $k<m$. Since $K^\cup_{k+1}=K^\cup_k\cup K_{k+1}$,  from (\ref{Cont+}) we get
$$\lin\Cont_d(K^\cup_{k+1})\le \Cont_d(K^\cup_k)+\Cont_d(K_{k+1})=\mu(K^\cup _k)+\mu(K_{k+1})=\mu(K^\cup_{k+1}).$$
  From (\ref{Cont-}) and Proposition \ref{Prop-Minkowski} (ii), we get
  $$\ulin\Cont_d(K^\cup_{k+1})\ge \Cont_d(K^\cup_k)+\Cont_d(K_{k+1})-\lim_{\eps\to 0^+} \mu_{k+1}(K_{k+1}\cap(K_{k+1}\cap K^\cup_k)^\eps)$$
  $$=\mu(K^\cup_k)+\mu(K_{k+1})-\mu(K_{k+1}\cap K^\cup _k)=\mu(K^\cup_{k+1}).$$
  Combining the above two displayed formulas  we get $\Cont_d(K^\cup_{k+1})=\mu(K^\cup_{k+1})$. The induction step is finished, and the proof is complete.
\end{proof}

\subsection{Continuous processes}\label{section-lifetime}
In this section we recall the setup used in \cite[Section 2]{decomposition} with some slight change of symbols.
Let
$$\Sigma=\bigcup_{0<\ha T\le \infty} C([0,\ha T),\R).$$
For each $f\in\Sigma$, the lifetime $\ha T_f$ of $f$ is the extended number in $(0,\infty]$ such that $[0,\ha T_f)$ is the domain of $f$.

For $f\in\Sigma$ and $t\in(0,\infty]$,  the killing map $\K_t:\Sigma\to\Sigma$ is defined by $\K_t(f)=f|_{[0,\ha T_f\wedge t)}$.
For a function $\tau:\Sigma\to (0,\infty]$, we define   $\K_\tau:\Sigma\to\Sigma$ by $\K_\tau(f)=\K_{\tau(f)}(f)$.
Define $$\Sigma_{\ha T}=\{f\in\Sigma: f(\ha T_f^-):=\lim_{t\to \ha T_f^-} f(t)\in\R\},\quad \Sigma_{\ha T}^{<\infty}=\{f\in\Sigma_{\ha T}:\ha T_f<\infty\} ,$$ $$\Sigma_{\oplus}=\{(f,g)\in\Sigma_{\ha T}\times \Sigma:   g(0)=f(\ha T_f^-)\},\quad \Sigma^{<\infty}_{\oplus}=\{(f,g)\in\Sigma_\oplus: \ha T_f<\infty\}.$$
For $(f,g)\in\Sigma^{<\infty}_{\oplus}$, we define the continuation $f\oplus g\in\Sigma$ by
$$f\oplus g(t)=\left\{\begin{array}{ll} f(t), & 0\le t<\ha T_f;\\
g(t-\ha T_f), & \ha T_f\le t<\ha T_f+\ha T_g=:\ha T_{f\oplus g}.
\end{array}
\right.$$
We define $f\ha\oplus g=(f\oplus g,\ha T_f)$ to record the time that $f$ and $g$ are joined together. We may  recover $f$ and $g$ from $f\ha\oplus g$. Let
$\Sigma_{(0)}=\{g\in\Sigma:g(0)=0\}$. For $f\in\Sigma_{\ha T}^{<\infty}$   and $g\in \Sigma_{(0)}$, we define $f\oplus_{+} g=f\oplus(f(\ha T_f^-)+g)$ and $f\ha\oplus_{+} g=f\ha\oplus (f(\ha T_f^-)+g)$.

For a function $\tau:\Sigma\to [0,\infty]$, let $\Sigma_{\oplus,\tau}=\{(f,g)\in \Sigma^2: 0\le \tau(f)<\ha T_f,f(\tau(f))=g(0)\}$. For $(f,g)\in\Sigma_{\oplus,\tau}$, we define
$$f\oplus_\tau g(t)=\left\{\begin{array}{ll} f(t), & 0\le t<\tau(f);\\
g(t-\tau(f)), & \tau(f)\le t<\tau(f)+\ha T_g=:\ha T_{f\oplus_\tau g}.
\end{array}
\right.$$
Note that if $\tau(f)>0$, then  $f\oplus_\tau g=\K_\tau(f)\oplus g$.

Let $\Sigma_t=\{f\in\Sigma:\ha T_f>t\}$, $0\le t<\infty$. Then $(\Sigma_t)$ is a decreasing family of subspaces of $\Sigma$ with $\Sigma_0=\Sigma$ and $\bigcap_{t=0}^\infty \Sigma_t=C([0,\infty),\R)$. Let $\pi_t$ be the map $\Sigma_t\ni f\mapsto f(t)$.
Define a filtration $\F^*=(\F^*_t)_{t\ge 0}$  on $\Sigma$ such that for $0\le t< \infty$, $\F^*_t$ is the $\sigma$-algebra generated by $(\Sigma_s,\pi_s)$ for $0\le s\le t$.  Let $\F$ be the right-continuation of $\F^*$, and $\F_\infty=\sigma(\bigcup_{0\le t<\infty}\F_t)$. Every probability measure on $(\Sigma,\F_\infty)$ is the law of a continuous stochastic process with random lifetime.  For any measure $\mu$ on $(\Sigma,\F_\infty)$, we use $\F^\mu$   to denote the $\mu$-completion of $\F$.

Let $\mu$ and $\nu$ be two  measures on $(\Sigma,\F_\infty)$, which are $\sigma$-finite on $\F_0$ (and so are $\sigma$-finite on each $\F_t$). We say that $\nu$ is locally absolutely continuous w.r.t.\ $\mu$, and write $\nu\tl \mu$, if for any $0\le t<\infty$, $\nu|_{ \F_t\cap\Sigma_t}$ is absolutely continuous w.r.t.\ $\mu|_{\F_t\cap\Sigma_t}$.
The process $M_t:=\frac{d\nu|_{\F_t\cap\Sigma_t}}{d\mu|_{\F_t\cap\Sigma_t}}$, $0\le t<\infty$, is called the  Radon-Nikodym (RN for short) process of $\nu$ w.r.t.\ $\mu$. A simple and useful fact is that, given $\mu$ and $(M_t)$, there is at most one $\nu$ such that $(M_t)$ is the RN process of $\nu$ w.r.t.\ $\mu$.

\begin{Example}
  Let $f$ be a random continuous process with random lifetime. Let $\tau$ be a positive $\F$-stopping time and $g=\K_\tau(f)$. Then the law of $g$ is locally absolutely continuous w.r.t.\ the law of $f$, and the RN process is $({\ind}_{[0,\tau)}(t))$. \label{example-K}
\end{Example}

The following  proposition is \cite[Proposition 2.1]{decomposition}.

\begin{Proposition}
Let $\mu$ be a measure on $(\Sigma,\F_\infty)$, which is $\sigma$-finite on $\F_0$.   Let $(\Upsilon,{\cal G})$ be a  measurable space. Let $\nu:\Up\times \F_\infty\to[0,\infty]$ be such that for every $\up\in\Up$, $\nu(\up,\cdot)$ is a finite measure on $\F_\infty$ that is locally absolutely continuous w.r.t.\ $\mu$. Suppose  there are measurable maps $M_t:(\Up,{\cal G})\times (\Sigma,\F_t)\to[0,\infty)$ for each $t\ge 0$ such that, for every $\up\in\Up$, the RN process of $\nu(\up,\cdot)$ w.r.t.\ $\mu$ is $(M_t(\up,\cdot))_{t\ge 0}$.  Then $\nu$ is a kernel from $(\Upsilon,{\cal G})$ to $(\Sigma,\F_\infty)$. Moreover, if $\xi$ is a $\sigma$-finite measure on $(\Upsilon,{\cal G})$ such that $\mu$-a.s., $\int_\Up M_t(\up,\cdot)d\xi(\up)<\infty$ for all $t\ge 0$, then $\int_\Up \nu(\up,\cdot) \xi(d\up)\tl\mu$, and the corresponding RN process
is $(\int_\Up M_t(\up,\cdot)d\xi(\up))_{t\ge 0}$. \label{Prop-fubini}
\end{Proposition}

Let $f\in\Sigma$. For a continuous and strictly increasing function $\theta$ on $[0,\ha T_f)$ with $\theta(0)=0$, $g:=f\circ \theta^{-1}\in \Sigma$ is called the time-change of $f$ via $\theta$, and we write $f\sim g$. If $f_1\sim f_2$ and $g_1\sim g_2$, we write $(f_1,g_1)\sim (f_2,g_2)$.
We will be interested in the quotient spaces $\til\Sigma:=\Sigma/\sim$ and $\til\Sigma^2=\Sigma^2/\sim$, and use $\pi$ and $\pi^2$ to denote the corresponding natural projections.
An element of $\til\Sigma$, often denoted by $[f]$, where $f\in\Sigma$, is called an MTC (modulo time-changes) function or curve. Sometimes we treat $[f]$ as a the set of all $g\in\Sigma$ such that $f\sim g$.

A set $S\subset\Sigma$ (resp. $S\subset \Sigma^2$) is called time-change invariant (TCI for short) if  $S=\pi^{-1}(\pi(S))$ (resp.\ $S=(\pi^2)^{-1}(\pi^2(S))$). It is clear that $\Sigma_{\ha T}$, $\Sigma_{(0)}$  and  $ \Sigma_{\oplus} $ are TCI, but $\Sigma_{\ha T}^{<\infty}$  and  $ \Sigma_{\oplus}^{<\infty} $ are not.  Let $\til\Sigma_{\ha T},\til\Sigma_{(0)},\til\Sigma_{\oplus}$ be the images of $\Sigma_{\ha T}$, $\Sigma_{(0)}$  and  $ \Sigma_{\oplus} $  under $\pi$ or $\pi^2$, respectively.
The operation $\oplus$ induces the operation $\til \Sigma_{\oplus}\ni ([f],[g])\mapsto [f]\oplus [g]\in \til\Sigma$ defined by $[f]\oplus [g]=[f_0\oplus g_0]$, where  $f_0\in[f]$ and $g_0\in [g]$ satisfy $(f_0,g_0)\in \Sigma_{\oplus}^{<\infty}$.   Such definition does not depend on the choice of $f_0$ and $g_0$.

Let $f\in\Sigma_{\ha T}$ and $\phi$ be continuously and strictly decreasing on $(0,\ha T_f)$ such that $\phi(\ha T_f^-)=0$. Let $f_\phi\in\Sigma_{\ha T}$ be defined by $f_\phi(0)=f(\ha T_f^-)$ and $f_\phi(t)=f\circ \phi^{-1}(t)$ for $0<t<\phi(0^+)$. Then $[f_\phi]\in \til\Sigma_{\ha T}$ is called the time-reversal of $[f]$, which does not depend on either $f$ or $\phi$, and the map  ${\cal R}:[f]\mapsto [f_\phi]$ is called the time-reversal map.

A function $\tau$ defined on $\Sigma$ is called a TCI time on $\Sigma$ if $\tau(f)\in [0,\ha T_f]$ for any $f\in\Sigma$, and if $f_2\in\Sigma$ is a time-change of $f_1\in\Sigma$ via some $\theta$, then $\tau(f_2)=\theta(\tau(f_1))$, where $\theta(\ha T_{f_1})$ is understood as $\ha T_{f_2}$. For example, $0$ and $\ha T$ are TCI times.
For $[f]\in \til\Sigma$, a function $\tau$ defined on $[f]$ as a subset of $\Sigma$ is called a TCI time on $[f]$ if the above properties hold for any $f_1,f_2\in [f]$. Let $[f]\in\til\Sigma$ and $\tau$ be a TCI time on $\Sigma$ or $[f]$, which is less than $\ha T$. We define $[f](\tau)=f(\tau(f))$, $[f]([0,\tau])=f([0,\tau(f)] )$, and $\K_\tau([f])=[\K_\tau(f)]$ when $\tau(f)>0$.
We call $[f]([0,\ha T))=f([0,\ha T_f))$ the image of $[f]$ and with a slight abuse of notation denote it by $[f]$.
If   $[g]\in\til\Sigma$ satisfies $[g](0)=[f](\tau)$ for some TCI time $\tau$ on $\Sigma$ or $[f]$, we define $[f]\oplus_\tau [g]=[f\oplus_\tau g]$. These definitions do not depend on the choices of the representatives of $[f]$ and $[g]$.

\begin{Remark}
  We emphasize that a TCI time is not defined on $\til\Sigma$ or any subset of $\til\Sigma$, but on $\Sigma$ or some $[f]$ as a subset of $\Sigma$.
\end{Remark}

\begin{Example}
   Given a set $S$, let $\sigma_S(f)=\inf(\{t\in[0,\ha T_f):f(t)\in S\}\cup\{\ha T_f\})$. Then $\sigma_S$ is a TCI $\F$-stopping time. 
\end{Example}

 For a TCI $\F$-stopping time $\tau$, let $\til\F_\tau$ be the family of the $\pi$-images of TCI sets in $\F_\tau$. We understand $\til\Sigma$ as a measurable space equipped with the $\sigma$-algebra $\til\F_{\ha T}$, and a probability measure on $(\til\Sigma,\til\F_{\ha T})$ as the law of some random process modulo time-changes.

The binary operations could be naturally extended to measures.
Let $\mu$ be a $\sigma$-finite measure  on $\Sigma$. If $\tau$ is a $\sigma$-finite $\mu$-kernel  from $\Sigma$ to $(0,\infty]$, we use $\K_\tau(\mu)$ to denote the pushforward  of $\mu\rotimes \tau$ under the map $(f,t)\mapsto \K_t(f)$.  If  $\nu$ is a $\sigma$-finite $\mu$-kernel from $\Sigma$ to $\Sigma$ such that $\mu\rotimes \nu$ is supported by $\Sigma_{\oplus}^{<\infty}$ (resp.\ $\Sigma_{\ha T}^{<\infty}\times \Sigma_{(0)}$), then we use $\mu\oplus\nu$ and $\mu\ha\oplus \nu$ (resp.\ $\mu\oplus_{+}\nu$ and $\mu\ha\oplus_{+}\nu$) to denote  the pushforward  of $\mu\rotimes\nu$ under the maps $(f,g)\mapsto f\oplus g$ and $(f,g)\mapsto f\ha\oplus g$ (resp.\ $(f,g)\mapsto f\oplus_{+} g$ and $(f,g)\mapsto f\ha\oplus_{+} g$). Let $\tau:\Sigma\to [0,\infty]$ be  $\F^\mu_\infty$-measurable.
If $\nu$ is a $\sigma$-finite $\mu$-kernel from $\Sigma$ to   $\Sigma$  such that $\mu\rotimes\nu$ is supported by $\Sigma_{\oplus,\tau}$, then we define $\mu\oplus_\tau \nu$ to be the  pushforward measure of $\mu\rotimes \nu$ under the map $(f,g)\mapsto f\oplus_\tau g$.

Let $\til\mu$ be a $\sigma$-finite measure on $\til\Sigma$. If $\tau$ is an measurable TCI   time, then $\K_\tau(\til\mu)$ is defined as the pushforward of $\til\mu $ under the map  $[f]\mapsto\K_\tau([f])$. If $\til\nu$ is a $\sigma$-finite $\til\mu$-kernel from $\til\Sigma$ to $\til\Sigma$ such that $\til\mu\rotimes \til\nu$ is supported by $\til \Sigma_{\oplus}$, then we use $\til\mu\oplus\til \nu$  to denote the pushforward of $\til\mu\rotimes \til\nu$ under the map $([f],[g])\mapsto [f]\oplus [g]$.
If  $\til\mu\rotimes \til\nu$ is supported by $\til \Sigma_{\oplus,\tau}$ for some measurable  TCI   time $\tau$, then we define $\til\mu\oplus_\tau\til\nu$ to be the pushforward of $\til\mu\rotimes \til\nu$ under the map $([f],[g])\mapsto [f]\oplus_\tau [g]$.

We also define $f\oplus[g]=[f]\oplus [g]$, $f\oplus_\tau [g]=[f]\oplus_\tau[g]$, $\mu\oplus \til \nu=\pi(\mu)\oplus \til \nu$ and $\mu\oplus_\tau \til \nu=\pi(\mu)\oplus_\tau \til \nu$ as long as the righthand sides are well defined.

The following proposition is   the remark right after \cite[Proposition 2.2]{decomposition}.

\begin{Proposition}
Let $\mu$ be a probability measure on $\Sigma$, and $(\theta_t)_{t\ge 0}$ be an $\F^\mu$-adapted right-continuous increasing process with $\theta_0=\theta_{0^-}=0$ and  $\EE_\mu[\theta_\infty]<\infty$. Then  $\K_{d\theta}(\mu)\tl \mu$ with the RN process being $(\EE_\mu[\theta_\infty| \F_t^\mu]-\theta_t)_{t\ge 0}$, where $d\theta$ is the measure determined by $\theta$.
\label{trancation-remark}
\end{Proposition}

If instead of $C([0,\ha T),\R)$ we use $C([0,\ha T),E)$  for some metric space $E$ to define $\Sigma$, the spaces $\Sigma$ and $\til\Sigma$ will be respectively denoted by $\Sigma^E$  and $\til\Sigma^E$. Most results in this section  still hold for $\Sigma^E$ and  $\til\Sigma^E$ except that we may not have the operations $\oplus_{+}$ and $\ha\oplus_{+}$. The  $\til\Sigma^E$ is a metric space with the distance  defined by
\BGE \dist_{\til\Sigma^E}([f],[g])=\inf\{\sup\{\dist_E(f'(t),g'(t)): {0\le t<\ha T_{f'}}\}:  f'\in [f],g'\in [g], \ha T_{f'}=\ha T_{g'}\}.\label{dist-E}\EDE

\begin{Lemma} 
  Let $g\in\Sigma^E$, $t_0\in [0,\ha T_g)$, and $\tau$ be an $\F$-stopping time. Define $\tau_{g;t_0}$ on $\Sigma^E $ by
  \BGE \tau_{g;t_0}(f)=\left\{\begin{array}{ll} (\tau(g\oplus_{t_0} f)-t_0)\vee 0, &\mbox{if }f(0)=g(t_0)
  ;\\
  \ha T_f, & \mbox{otherwise}.
  \end{array}
  \right.\label{tau-g-formula'}\EDE
  Then $\tau_{g;t_0}$ is also an $\F$-stopping time. \label{tau-g'}
\end{Lemma}
\begin{proof} Let  $x_0=g(t_0)\in E$.
  Since $\F$ is the right-continuation of $\F^*$, it suffices to show that for any $t\in(0,\infty)$, $\{f\in\Sigma:\tau_{g;t_0}(f)<t\}\in \F^*_{t}$, which is equivalent to $$\{f\in\Sigma:f(0)=x_0,\tau(g\oplus_{t_0} f)< t_0+t
 \} \cup \{f\in\Sigma:f(0)\ne x_0,\ha T_f<t\} \in \F^*_{t}.$$
 The fact that $ \{f\in\Sigma:f(0)\ne x_0,\ha T_f<t\} \in \F^*_{t}$ is obvious since $\ha T$ is an $\F$-stopping time.
  Since $\tau$ is an $\F$-stopping time, we have
  $\{\tau<t_0+t\}\in \F^*_{t_0+t}$.  Thus, it suffices to show that, for any $A\in\F^*_{t_0+t}$, $\{f\in\Sigma:f(0)=x_0,g\oplus_{t_0} f\in A\}\in \F^*_{t}$. By the definition of $\F^*_{t_0+t}$, it suffices to prove that for any $s\in[0,t_0+t]$ and  $U\in{\cal B}(E)$,  $\{f\in\Sigma: f(0)=x_0,(g\oplus_{t_0} f)(s)\in U\}\in \F^*_{t}$. If $0\le s< t_0$, the set is either $\{f\in\Sigma:f(0)=x_0\}$ or empty depending on whether $g(s)\in U$, and so belongs to $\F^*_0\subset \F^*_t$. If $s\ge t_0$, then the set becomes $\{f\in\Sigma:f(0)=x_0,f(s-t_0)\in U\}$, which belongs to $\F^*_{s-t_0}\subset \F^*_t$ since $s\le t_0+t$.
\end{proof}

\begin{Lemma}
Let $[g]\in\til\Sigma^E $ and $\sigma<\ha T$ be a TCI time on $[g]$. Let $[f]\in\til\Sigma^E$ be such that $[f](0)=[g](\sigma)$. Let $\tau$ be a TCI time on $[g]\oplus_\sigma [f]$. 
Define  $\tau_{[g];\sigma} $ on $[f]$ by
\BGE \tau_{[g];\sigma}(f')= \tau_{g';\sigma(g')}(f'),\quad f'\in[f],\quad g'\in[g]. \label{tau-g-formula''}\EDE
Then $\tau_{[g];\sigma}$  is a well-defined  TCI time on $[f]$. \label{tau-[g]'}
\end{Lemma}
\begin{proof}
If there are no $f'\in[f]$ and $g'\in[g]$ such that  $\tau(g'\oplus_{\sigma(g')} f')\ge \sigma(g')$, then $\tau(g'\oplus_{\sigma(g')} f')< \sigma(g')$, and so $\tau_{g';\sigma(g')}(f')=0$. The statement is trivial.

Suppose there exist $f_0\in[f]$ and $g_0\in[g]$ such that  $\tau(g_0\oplus_{\sigma(g_0)} f_0)\ge \sigma(g_0)$.
 Since $\tau$ is a TCI time on $[g]\oplus_\sigma[f]$, we have $\tau(g_0\oplus_{\sigma(g_0)} f_0)\in [\sigma(g_0),\ha T_{g_0\oplus_{\sigma(g_0)} f_0}]=[\sigma(g_0),\sigma(g_0)+\ha T_{f_0}]$. Thus, $\tau(g_0\oplus_{\sigma(g_0)} f_0)-\sigma(g_0)\in [0,\ha T_{f_0}]$.
Let $g'$ be a time-change of $g_0$ via $\psi$, and $f'$ be a time-change of $f_0$ via $\theta$. 
Define
$$\phi(t)=\left\{\begin{array}{ll} \psi(t), & \mbox{if } 0\le t\le \sigma(g_0);\\
\sigma(g')+\theta(t-\sigma(g_0)),& \mbox{if }\sigma(g_0)\le t< \sigma(g_0)+\ha T_{f_0}.
\end{array}
\right.$$
 Then $g'\oplus_{\sigma(g')} f' $ is a time-change of $g_0\oplus_{\sigma(g_0)} f_0 $ via $\phi$.  Since $\tau$ is a TCI time on $[g]\oplus_\sigma [f]$, and $\tau(g_0\oplus_{\sigma(g_0)} f_0)\ge \sigma(g_0)$, we have
$$\tau(g'\oplus_{\sigma(g')} f')-\sigma(g')=\phi(\tau(g_0\oplus_{\sigma(g_0)} f_0))-\sigma(g')=\theta(\tau(g_0\oplus_{\sigma(g_0)} f_0)-\sigma(g_0))\in [0,\ha T_{f'}],$$which shows that the definition of $\tau_{[g];\sigma}$ does not depend on the choice of $g'\in [g]$, and $\tau_{[g];\sigma}$ is a TCI time on $[f]$.
\end{proof}


We will use the following equality. Suppose $[g],[f],[h]\in\til\Sigma^E$.  Let $\sigma$ be a TCI time on $[g]$ such that $\sigma(g)<\ha T_g$ and $[g](\sigma)=[f](0)$. Let $\tau$ be a TCI time on $[g]\oplus_\sigma[f]$ such that for any $g'\in[g]$, $f'\in[f]$, $\sigma(g')\le \tau(g'\oplus_{\sigma(g')} f')<\ha T_{g'\oplus_{\sigma(g')} f'}$, and $([g]\oplus_\sigma[f])(\tau)=[h](0)$. Then
\BGE ([g]\oplus_\sigma[f])\oplus_\tau[h]=[g]\oplus_\sigma([f]\oplus_{\tau_{[g];\sigma}}[h]).\label{fgh}\EDE

Another remark to make here is that if $\tau$ is the hitting time of a set $S$ and $\sigma(g)<\tau(g)$, then  $\tau_{[g];\sigma}$ agrees with $\tau$ on the set of $[f]$ such that $[f(0)]=[g](\sigma)$.

An element  $f\in\Sigma^E$ is called a rooted loop if $\lim_{t\uparrow \ha T_f} f(t)=f(0)$, and $f(0)$ is called its root. Two rooted loops $f$ and $g$ are called translation equivalent if $\ha T_f=\ha T_g<\infty$, and after being extended to functions on $\R$ with period $\ha T_f$, $f$ and $g$ satisfy that $f=g(a+\cdot)$ for some $a\in \R$. If $f\in\Sigma^E$ is a rooted loop, then $[f]\in\til\Sigma^E$ is called a rooted MTC loop. Two rooted MTC loops $[f]$ and $[g]$ are called translation equivalent if there are $f_0\in [f]$ and $g_0\in[g]$ such that $f_0$ and $g_0$ are translation equivalent. We call the translation equivalent class, denoted by $[[f]]$,   an unrooted MTC loop, and use $\til\pi$ to denote the projection map $[f]\mapsto [[f]]$.

\subsection{Loewner curves} \label{section-Loewner}
Let $\HH=\{z\in\C:\Imm z>0\}$ be the open upper half plane. A set $K\subset\HH$ is called an $\HH$-hull if $K$ is bounded and $\HH\sem K$ is a simply connected domain. For each $\HH$-hull $K$, there is a unique conformal map $g_K$ from $\HH\sem K$ onto $\HH$ such that $g_K(z)-z=O(1/z)$ as $\HH\sem K\ni z\to \infty$.
The number $\hcap(K):=\lim_{z\to\infty} z(g_K(z)-z)$ is called the $\HH$-capacity of $K$, which satisfies $\hcap(\emptyset)=0$ and $\hcap(K)>0$ if $K\ne \emptyset$.

We use the symbol $\rad_w(S):=\sup\{|z-w|:z\in K\cup\{w\}\}$ for $w\in\C$ and $S\subset\C$.
By \cite[Formula (3.12)]{Law1}, for any $\HH$-hull $K$ and any $x\in\R$,
\BGE |g_K(z)-z|\le 3 \rad_x(K),\quad \forall z\in \HH\sem K.\label{g-z}\EDE

For a nonempty $\HH$-hull $K$, let $a_K=\min(\lin K\cap \R)$ and $b_K=\max(\lin K\cap \R)$. By Schwarz reflection principle, $g_K$ extends to a conformal map from $\ha\C\sem K^{\doub}$ onto $\ha\C\sem [C_K,D_K]$, where $K^{\doub}:=K\cup[a_K,b_K]\cup \{\lin z:z\in K\}$ and $[C_K,D_K]$ is a compact real interval, such that $g_K(\infty)=\infty$ and $g_K'(\infty)=1$. We have $C_K=\lim_{x\uparrow a_K} g_K(x)$ and $D_K=\lim_{x\downarrow b_K} g_K(x)$. We define prime ends $\ha a_K=g_K^{-1}(C_K)$ and $\ha b_K=g_K^{-1}(D_K)$ of $\HH\sem K$, which respectively determine $a_K$ and $b_K$. Since $g_K'(\infty)=1$, by Koebe's $1/4$ theorem (cf.\ \cite{Ahl}), for any $x\in [a_K,b_K]$,
\BGE ( D_K-C _K)/4\le \rad_x(K^{\doub})=\rad_x(K)\le D_K-C_K. \label{Koebe0}\EDE
More generally, for any $a\le x\le b$,  $g_K$ maps $\ha\C\sem (K^{\doub}\cup [a,b])$ conformally onto $\ha\C\sem [C_K^a,D_K^a]$, where $C_K^a:=\lim_{x\uparrow a\wedge a_K} g_K(x)$ and $D_K^b:=\lim_{x\downarrow b\vee b_K} g_K(b)$. So we have
\BGE ( D_K^b-C_K^a )/4\le  \rad_x(K\cup [a,b])\le D_K^b-C_K^a. \label{Koebe0'}\EDE

For $W\in\Sigma$, the chordal Loewner equation driven by $W$ is the following differential equation in the complex plane:
$$ \pa g_t(z)=\frac{2}{g_t(z)-W_t},\quad 0\le t<\ha T_W;\quad g_0(z)=z. $$ 
For each $z\in \C$, let $\tau_z$ denote the biggest extended number in $[0,\ha T_W]$ such that the solution $t\mapsto g_t(z)$ exists on $[0,\tau_z)$. Here $\tau_{W_0}=0$ and $\tau_z>0$ for $z\ne W_0$. For $0\le t<\ha T_W$, let $K_t=\{z\in\HH:\tau_z\le t\}$ and $H_t=\HH\sem K_t$.  It turns out that each $K_t$ is an $\HH$-hull with $\hcap(K_t)=2t$, and $g_t=g_{K_t}$.
We call $g_t$ and $K_t$ the chordal Loewner maps and hulls, respectively, driven by $W$. 
The Loewner equations in this paper are all chordal Loewner equations, and we will omit the word ``chordal''.

For $0<t<\ha T_W$, we have $K_t\ne \emptyset$ and then define $a_t=a_{K_t}$, $b_t=b_{K_t}$, $C_t=C_{K_t}$, $D_t=D_{K_t}$, $\ha a_t=\ha a_{K_t}$ and $\ha b_t=\ha b_{K_t}$. We also define $a_0,b_0,C_0,D_0,\ha a_0,\ha b_0$ to be all $W_0$. 
Every prime end of $\HH$ on the interval $(b_t,a_t)_{\ha\pa\HH}=(b_t,\infty)\cup\{\infty\}\cup(-\infty,a_t)$ is identified with a prime end of $H_t$ on the interval $(\ha b_t,\ha a_t)_{\ha\pa H_t}$, and vice versa. We use $I_t$ to denote the common prime end interval, which is the biggest prime end interval shared by $\HH$ and $H_t$. 


Suppose for every $t\in[0,\ha T_W)$, $g_t^{-1}$ extends continuously from $\HH$ to $\lin\HH$, and $\eta(t):=g_t^{-1}(W_t)$, $0\le t<\ha T_W$, is a continuous curve in $\lin\HH$. Then we say that $\eta$ is the capacity parametrized (CP for short) Loewner curve driven by $W$. In this case, $H_t$ and $K_t$ are determined by $\eta$ such that $H_t $ is the unbounded connected component of $\HH\sem \eta[0,t]$, and $K_t=\HH\sem H_t$. 

When we want to emphasize the dependence on $W$, we will use the symbols $g^W_t$, $\tau^W_z$, $K^W_t$, $H^W_t$, $a^W_t$, $b^W_t$, $C^W_t$, $D^W_t$, $\ha a^W_t$, $\ha b^W_t$, $I^W_t$, and $\eta^W(t)$. If $W$ generates a CP Loewner curve $\eta$, then $W$ is also determined by $\eta$, and we will use the symbols $W^\eta_t$, $g^\eta_t$, $\tau^\eta_z$, $K^\eta_t$, $H^\eta_t$, $a^\eta_t$, $b^\eta_t$, $C^\eta_t$, $D^\eta_t$, $\ha a^\eta_t$, $\ha b^\eta_t$ and $I^\eta_t$ to emphasize the dependence on $\eta$.

Not every $W\in\Sigma$ generates a CP Loewner curve. Let $\Sigma^\Lo$ denote the set of  $W\in\Sigma$, which generates a CP Loewner curve $\eta^W$, and let $\Lo$ denote the map from $W\in\Sigma^\Lo$ to $\eta^W\in\Sigma^{\lin\HH}$, and let $\ha\Lo$ denote the map from $(W,t)$, where $W\in\Sigma^\Lo$ and $0\le t<\ha T_W$, to $(\Lo(W),\Lo(W)(t))$.

If $\eta'$ is a time-change of a CP Loewner curve $\eta$, then we call $\eta'$ an arbitrarily parametrized (AP for short) Loewner curve in $\HH$ aimed at $\infty$. If $f$ maps $\HH$ conformally onto $D$, then $f\circ \eta'\in\Sigma^{\ha D}$ is called an AP Loewner curve in $D$ aimed at the prime end $f(\infty)$,  and the MTC curve $[f\circ \eta]\in\til \Sigma^{\ha D}$ is called an MTC Loewner curve in $D$ aimed at $f(\infty)$.

We will use the following symbols.
For $D,S,T\subset\ha\C$, we use  $D(S;T)$ to denote the connected component of $D\sem S$ whose closure contains $T$, whenever such component exists and is unique.
If $D$ is a simply connected domain, $w$ is a prime end of $D$, and $D\sem S$ has a connected component  which is a neighborhood of $w$ in $D$, then we use $D(S;w)$ to denote this component. For example, if $\eta$ is a CP Loewner curve, then for any $t\in [0,\ha T_\eta)$, $H^\eta_t=\HH(\eta[0,t];\infty)$.

Let $D$ be a simply connected domain and  $\eta\in \Sigma^{\ha D}$. For $t_0\in(0,\ha T_\eta)$, if there is $\eps_0>0$ such that $D(\eta[0,t_0];\gamma[t_0,t_0+\eps)\cap (D\sem \eta[0,t_0]))$ is a well-defined domain for any $\eps\in(0,\eps_0]$, then such domain does not depend on $\eps$ or $\eps_0$, and we denote it by $D(\eta|t_0^+)$. 
If $D(\eta|t_0^+)$ is a well-defined simply connected domain and has a prime end $w$  such that, whenever $t_n\to t_0$ along $S_{ t_0}:=\{t\in(t_0,\ha T_\eta):\eta(t)\in D(\eta|t_0^+)\}$, we have $\eta(t_n)\to w$, then we use $\eta(t_0^+)$ to denote this prime end $w$.
If $[\eta]\in \til\Sigma^{\ha D}$ and $\tau$ is a TCI time on $[\eta]$, which is less than $\ha T$, we define $D ([\eta]|{\tau}^+)=D (\eta|{\tau(\eta)}^+)$ when the latter domain is well-defined,
 and define $[\eta](\tau^+)=\eta(\tau(\eta)^+)$ when the prime end $\eta(\tau(\eta)^+)$ is well defined. The definitions clearly do not depend on the choice of the representative of $[\eta]$.

\begin{Proposition}
  \begin{enumerate}
    \item [(i)] Let $\eta$ is a CP Loewner curve. Then for any $t\in [0, \ha T_\eta)$, the domain $\HH(\eta|t^+) $ is well defined and equals   $H^\eta_t=\HH(\eta[0,t];\infty)$; and the prime end $\eta(t^+)$ of $\HH(\eta|t^+)$ is well defined, which equals $(g^\eta_t)^{-1}(W^\eta_t)$ and determines $\eta(t)$.
    \item [(ii)]  Let $\eta$ be an AP Loewner curve in a simply connected domain $D$ aimed at   $w_\infty$. Then for any $t\in[0,\ha T_\eta)$, the domain $D(\eta|t^+)$ is well defined and equals $D(\eta[0,t];w_\infty)$; and the prime end $\eta(t^+)$ of $D(\eta|t^+)$ is well defined and determines $\eta(t)$. If $\tau$ is a TCI time on the MTC Loewner curve $[\eta]$, which is less than $\ha T$, then  $D([\eta]|\tau^+)$ is well defined and equals $D([\eta]([0,\tau]);w_\infty)$, and  $[\eta](\tau^+)$ is well defined and determines $[\eta](\tau)$.
  \end{enumerate}
 \label{indep-target-1}
\end{Proposition}
\begin{proof}
  Part (i) follows from the well-known facts (cf.\ \cite{Law1}) that  for each $t_0\in [0,\ha T)$, the set $S^\eta_{t_0}:=\{t\in (t_0,\ha T):\eta(t)\in \HH\sem \eta[0,t_0]\}$ satisfies (1) $t_0=\inf {S^\eta_{t_0}}$; (2) $\eta(S^\eta_{t_0})\subset H^\eta_{t_0}= \HH(\eta[0,t_0];\infty)$;   (3) if $t_n\to t_0$ along $S^\eta_{t_0}$, then $g^\eta_{t_0}(\eta(t_n))\to W^\eta_{t_0}$;
  and (4) if $z\to W^\eta_{t_0}$ in $\HH$, then $(g^\eta_{t_0})^{-1}(z_n)\to \eta(t_0)$. Part (ii) follows immediately from (i).
\end{proof}

Thanks to the above proposition, we will use the symbol $D(\eta|t^+)$ instead of $D(\eta[0,t];w_\infty)$ because it reflects the fact that the domain depends only on $\eta(s)$, $0\le s<t'$, for any $t'\in(t,\ha T_\eta)$, but not on $w_\infty$. Similarly, we will use $D([\eta]|\tau^+)$ instead of $D([\eta]([0,\tau]);w_\infty)$.

Let $\eta$ be an AP Loewner curve in $\HH$ aimed at $\infty$. Then there is a unique CP Loewner curve $\eta_0$ such that $\eta$ is a time-change of $\eta_0$ via some $\theta:[0,\ha T_{\eta_0})\to [0,\ha T_\eta)$. For $t\in [0,\ha T_\eta)$, we define  $I^\eta_t=I^{\eta_0}_{\theta^{-1}(t)}$.
We similarly define $a^\eta_t,b^\eta_t\in\R$ and $\ha a^\eta_t,\ha b^\eta_t\in \ha\pa \HH(\eta|t^+)$. 

Let $\eta$ be an AP Loewner curve in $D$ aimed at $w_\infty$. For $t\in[0,\ha T_\eta)$, we define $I^{\eta}_{t^+;D}= f(I^{f^{-1}\circ \eta}_t)$, where $f$ is  a conformal map from $\HH$ onto $D$ such that $f(\infty)=w_\infty$. The definition does not depend on the choice of $f$. We omit $w_\infty$ because $I^{\eta }_{t^+;D}$ is the biggest prime end interval shared by $D$ and $D(\eta|t^+)$, and so is determined by $\eta(s)$, $0\le s<t'$, for any $t'\in(t,\ha T_\eta)$.  We similarly define prime ends $a^{\eta }_{t^+;D}$ and $b^{\eta}_{t^+;D}$  (resp.\ $\ha a^{\eta }_{t^+;D}$ and $\ha b^{\eta}_{t^+;D}$) of $D$ (resp.\ $D(\eta|t^+)$). 

If $[\eta]$ is an MTC Loewner curve in $D$ aimed at $w_\infty$, and $\tau$ is a TCI time on $[\eta]$ with $\tau(\eta)<\ha T_\eta$, then we define $I^{[\eta]}_{\tau^+;D}=I^{\eta}_{\tau(\eta)^+;D}$,
and similarly
 $a^{[\eta]}_{\tau^+;D}$, $b^{[\eta]}_{\tau^+;D}$, $\ha a^{[\eta]}_{\tau^+;D}$ and $\ha b^{[\eta]}_{\tau^+;D}$. These definitions do not depend on the choice of the representative of $[\eta]$.

\begin{Proposition}
  \begin{enumerate}
    \item [(i)] Let $\eta$ be a CP Loewner curve, and $t_0\in [0,\ha T_\eta)$. Let $u\in I^\eta_{t_0}$. Then  $\eta(t)$, $0\le t< \tau^\eta_u$, is an AP Loewner curve in $\HH$ aimed at $u$.
    \item [(ii)] Let $\eta$ be an AP Loewner curve in $D$ aimed at $w_\infty$. Let $t\in [0,\ha T_\eta)$. Then for any $u\in I^{\eta}_{t^+;D}$, there is  $t'\in (t,\ha T_\eta)$ such that $\K_{t'}(\eta)$ is an AP Loewner curve in $D$ aimed at  $u$
  \end{enumerate}
  \label{indep-target-2}
\end{Proposition}
\begin{proof}
Part (i) follows from the work of \cite{SW}, and Part (ii)  follows from (i).
\end{proof}

\begin{Definition}
  We call $\eta\in  \Sigma^{\ha D}$ an AP Loewner curve in $D$ (without a specified target) if for any $t\in [0,\ha T_\eta)$, $\K_t(\eta)$ is an AP Loewner curve in $D$ aimed at some  prime end of $D$. For such $\eta$, we   call the MTC curve $[\eta]\in \til\Sigma^{\ha D}$ an MTC Loewner curve in $D$. If $[\eta]$ is both an MTC loop rooted at $w$ defined at the end of Section \ref{section-lifetime} and an MTC Loewner curve in $D$, then we call $[\eta]$ an (MTC) Loewner bubble in $D$ rooted at $w$.
\end{Definition}


An AP Loewner curve aimed at a prime end is clear an AP Loewner curve without a specified target. The converse is not true. Below is an example.

\begin{Example}
  Let $\theta$ be an increasing homeomorphism from $(0,\infty)$ onto $(-\infty,\infty)$. Then the curve $\gamma:[0,\infty)\to \lin\HH^{\#}$  defined by $\gamma(0)=\infty$, $\gamma(t)=(\theta(t),\sin^2(\theta(t)))$, $0<t<\infty$, is  an AP Loewner curve, but no prime end can serve as its target. \label{Example-no-target}
\end{Example}

Suppose $\eta$ is an AP Loewner curve in $D$ without a specified target. Let $t_0\in [0,\ha T_\eta)$. Take $t_0'\in (t_0,\ha T_\eta)$. Then $\K_{t_0'}(\eta)$ is an AP Loewner curve in $D$ aimed at some prime end. By the earlier discussion, we have the existence of the domain $D(\K_{t_0'}(\eta)|t_0^+)$  and its prime end $\K_{t_0'}(\eta)(t_0^+)$  and prime end interval $I^{\K_{t_0'}(\eta)}_{t_0^+;D}$. This immediately implies the existence of  $D(\eta|t_0^+)$ and $\eta(t_0^+)$ and that they respectively equal  $D(\K_{t_0'}(\eta)|t_0^+)$  and  $\K_{t_0'}(\eta)(t_0^+)$.
Sine $I^{\K_{t_0'}(\eta)}_{t_0^+;D}$ is the greatest prime end interval shared by $D$ and $D(\K_{t_0'}(\eta)|t_0^+)=D(\eta|t_0^+)$, it depends only on $\eta(t)$, $0\le t<t'$ for any $t'\in (t_0,\ha T_\eta)$. Thus, $I^{\K_{t_0'}(\eta)}_{t_0^+;D}$ does not depend on $t_0'$, and so we denote it by $I^{\eta}_{t_0^+;D}$. We  similarly define $a^{\eta}_{t_0^+;D}$, $b^{\eta}_{t_0^+;D}$, $\ha a^{\eta}_{t_0^+;D}$ and  $\ha b^{\eta}_{t_0^+;D}$.
If  $\tau$ is a TCI time on $[\eta]$, which is less than $\ha T$, we then have the existence of the domain $D([\eta]|\tau^+)=D(\eta|\tau(\eta)^+)$ and it prime end $[\eta](\tau^+)=\eta(\tau(\eta)^+)$. We define $I^{[\eta]}_{\tau^+;D}=I^\eta_{\tau(\eta)^+;D}$, and similarly $a^{[\eta]}_{\tau^+;D}$, $b^{[\eta]}_{\tau^+;D}$, $\ha a^{[\eta]}_{\tau^+;D}$ and  $\ha b^{[\eta]}_{\tau^+;D}$.

 For an MTC Loewner curve $[\eta]$ in $\HH$ started from a real number, we define $\ha\sigma(\eta)$ to be the biggest $t_0\in [0,\ha T_\eta]$ such that $\infty\in I^\eta_t$ for $0\le t<t_0$. Then $\ha\sigma$ is a positive TCI stopping time, and  there is a unique CP Loewner curve $\eta^*$ such that $[\eta^*]=\K_{\ha\sigma}([\eta])$. The map $[\eta]\mapsto \eta^*$ is called the capacity parametrization map, and is denoted by $\cal P$.
 If  $\ha\sigma(\eta)< \ha T_\eta$, the part of $[\eta]$ after $\ha\sigma$ grows in a connected component of $\HH\sem [\eta]([0,\ha\sigma))$. Thus, we have \BGE \rad_x([\eta])=\rad_x({\cal P}([\eta])),\quad\mbox{for any }x\in\R.\label{radxPeta}\EDE

Recall that if $\gamma$ is a CP Loewner curve and $s_0\in [0, \ha T_\gamma)$, then the conformal map $(g^\gamma_{s_0})^{-1}$ from $\HH$ onto $ \HH(\gamma|s_0^+)$  extends continuously to $\lin\HH$, which implies that the embedding from $\HH(\gamma|s_0^+)$ into $\HH$  extends to a continuous map from $\ha{\HH(\gamma|s_0^+)}$  into $\ha \HH$. Thus, if $\gamma$ is an AP Loewner curve in $D$, and $s_0\in [0,\ha T_\gamma)$, then the embedding from $D(\eta|s_0^+)$ into $D$ extends to a continuous map  from $\ha{D(\eta|s_0^+)}$ into $\ha D$. So, every curve $\eta$ in $\ha{D(\eta|s_0^+)}$ can be viewed as a curve in $\ha D$.

\begin{Lemma}
  Let $\gamma$ be an AP Loewner curve in $D$. Let $s_0\in[0,\ha T_\gamma)$. Let $\eta$ be an AP Loewner curve in $D(\gamma|s_0^+)$ started from $\gamma(s_0^+)$ and aimed at $u\in I^\gamma_{s_0^+;D}\cup \{\ha a^\gamma_{s_0^+;D}, \ha b^\gamma_{s_0^+;D}\}$. Then $\gamma\oplus_{s_0} \eta$ is an AP Loewner curve in $D$, where the $\eta$ is understood as a curve in $\ha D$.
   \label{continue-Loewner}
\end{Lemma}
\begin{proof}
First suppose $\gamma$ has a target $w_\infty$. Assume $u=w_\infty$. By  conformal mapping and time-change we may assume that $D=\HH$ and $\gamma$ is a CP Loewner curve. Since $g^\gamma_{s_0}$ maps $\HH(\gamma|s_0^+)$ conformally onto $\HH$, fixes $\infty$ and sends $\gamma(s_0+)$ to $W^\gamma_{s_0}$, $\eta':=g^\gamma_{s_0}\circ \eta$ is an AP Loewner curve in $\HH$ started from $W^\gamma_{s_0}$ and aimed at $\infty$. By another time-change we may assume that $\eta'$ is also a CP Loewner curve. The conclusion follows since $\gamma\oplus_{s_0} \eta=\gamma\oplus_{s_0} (g^\eta_{s_0})^{-1}\circ \eta'$ is the CP Loewner curve driven by $W^\gamma\oplus_{s_0} W^{\eta'}$.



Now we do not assume $w_\infty=u$ but assume $u\in I^\gamma_{s_0^+;D}$. Then $\gamma\oplus_{s_0} \eta$ is an AP Loewner curve in $D$ aimed at $u$ since from $u\in I^\gamma_{s_0^+;D}$ we know there is some $s_0'\in (s_0,\ha T_\gamma)$ such that $\K_{s_0'}(\gamma)$ is an AP Loewner curve in $D$ aimed at $u$, and $\gamma\oplus_{s_0} \eta=\K_{s_0'}(\gamma)\oplus_{s_0} \eta$.

Assume now $u$ equals $ \ha a^\gamma_{s_0^+;D}$ or $\ha b^\gamma_{s_0^+;D}$. For any $t_0\in (0,\ha T_\eta)$, $I^{\eta}_{t_0^+; D(\gamma|s_0^+)}$ is an open prime end interval of $D(\gamma|s_0^+)$ containing $u$, and so has nonempty intersection with $I^\gamma_{s_0^+;D}$ since $u$ is an endpoint of $I^\gamma_{s_0^+;D}$. Choosing $u'$ from the intersection. Then $\K_{t_0}(\eta)$ is aimed at such $u'$. From the above paragraph, $\K_{s_0+t_0}(\gamma\oplus_{s_0} \eta)=\gamma\oplus_{s_0} \K_{t_0}(\eta)$ is an AP Loewner curve in $D$ aimed at $u'$. Since this holds for any $t_0\in (0,\ha T_{\eta})$, we see that $\gamma\oplus_{s_0}\eta$ is an AP Loewner curve in $D$.

Finally, we do not assume that $\gamma$ has a target. Let $s_0'\in(s_0,\ha T_\gamma)$ and $\gamma'=\K_{s_0'}(\gamma)$. Then $\gamma'$ is an AP Loewner curve in $D$ with a target, and $I^{\gamma'}_{s_0}=I^\gamma_{s_0}$, $\ha a^{\gamma'}_{s_0}=\ha a^\gamma_{s_0}$ and $\ha b^{\gamma'}_{s_0}=\ha b^\gamma_{s_0}$.
So, by the above result, $\gamma\oplus_{s_0} \eta=\gamma'\oplus_{s_0}\eta$ is an AP Loewner curve in $D$.
\end{proof}



Let $D$ be a simply connected domain, $w\in\ha\pa D$ and $v\in \ha\pa_w D$. Let $\eta$ be an AP Loewner curve in $D$ started from $w$, and $t\in [0,\ha T_\eta)$. We define $[v]^{\eta}_{t^+;D}\in \ha\pa_{\eta(t^+)} D(\eta|t^+)$ as follows. If $v\in I^{\eta}_{t^+;D}$, $[v]^{\eta}_{t^+;D}=v$; if $v\in [w^+, b^{\eta}_{t^+;D}]_{\ha\pa D}$, $[v]^{\eta}_{t^+;D}= \ha b^{\eta}_{t^+;D}$ or $(\ha b^{\eta}_{t^+;D})^+$ depending on whether $\eta(t^+)\ne \ha b^{\eta}_{t^+;D}$; and if $v\in [a^{\eta}_{t^+;D},w^-]_{\ha\pa D}$, $[v]^{\eta}_{t^+;D}= \ha a^{\eta}_{t^+;D}$ or $(\ha a^{\eta}_{t^+;D})^-$ depending on whether $\eta( t^+)\ne \ha a^{\eta}_{ t^+;D}$. In some situation,
when $[v]^{\eta}_{t^+;D}=( \ha b^{\eta}_{t^+;D})^+$ (resp.\ $(\ha a^{\eta}_{t^+;D})^-$), we treat it as the  genuine prime end $\ha b^{\eta}_{t^+;D}$ (resp. $\ha a^{\eta}_{t^+;D}$). When $\eta$ is a CP Loewner curve, we write $[v]^\eta_{t}$ for $[v]^\eta_{t^+;\HH}$. If $\tau<\ha T$ is a TCI time on an MTC Loewner curve $[\eta]$,   we define $[v]^{[\eta]}_{\tau^+;D}=[v]^\eta_{\tau(\eta)^+;D}$. 

\begin{Lemma}
Let $D$ be a simply connected domain with two distinct prime ends $w_0,w_\infty$.
Let $\gamma$ be an AP Loewner curve in $D$ started from $w_0$ and aimed at $w_\infty$,  $s_0\in [0,\ha T_\gamma)$, and $\eta$ be an AP Loewner curve in $D(\gamma|s_0^+)$ started from $\gamma(s_0^+)$ and aimed at $w_\infty$. Let $t_0\in [0,\ha T_\eta)$. Then
  \BGE [[v]^{\gamma}_{s_0^+;D}]^{\eta}_{t_0^+; D(\gamma|s_0^+)}=[v]^{\gamma\oplus_{s_0} \eta}_{(s_0+t_0)^+;D},\quad \forall v\in \ha\pa_{w_0} D.\label{[[v]]}\EDE
\label{[[v]]-lem}
\end{Lemma}
\begin{proof}
  From the first paragraph of the previous proof, we may assume that $D=\HH$, $w_0\in\R$,   $w_\infty=\infty$, and both $\gamma$   and $\eta':=g^\gamma_{s_0}(\eta)$ are CP Loewner curves.  Then $g^{\gamma\oplus_{s_0} \eta}_{s_0+t_0}= g^{\eta'}_{t_0}\circ g^\gamma_{s_0}$.
   Applying the common map  to both sides of (\ref{[[v]]}) and using the fact that $g^{\gamma}_{s_0}$ maps $\HH(\gamma|s_0^+)$ conformally onto $\HH$, we see that (\ref{[[v]]}) is equivalent to
 \BGE g^{\eta'}_{t_0}( [g^{\gamma}_{s_0}([v]^{\gamma}_{s_0 })]^{\eta'}_{t_0}) =g^{\gamma\oplus_{s_0} \eta'}_{s_0+t_0}([v]^{\gamma\oplus_{s_0} \eta}_{s_0+t_0}),\quad \forall v\in \ha\pa_{w_0} \HH. \label{[[v]]'}\EDE
Let $v\in\ha\pa_{w_0}\HH$. If $v=\infty$,  (\ref{[[v]]'}) is trivial since both sides equal $\infty$. Suppose $v\ne \infty$. Assume by symmetry that $v\in [w_0^+,\infty)$. We will use the following equality.
For a CP Loewner curve $\xi$,
\BGE   g^\xi_t( [u]^\xi_{t})=  (W^\xi_t)^+\vee \lim_{x\downarrow u \vee b^\xi_t} g^\xi_t(x)\ge D^\xi_t,\quad \mbox{if }t\in[0,\ha T_\xi) \mbox{ and } u\in[\xi(0)^+,\infty).\label{vetat}\EDE
To see that it holds, note that if $u\in (b^\xi_t,\infty)$, then $[u]^\xi_t=u$ and $\lim_{x\downarrow u \vee b^\xi_t} g^\xi_t(x)=g^\xi_t(u)>D^\xi_t$; and if $u\in [\xi(0)^+, b^\xi_t]$, then $g^\xi_t([u]^\xi_t)=(W^\xi_t)^+\vee D^\xi_t$ and  $\lim_{x\downarrow u \vee b^\xi_t} g^\xi_t(x)=D^\xi_t $.

Since $\lim_{x\downarrow  b^\gamma_{s_0}} g^\gamma_{s_0}(x)= D^\gamma_{s_0}\ge W^\gamma_{s_0}$, by (\ref{vetat})  $ x\downarrow v\vee b^\gamma_{s_0}\LR g^\gamma_{s_0}(x)\downarrow g^{\gamma}_{s_0}([v]^{\gamma}_{s_0 })$.
Since $K^{\gamma\oplus\eta}_{s_0+t_0}=K^{\gamma}_{s_0}\cup (g^{\gamma}_{s_0})^{-1}(K^{\eta'}_{t_0})$, we get
$ x\downarrow b^{\gamma\oplus_{s_0} \eta}_{s_0+t_0}\LR g^\gamma_{s_0}(x)\downarrow D^\gamma_{s_0}\vee b^{\eta'}_{t_0}$.  These two equivalence relations together with the inequalities $ b^{\gamma\oplus_{s_0} \eta}_{s_0+t_0}\ge b^\gamma_{s_0}$ and $g^{\gamma}_{s_0}([v]^{\gamma}_{s_0 })\ge D^\gamma_{s_0}$ imply that
\BGE x\downarrow v\vee b^{\gamma\oplus_{s_0} \eta}_{s_0+t_0}\LR g^\gamma_{s_0}(x)\downarrow g^{\gamma}_{s_0}([v]^{\gamma}_{s_0 })\vee b^{\eta'}_{t_0}.\label{down-vee}\EDE
Since $W^{\gamma\oplus_{s_0}\eta}=W^\gamma\oplus_{s_0} W^{\eta'}$   and $g^{\gamma\oplus_{s_0}\eta}_{s_0+t_0}=g^{\eta'}_{t_0}\circ g^\gamma_{s_0}$, we have
$$g^{\gamma\oplus_{s_0} \eta'}_{s_0+t_0}([v]^{\gamma\oplus_{s_0} \eta}_{s_0+t_0}){\stackrel{(\ref{vetat})}{=}}(W^{\gamma\oplus_{s_0}\eta}_{s_0+t_0})^+\vee \lim_{x\downarrow v\vee b^{\gamma\oplus_{s_0} \eta}_{s_0+t_0}} g^{\gamma\oplus_{s_0}\eta}_{s_0+t_0}(x)$$ $${\stackrel{(\ref{down-vee})}{=}}(W^{ \eta'}_{ t_0})^+\vee \lim_{y \downarrow   g^{\gamma}_{s_0}([v]^{\gamma}_{s_0 })\vee b^{\eta'}_{t_0}} g^{ \eta'}_{ t_0}(y){\stackrel{(\ref{vetat})}{=}}  g^{\eta'}_{t_0}( [g^{\gamma}_{s_0}([v]^{\gamma}_{s_0 })]^{\eta'}_{t_0}).$$ So we get the desired (\ref{[[v]]'}), and the proof is complete.
 \end{proof}


\begin{Corollary}
  Let $[\gamma]$ be an MTC Loewner curve in $D$ started from $w_0$ and aimed at $w_\infty$, and $\sigma$ be a TCI time on $[\gamma]$, which is less than $\ha T$. Let $[\eta]$ be an MTC Loewner curve in $D([\gamma]|\sigma^+)$ started from $[\gamma](\sigma^+)$ and aimed at $w_\infty$. Let $\tau$ be a TCI time on $[\gamma]\oplus_\sigma[\eta]$, which is less than $\ha T$. Suppose  $\tau(\gamma'\oplus_{\sigma(\gamma')}\eta')\ge \sigma(\gamma')$ for all $\gamma'\in[\gamma]$ and $\eta'\in[\eta]$.  Then
$$ [[v]^{[\gamma]}_{\sigma^+;D}]^{[\eta]}_{(\tau_{[\gamma];\sigma})^+; D([\gamma]|\sigma^+)}=[v]^{[\gamma]\oplus_{\sigma} [\eta]}_{\tau^+;D},\quad \forall v\in \ha\pa_{w_0} D.$$
\label{[[v]]-cor}
\end{Corollary}


\subsection{SLE$_\kappa(\ulin\rho)$ processes} \label{section-SLE}
We now review   multi-force-point SLE$_\kappa(\ulin\rho)$ curves from \cite{MS1}. 
Let $\kappa>0$ and $\ulin\rho=(\rho_1,\dots,\rho_m)\in\R^m$.
Let $w\in\R$ and $v_1,\dots, v_m\in \ha\pa_w \HH$ be such that $\sum_{j:v_j=w^+}\rho_j>-2$ and $\sum_{j:v_j=w^-}\rho_j>-2$. Let $\ulin v=(v_1,\dots,v_m)$. Consider the following system of SDE/ODE with $B_t$ being a standard Brownian motion:
\begin{align*}
dW_t& =\sum_{j=1}^m  {\ind}_{\{W_t\ne V^j_t\}} \frac{\rho_j}{W_t-V^j_t}\,dt+\sqrt\kappa dB_t,\quad W_0=w; \\
dV^j_t& ={\ind}_{\{W_t\ne V^j_t\}} \frac 2{V^j_t-W_t}\,dt,\quad V^j_0=v_j,\quad 1\le j\le m.
\end{align*}
Here if some $v_j$ equals $\infty$, then $V^j_t$ is constant $\infty$, and $\frac{1}{V^j_t-W^j_t}$ is constant $0$.
It is known that a weak solution in the integral sense with the maximal interval, say $[0,\ha T)$, of the system exists and is unique in law, and the $W$ in the solution a.s.\ generates a CP Loewner curve $\eta$, which we call a  CP SLE$_\kappa(\ulin\rho)$ curve started from $w$ with force points $\ulin v$, or simply a CP SLE$_\kappa(\ulin\rho)$ at $(w;\ulin v)$. We let $\PP_{\kappa;\ulin\rho}^{w;\ulin v}$ denote the law of  $W_t$, $0\le t<\ha T$, For each $j$, the $V^j$ is called the force point process started from $v_j$.
Below are some basic properties that will be used later.
\begin{itemize}
\item  If $\ha T=\infty$,  then a.s.\ $\lim_{t\to\infty} \eta(t)=\infty$. If $\ha T<\infty$, then   $\eta(\ha T):=\lim_{t\to \ha T^-}\eta(t)$,   $W_{\ha T}:=\lim_{t\to \ha T^-}W_t$ and $V^j_{\ha T}:=\lim_{t\to \ha T^-}V^j_t$ a.s.\ all exist and are finite.
\item Every $V^j$ is   determined by the formula  $V^j_t=\la v_j\ra^W_t$, where $\la v_j\ra^W_t$ is defined to be the $g^W_t([v_j]^{\eta^W}_t)$  without any possible superscript ``$+$'' or ``$-$''. Equivalently,  if $t<\tau^W_{v_j}$, then $V^j_t=g^W_t(v_j)$; otherwise   $V^j_t=D^W_t$ (resp.\ $C^W_t$) if $v_j>w$ (resp.\ $v_j<w$).
\item $W$ and $\eta$ satisfy the following translation and scaling property. If $a>0$ and $b\in\R$, then $ a W_{\cdot /a^2}+b$ has the law $\PP^{a w+b;(av_1+b,\dots,av_m+b)}_{\kappa;\ulin\rho}$, and generates the curve $a \eta(\cdot /a^2)+b$.
\item Suppose all $v_j$ lie on the same side of $w$, and $|v_j-w|$ increases in $j$. If $\sum_{j=1}^k \rho_j>-2$ for all $1\le k\le m$, then a.s.\ $\ha T=\infty$; if the inequalities hold for $1\le k\le m-1$, $v_m\ne \infty$, and $\sum_{j=1}^m \rho_j\le \frac\kappa 2-4$, then a.s.\ $\tau_{v_m}<\infty$ and $\lim_{t\uparrow \tau_{v_m}} \eta(t)= v_m$.
\item
    $\eta$ satisfies the following Domain Markov Property (DMP): if $\tau$ is an $\F$-stopping time, then conditionally on $\F_\tau$ and the event that $\tau<\ha T$, there is a CP SLE$_\kappa(\ulin\rho)$  $\eta'$ at $(W_\tau;\ulin {\til V}_\tau)$ such that $\tau(\eta+\cdot)=(g^\eta_\tau)^{-1} (\eta')$, where $\ulin {\til V}_\tau=(\til V^1_\tau,\dots, \til V^m_\tau)$ and  $\til V^j_\tau$ is defined such that $ \til V^j_\tau=V^j_\tau$ if $V^j_\tau\ne W_\tau$, and $\til V^j_\tau=W_\tau^+$ (resp.\ $W_\tau^-$) if $V^j_\tau=W_\tau$ and $v_j>w$ (resp.\ $v_j<w$).
\item A single-force-point CP SLE$_\kappa(\rho)$ curve $\eta$ at $(w;v)$ with  $v\in[w^+,\infty)$ has the following boundary behavior. Let $I=(w,\infty)$ if $v=w^+$ and $I=[v,\infty)$ if $v\in(w,\infty)$. If $\rho\ge \frac\kappa 2-2$, then a.s.\ $\eta\cap I=\emptyset$; if $\rho\le \frac\kappa 2-4$, then a.s.\ $\eta\supset I$; and if $\rho\in(\frac\kappa 2-4,\frac\kappa 2-2)$, then a.s.\ $\eta\cap I$ is unbounded, but for every fixed $u\in I$, a.s.\ $ \eta\not\ni u$.
\end{itemize}

Let $D$ be a simply connected domain in $\ha\C$, $w_0\ne w_\infty\in\ha\pa D$ and $v_1,\dots,v_m\in \ha\pa_{w_0} D$. Let $\kappa>0$ and $\ulin\rho=(\rho_1,\dots,\rho_m)\in\R^m$. Let $f$ be a conformal map from $\HH$ onto $D$ such that   $f(\infty)=w_\infty$.
Let $\eta$ be a CP SLE$_\kappa(\ulin\rho)$ at $(f^{-1}(w_0);f^{-1}(v_1),\dots,f^{-1}(v_m))$. Then the random  MTC  curve $[f\circ \eta]$  is called an (MTC) SLE$_\kappa(\ulin\rho)$  in $D$ started from $w_0$, with force points $v_1,\dots,v_m$, and aimed at $w_\infty$, or an SLE$_\kappa(\ulin\rho)$  at $D:(w_0;\ulin v)\to w_\infty$ for simplicity.
By the scaling and translation property of CP SLE$_\kappa(\ulin\rho)$, the law of   SLE$_\kappa(\ulin\rho)$  in $D$ does not depend on the choice of $f$.


An SLE$_\kappa(\ulin\rho)$ curve $[\eta]$ at $D:(w_0;v_1,\dots,v_m)\to w_\infty$  satisfies the following DMP. For any TCI $\F$-stopping time $\tau$, conditionally on $\til\F_\tau$ and the event $\{\tau<\ha T\}$, the part of $[\eta]$ after $\tau$ is an SLE$_\kappa(\ulin\rho)$ at $D([\eta]|\tau^+):([\eta](\tau)^+;[v_1]^{[\eta]}_{\tau^+;D},\dots,[v_m]^{[\eta]}_{\tau^+;D})\to w_\infty$.
This fact follows from the DMP of CP SLE$_\kappa(\ulin\rho)$.  

\subsection{Radial Bessel processes} \label{Section-Bessel}

Let $\delta_+,\delta_-\in\R$. Let $B_t$  be a standard Brownian motion. By \cite[Section 3]{RY}, for any $x\in[-1,1]$, the SDE
 \BGE dZ_t=\sqrt{0\vee (1-Z_2^2)}dB_t-\frac{\delta_+}4(Z_t+1)dt -\frac{\delta_-}4(Z_t-1),\quad Z_0=x,\label{sigma-b}\EDE
 has a unique strong solution with lifetime $\infty$. If $\delta_+>0$ and $Z_0\le 1$, then $Z$ does not take values in $(1,\infty)$, and if $\delta_->0$ and $Z_0\ge -1$, then $Z$ does not take values in $(-\infty,-1)$. Thus, if $\delta_+,\delta_->0$, and $Z_0\in[-1,1]$, then $Z$ stays in $[-1,1]$, and the SDE (\ref{sigma-b}) becomes
  \BGE dZ_t=\sqrt{1-Z_t^2}dB_t-\frac{\delta_+}4(Z_t+1)dt-\frac{\delta_-}4(Z_t-1)dt.\label{Bessel-SDE}\EDE
  We call the solution $(Z_t)_{t\ge 0}$ a radial Bessel process of dimension $(\delta_+,\delta_-)$ started from $x$, and use $\nu_x^{\delta_+,\delta_-}$ to denote its law, which is a probability measure on $C([0,\infty),\R)$. It satisfies a Markov property, which  could be described by the formula:
\BGE {\ind}_{Z\in\{\tau<\ha T\}} \nu^{\delta_+,\delta_-}_x(dZ)\oplus_\tau \nu^{\delta_+,\delta_-}_{Z_{\tau}}={\ind}_{\{\tau<\ha T\}} \nu^{\delta_+,\delta_-}_x,\quad \forall \F\mbox{-stopping time }\tau.\label{Markov-nu}\EDE

\begin{Proposition}
  If $\delta_+,\delta_->0$ and $x\in [-1,1]$, then the solution of SDE (\ref{Bessel-SDE}) has transition density
   $$ p_t(x,y)= w_{\alpha_+,\alpha_-}(y) \sum_{n=0}^\infty \frac{P_n^{(\alpha_+,\alpha_-)}(x) P_n ^{(\alpha_+,\alpha_-)}(y)}{ \int_{-1}^1 w_{\alpha_+,\alpha_-}(s) P_n^{(\alpha_+,\alpha_-)}(s)^2 \mA(ds)}\exp({-\frac 1 2 n(n+1+\alpha_++\alpha_-) t}) .$$ 
   where  $\alpha_\pm =\frac{\delta_\pm}2-1$, $ w_{\alpha_+,\alpha_-}(s)= (1-s)^{\alpha_+}(1+s)^{\alpha_-}$,  and $P_n^{(\alpha_+,\alpha_-)}$ are Jacobi polynomials with indices $(\alpha_+,\alpha_-)$. 
   \label{Bessel-transition}
\end{Proposition}
\begin{proof}
  See the last remark in \cite[Appendix B]{tip}.
\end{proof}


\begin{Proposition}
Using the symbols in Proposition \ref{Bessel-transition}, let $p_\infty(x)=\frac{ w_{\alpha_+,\alpha_-}(x)} {\int_{-1}^1w_{\alpha_+,\alpha_-}(s)\mA(ds) }$. Then for any $y\in[-1,1]$ and $t>0$,
  \BGE \int_{-1}^1 p_\infty(x) p_t(x,y)\mA( dx)=p_\infty(y) ;\label{invariant=}\EDE
  and there exist constants $C,L\in(0,\infty)$ such that for any $x,y\in[-1,1]$,
  \BGE |p_t(x,y)- p_\infty(y)|< C p_\infty(y) e^{-\frac{\delta_++\delta_-}4 t},\quad \mbox{if } t>L.\label{transition-rate-i}\EDE
\label{Prop-invariant}
\end{Proposition}
\begin{proof}
  This is similar to \cite[Corollary 8.8 and Formula (8.9)]{tip}.
\end{proof}

By (\ref{invariant=}), $p_\infty$ is the  invariant probability density of a radial Bessel process of dimension $(\delta_+,\delta_-)$. By a standard argument,  there exists a random process $Z_t$, $t\in\R$, such that for any fixed $t_0\in\R$, $Z_{t_0}\sim \ind_{(-1,1)} p_\infty\cdot \mA$, and conditionally on $Z_t$, $t\le t_0$,  $(Z_{t_0+t})_{t\ge 0}$ has the law $\nu^{\delta_+,\delta_-}_{Z_{t_0}}$. Such $Z$ is called a stationary radial Bessel process of dimension $(\delta_+,\delta_-)$, and its law is denoted by $\nu^{\delta_+,\delta_-}_{\R}$, which is a probability measure on $C(\R,\R)$. Using the symbols in Section \ref{section-lifetime} extended to \BGE \Sigma_{\R}:=\bigcup_{\ha T\in \R\cup \{+\infty\}} C((-\infty,\ha T),\R), \label{sigma-R}\EDE we may express the above Markov property formally by
\BGE \nu^{\delta_+,\delta_-}_{\R}(dZ)\oplus_{t_0} \nu^{\delta_+,\delta_-}_{Z_{t_0}}= \nu^{\delta_+,\delta_-}_{\R},\quad \forall t_0\in\R.\label{Markov-mu-R}\EDE
Let $\F^{\R}=(\F^{\R}_t)_{t\in\R}$ be the right-continuation of the natural filtration on $\Sigma_{\R}$.


For the rest of this subsection, let $\delta_+>0$, $\delta_-<2$, and $\delta_-^*=4-\delta_->2$.  Let $p_t^*(x,y)$ and $p_\infty^*(y)$ be respectively the transition density and stationary density for $\nu^{\delta_+,\delta_-^*}_\cdot$.
  Since $\delta_-<2$, a process $Z$ following the SDE (\ref{Bessel-SDE}) hits $-1$ at a finite time. We kill $Z$ once it hits $-1$. Then the resulted process stays in $(-1,1]$, and so satisfies (\ref{Bessel-SDE}) throughout its lifespan. We use $\mu_x^{\delta_+,\delta_-}$ to denote the law of such truncated $Z$, which is a probability measure on $\Sigma$. Then $\mu_x^{\delta_+,\delta_-}$ also satisfies (\ref{Markov-nu}).

\begin{Lemma} Let $\alpha_-=\frac{\delta_-}2-1<0$ be as in Proposition \ref{Bessel-transition} and $\alpha_0=\frac 18 \delta_+(2-\delta_-)>0$.  For  $Z\in\Sigma$ with $Z_0>-1$,  define
 \BGE M^Z_t=e^{-\alpha_0 t} \Big( \frac{1+Z_t}2 \Big)^{\alpha_-}\label{MZt}\EDE
 up to and not including the first time that $Z_t=-1$. Then for any $x\in(-1,1)$ and any   $\F$-stopping time $\tau$,
 \BGE \frac{ d \mu_x^{\delta_+,\delta_-} |\F_\tau\cap \{\ha T>\tau\}}{  d\nu_x^{\delta_+,\delta_-^*} |\F_\tau\cap \{\ha T>\tau\}}=\frac{M^\cdot_\tau}{M^\cdot_0},\label{Girsanov}\EDE
 i.e., $Z\mapsto {M^Z_\tau}/{M^Z_0} $ is the RN derivative of $\mu_x^{\delta_+,\delta_-} $ w.r.t.\ $ \nu_x^{\delta_+,\delta_-^*}$ on $\F_\tau\cap \{\ha T>\tau\}$.
  \label{Girsanov-lem}
\end{Lemma}
\begin{proof}
Let $N^Z_t=M^Z_t/M^Z_0$.
  Suppose $(Z_t)_{t\ge 0}\sim \nu_x^{\delta_+,\delta_-^*}$. It is straightforward to check that $N^Z_t$ is a local martingale using It\^o's formula. Let $\tau_n$ be the first time such that $N^Z_t $. Then $N^Z_{t\wedge \tau_n}$ is a bounded martingale. Using Girsanov Theorem, we find that $d\mu_x^{\delta_+,\delta_-} /d\nu_x^{\delta_+,\delta_-^*} =N^Z_{\tau_n}$ on $\F_{\tau_n}$, which   implies that $d\mu_x^{\delta_+,\delta_-} /d\nu_x^{\delta_+,\delta_-^*} =\EE[N^Z_{\tau_n}|\F_{\tau\wedge\tau_n}]= N^Z_{\tau\wedge \tau_n}$ on $\F_{\tau\wedge\tau_n}$. Since $\F_\tau\cap \{\tau<\tau_n\}=\F_{\tau\wedge \tau_n}\cap \{\tau<\tau_n\}$ and $\{\tau<\ha T\}=\bigcup_n \{\tau<\tau_n\}$, we get (\ref{Girsanov}).
\end{proof}


Recall that $\Sigma_t=\{f\in\Sigma:\ha T_f>t\}$ and $\pi_t(f)=f(t)$ for $f\in\Sigma_t$. For a random time $\tau$, we define $\Sigma_\tau=\{f\in\Sigma:\ha T_f>\tau(f)\}$, and $\til \Sigma_\tau=\pi(\Sigma_\tau)$ if $\tau$ is TCI. For $0\le t'<t$, we use $\pi_{t',t}$ to denote the map $\Sigma_{t}\ni f\mapsto (f(t'),f(t))$. Recall the space $\Sigma_{\R}$ defined in (\ref{sigma-R}) and the filtration $\F^\R$ defined after (\ref{Markov-mu-R}).
We similarly define the spaces $\Sigma^{\R}_t=\{f\in\Sigma_{\R}:\ha T_f>t\}$ and $\Sigma^{\R}_\tau=\{f\in\Sigma_{\R}:\ha T_f>\tau(f)\}$ and the  maps $\pi_t^{\R} $ and $\pi_{t',t}^{\R}$ on $\Sigma^{\R}_t$ for $t'<t\in\R$.

\begin{Lemma}
There is a unique $\sigma$-finite measure $\mu^{\delta_+,\delta_-}_{\R}$ on $\Sigma_{\R}$  such that for each fixed $\F^{\R}$-stopping time $\tau>-\infty$, the following holds.
 \begin{itemize}
           \item [(i)] $\mu^{\delta_+,\delta_-}_{\R}$ satisfies the Markov property:
         $${\ind}_{\Sigma^\R_\tau} \mu^{\delta_+,\delta_-}_{\R}(dZ) \oplus_\tau \mu_{Z_\tau}^{\delta_+,\delta_-}={\ind}_{\Sigma^\R_\tau} \mu^{\delta_+,\delta_-}_{\R}.$$
         \item [(ii)] $\mu^{\delta_+,\delta_-}_{\R}\ll \nu_{\R}^{\delta_+,\delta_-^*}$ on $\F_\tau\cap\Sigma^\R_\tau$ with the RN derivative being the $M^\cdot_\tau$ in (\ref{MZt}).
         \item [(iii)] For any  deterministic  time $t_0\in\R$, $\mu^{\delta_+,\delta_-}_{\R}|_{\Sigma^{\R}_{t_0}}$ is a finite measure.
         \end{itemize}
  Moreover, if another measure  satisfies (iii) and (i) for deterministic times, then it equals some constant times $\mu^{\delta_+,\delta_-}_{\R}$.
  \label{stationary-quasi}
\end{Lemma}
\begin{proof}
  For each $t_0\in\R$, we define a measure $\mu_{t_0}$ on $\Sigma^{\R}_{t_0}$ by
  \BGE \mu_{t_0}=\nu_{\R}^{\delta_+,\delta_-^*}(dZ)\oplus_{t_0}  M^Z_{t_0}\cdot \mu^{\delta_+,\delta_-}_{Z_{t_0}}. \label{mu_t}\EDE
  The $\mu_{t_0}$ is a finite measure since its total mass is
  $$e^{-\alpha_0 t_0}\int_{-1}^1 \Big (\frac{1+x}2\Big)^{\alpha_-} p_\infty^*(x) \mA(dx)\lesssim e^{-\alpha_0 t_0}\int_{-1}^1 (1-x)^{\frac{\delta_+}2-1}\mA(dx)<\infty. $$

Suppose $t_1<t_2$. Using (\ref{Markov-nu},\ref{Markov-mu-R},\ref{Girsanov},\ref{mu_t})  we get
\begin{align*}
   \mu_{t_2}&  \stackrel{ {(\ref{mu_t})}}{=}\nu_{\R}^{\delta_+,\delta_-^*}(dY)\oplus_{t_2}  M^Y_{t_2}\cdot \mu^{\delta_+,\delta_-} _{Y_{t_2}}\\
   &\stackrel{(\ref{Markov-mu-R})}{=}(\nu_{\R}^{\delta_+,\delta_-^*}(dZ)\oplus_{t_1} \nu^{\delta_+,\delta_-^*}_{Z_{t_1}}(dX) )\oplus_{t_2} M^{Z\oplus_{t_1} X}_{t_2}\cdot \mu^{\delta_+,\delta_-} _{(Z\oplus_{t_1} X)_{t_2}} \\
   &   \stackrel{\phantom{(\ref{Markov-mu-R})}}{=}\nu_{\R}^{\delta_+,\delta_-^*}(dZ)\oplus_{t_1}M^Z_{t_1}\cdot (( M^X_{t_2-t_1}/M^X_0)\cdot \nu^{\delta_+,\delta_-^*}_{Z_{t_1}}(dX) \oplus_{t_2-t_1}  \mu^{\delta_+,\delta_-}_{X_{t_2-t_1}})\\
  &   \stackrel{(\ref{Girsanov})}{=}\nu_{\R}^{\delta_+,\delta_-^*}(dZ)\oplus_{t_1}M^Z_{t_1}\cdot ({\ind}_{\Sigma_{t_2-t_1}} \mu^{\delta_+,\delta_-}_{Z_{t_1}}(dX) \oplus_{t_2-t_1}  \mu^{\delta_+,\delta_-}_{X_{t_2-t_1}})\\
  &  \stackrel{(\ref{Markov-nu})}{=}\nu_{\R}^{\delta_+,\delta_-^*}(dZ)\oplus_{t_1}M^Z_{t_1}\cdot {\ind}_{\Sigma_{t_2-t_1}} \mu^{\delta_+,\delta_-}_{Z_{t_1}}\\
   &\stackrel{\phantom{(\ref{Markov-mu-R})}}{=}\ind_{Z\oplus Y\in\Sigma^{\R}_{t_2}} (\nu_{\R}^{\delta_+,\delta_-^*}(dZ)\oplus_{t_1} M^Z_{t_1} \cdot \mu^{\delta_+,\delta_-}_{Z_{t_1}}(dY))\stackrel{ {(\ref{mu_t})}}{=} \ind_{\Sigma^{\R}_{t_2}} \mu_{t_1}.
\end{align*}
Here in the third ``$=$'' we used the facts that when $X$ follows the conditional law $\nu^{\delta_+,\delta_-^*}_{Z_{t_1}}$, then $M^X_0=M^Z_{t_1}$, $M^X_{t_2-t_1}=M^{Z\oplus_{t_1}X}_{t_2}$ and $(Z\oplus_{t_1} X)_{t_2}=X_{t_2-t_1}$.
Since $\mu_{t_1}|_{\Sigma^{\R}_{t_2}}=\mu_{t_2}$ for any $t_1<t_2\in\R$, and $\Sigma_{\R}=\bigcup_{n\in\N}\Sigma^{\R}_{-n}$, there is a unique measure $\mu^{\delta_+,\delta_-}_{\R}$ on $\Sigma_{\R}$ such that $\mu^{\delta_+,\delta_-}_{\R}|_{\Sigma^{\R}_t}=\mu_t$ for each $t\in\R$.

We now check that $\mu^{\delta_+,\delta_-}_{\R}$ satisfies (i)-(iii). Since $\mu^{\delta_+,\delta_-}_{\R}|_{\Sigma^{\R}_{t_0}}=\mu_{t_0}$, and $\mu_{t_0}$ is a finite measure, we see that $\mu^{\delta_+,\delta_-}_{\R}$ satisfies (iii). Since ${\ind}_{\Sigma^{\R}_{t_0}}\mu^{\delta_+,\delta_-}_{\R}= \mu_{t_0}$, by the definition of $\mu_{t_0}$, $\mu^{\delta_+,\delta_-}_{\R}$ satisfies (i) and (ii) for deterministic times. Let $\tau$ be any $\F^{\R}$-stopping time. First, suppose there is a deterministic time $t_0\in\R$ such that $\tau> t_0$. Let $\tau_{Z;t_0}$ be  as defined in (\ref{tau-g-formula'}). Since $\mu^{\delta_+,\delta_-}_{\R}$ satisfies (i) for deterministic times, using (\ref{Markov-nu}) for $\mu_{x}^{\delta_+,\delta_-}$ and (\ref{fgh}) we get
\begin{align*}
& \ind_{\Sigma^\R_\tau} \mu^{\delta_+,\delta_-}_{\R}=\ind_{\Sigma^\R_{t_0}}\mu^{\delta_+,\delta_-}_{\R}(dZ) \oplus_{t_0} \ind_{\Sigma_{\tau_{Z;t_0}}} \mu_{Z_{t_0}}^{\delta_+,\delta_-}\\
&\stackrel{(\ref{Markov-nu})}{=}\ind_{\Sigma^\R_{t_0}}\mu^{\delta_+,\delta_-}_{\R}(dZ) \oplus_{t_0} \Big( \ind_{\Sigma_{\tau_{Z;t_0}}}   \mu_{Z_{t_0}}^{\delta_+,\delta_-}(dW)\oplus_{\tau_{Z;t_0}}\mu_{W_{\tau_{Z;t_0}}}^{\delta_+,\delta_-} \Big ) \\ & \stackrel{ {(\ref{fgh})}}{=}\ind_{ \Sigma^\R_{\tau}} \mu^{\delta_+,\delta_-}_{\R}(dX)\oplus_{\tau} \mu_{X_\tau}^{\delta_+,\delta_-},
\end{align*}
where in the last step we use $X$ to replace $Z\oplus_{t_0}W$. So we get (i) for such $\tau$.

 Since  $\mu^{\delta_+,\delta_-}_{\R}$ satisfies (i,ii) for deterministic times, using (\ref{Girsanov},\ref{Markov-mu-R}) and $\tau> t_0$, we get
\begin{align*}
  &   \K_\tau(\mu^{\delta_+,\delta_-}_{\R})=\ind_{\Sigma^{\R}_{t_0}}  \mu^{\delta_+,\delta_-}_{\R}(dZ) \oplus_{t_0} \K_{\tau_{Z;t_0}}( \mu_{Z_{t_0}}^{\delta_+,\delta_-})\\
  & \stackrel{ {(\ref{Girsanov})}}{=} M^Z_{t_0}\cdot \nu_{\R}^{\delta_+,\delta_-^*} (dZ)  \oplus_{t_0} (M^{Z\oplus_{t_0} W}_\tau/M^{Z\oplus_{t_0} W}_{t_0}) \cdot \K_{\tau_{Z;t_0}}(\nu^{\delta_+,\delta_-^*}_{Z_{t_0}}) (dW)\\
  &\stackrel{\phantom{(\ref{Markov-mu-R})}}{=} M^{Z\oplus_{t_0} W}_\tau\cdot \K_\tau( \nu_{\R}^{\delta_+,\delta_-^*} (dZ)  \oplus_{t_0} \nu^{\delta_+,\delta_-^*}_{Z_{t_0}}(dW))
   \stackrel{ {(\ref{Markov-mu-R})}}{=} M^X_\tau\cdot \K_\tau( \nu_{\R}^{\delta_+,\delta_-^*} (dX)),
\end{align*}
where in the last step we use $X$ again to replace $Z\oplus_{t_0}W$. So we get (ii) for such $\tau$

Now we do not assume that $\tau$ is uniformly bounded below. For $t_0\in\R$, let $E_{t_0}=\{\tau>t_0\}$. Since  (ii) holds for $\tau\vee t_0$, $d\mu^{\delta_+,\delta_-}_{\R}/d\nu_{\R}^{\delta_+,\delta_-^*}=M^\cdot_{\tau\vee t_0} $ on $\F_{\tau\vee t_0}\cap \Sigma^{\R}_{\tau\vee t_0}$. Since $\F_{\tau\vee t_0}\cap E_{t_0}= \F_\tau\cap E_{t_0}$ and $E_{t_0}\subset\{\tau\vee t_0=\tau\}$, we get $d\mu^{\delta_+,\delta_-}_{\R}/d \nu_{\R}^{\delta_+,\delta_-^*}=M^\cdot_\tau$ on $\F_{\tau}\cap \Sigma^{\R}_\tau\cap E_{t_0}$. Since $\Sigma_{\R}=\bigcup_{n\in\N} E_{-n}$, we see that  (ii) holds for any $\R^{\R}$-stopping time $\tau$.

Since (i) holds for $\tau\vee t_0$, we get
$${\ind}_{\Sigma^{\R}_{\tau\vee t_0}}  \mu^{\delta_+,\delta_-}_{\R}(dZ) \oplus_{\tau\vee t_0} \mu_{Z_{\tau\vee t_0}}^{\delta_+,\delta_-}={\ind}_{\Sigma^{\R}_{\tau\vee t_0}} \mu^{\delta_+,\delta_-}_{\R}.$$
Since $E_{t_0}\cap \Sigma^{\R}_{\tau}\subset\{\tau=\tau\vee t_0<\ha T\}\subset \Sigma^{\R}_{\tau\vee t_0}$, the above formula implies
$$\ind_{E_{t_0}\cap \Sigma^{\R}_{\tau}}\cdot  \mu^{\delta_+,\delta_-}_{\R}(dZ) \oplus_{\tau } \mu_{Z_{\tau}}^{\delta_+,\delta_-}= \ind_{E_{t_0}\cap\Sigma^{\R}_{\tau}}\cdot  \mu^{\delta_+,\delta_-}_{\R}.$$
Since $\Sigma_{\R}=\bigcup_{n\in\N} E_{-n}$, we see that  (i) holds for any $\R^{\R}$-stopping time $\tau$.

Suppose $\bar \mu$ satisfies (i) and (iii) for deterministic times. Let $0\le t_1<t_2$.  By (\ref{Girsanov}),
\BGE \ind_{\{\ha T>t_1\}}\bar\mu(dZ) \oplus_{t_1} \frac{M^X_{t_2-t_1}}{M^X_0} \ind_{ \{\ha T>t_2-t_1\}} \nu_{Z_{t_1}}^{\delta_+,\delta_-^*}(dX)=\ind_{ \{\ha T>t_2\}}\bar \mu. \label{til-mu-DMP}\EDE
For each $t\in\R$, let $\bar \mu_t=\pi_t^{\R}(\ind_{\Sigma^{\R}_t} \bar\mu)$. Then $\bar\mu_t$ is a finite measure on $[-1,1]$. 
By (\ref{til-mu-DMP},\ref{MZt})  $\bar \mu_t\ll \mA$ for any $t\in\R$, and  $\bar q_t(y):=\bar\mu_t(dy)/\mA(dy)$  satisfy
$$ \bar q_{t_2}(y)=\int_{-1}^1 \bar q_{t_1}(x) \frac{e^{-\alpha_0 t_2} (1+y)^{\alpha_-}} {e^{-\alpha_0 t_1} (1+x)^{\alpha_-}}  p_{t_2-t_1}^*(x,y) \mA(dx),\quad y\in(-1,1),\quad t_1<t_2.$$
Let $ \bar p_t(x)=e^{\alpha_0 t} (1+x)^{-\alpha_-} \bar q_t(x)$. By the above formula we have
\BGE \bar p_{t_2}(y)=\int_{-1}^1  \bar p_{t_1}(x) p_{t_2-t_1}^*(x,y) \mA(dx),\quad y\in(-1,1),\quad t_1<t_2.\label{hap=}\EDE
Since $\int_{-1}^1 p_t^*(x,y)dy=1$, by Fubini Theorem and that $\alpha_-<0$, there is a constant $\bar C\in[0,\infty)$ such that $\int_{-1}^1 \bar p_t(x)\mA(dx)=\bar C$ for all $t\ge 0$.
From (\ref{hap=},\ref{transition-rate-i}), there are constants $C,L\in(0,\infty)$ such that if $t_2-t_1>L$, then
$$|\bar p_{t_2}(y)- \bar C p_\infty^*(y)|\le \int_{-1}^1  \bar p_{t_1}(x)| p_{t_2-t_1}^*(x,y)-p_\infty^*(y)|\mA(dx)\le   \bar C C p_\infty^*(y) e^{-\frac{\delta_++\delta_-^*}4 (t_2-t_1)}.$$
Fixing $t_2$ and letting $t_1\to -\infty$, we find that $\bar p_{t_2}=  \bar C p_\infty^*$. Thus, $\bar q_{t_2}(x)= \bar C e^{-\alpha_0 t_2} (1+x)^{\alpha_-} p_\infty^*(x)$. The same argument applied to $\mu^{\delta_+,\delta_-}_{\R}$ shows that for each $t\in\R$, $\mu_t:= \pi_t^{\R}(\ind_{\Sigma^{\R}_t} \mu^{\delta_+,\delta_-}_{\R})$ is absolutely continuous w.r.t.\ $\mA$, and its density function, denoted by $q_t$, equals $C_0e^{-\alpha_0 t} (1+x)^{\alpha_-} p_\infty^*(x)$ for some $C_0\in (0,\infty)$.  Thus, $\bar \mu_t=\ha C \mu_t$ for $t\in\R$, where $\ha C:=   \bar C/C_0$. Since  $\bar \mu$ and $\mu^{\delta_+,\delta_-}_{\R}$ both satisfy (i) for $\tau=t$, we get $\pi_{[t,\infty)}^{\R} ( \bar \mu_t)=\ha C\pi_{[t,\infty)}^{\R} (\mu_t)$, where  $\pi_{[t,\infty)}^{\R}$ is the map $Z\mapsto Z|_{[t,\infty)}$. Since this holds for all $t\in\R$,   we get $ \bar \mu=\ha C\mu^{\delta_+,\delta_-}_{\R}$.
\end{proof}

We are interested in the conditional probability measures:  $\mu^{\delta_+,\delta_-}_x[\cdot|\Sigma_t] $ and $\mu^{\delta_+,\delta_-}_{\R}[\cdot|\Sigma^{\R}_t]$.

\begin{Lemma}
Let $x\in[-1,1]$, $t_1>0$, $t_2\in\R$, and $s>0$. Then $\pi_{t_1}(\mu^{\delta_+,\delta_-}_x[\cdot|\Sigma_{t_1+s}])$ and $\pi_{t_2}^{\R}(\mu^{\delta_+,\delta_-}_{\R}[\cdot|\Sigma^{\R}_{t_2+s}])$ are both absolutely continuous w.r.t.\ $\mA$, and the densities are respectively $\ind_{[-1,1]}p_{t_1}^s(x,\cdot)$ and $\ind_{[-1,1]}p_{\infty}^s$, which are given by
\BGE p_t^s(x,y)=\frac{p_t^*(x,y) w_s(y)}{\int_{-1}^1 p_t^*(x,z) w_s(z)\mA(dz)},\quad p_\infty^s(y) =\frac{p_\infty^*(y) w_s(y)}{\int_{-1}^1 p_\infty^*(z) w_s(z)\mA(dz)},\label{pinfsxy}\EDE
\BGE w_s(y)=\int_{-1}^1\Big(\frac{1+z}2\Big)^{\alpha_-} p_s^*(y,z) \mA(dz),\label{wtx}\EDE
\end{Lemma}
\begin{proof}
Let $U\in{\cal B}([-1,1])$.
By Lemma \ref{Girsanov-lem} and that  $p_t^*(x,y)$ is the transition density for $\nu^{\delta_+,\delta_-^*}_\cdot$,
\begin{align*}
  &  \mu^{\delta_+,\delta_-}_x (\pi_{t_1}^{-1}(U)\cap \Sigma_{t_1+s})=\mu^{\delta_+,\delta_-}_x \circ \pi_{t_1,t_1+s}^{-1}(U\times [-1,1])\\
 =&e^{-\alpha_0 (t_1+s)}\int_{U}  \int_{-1}^1\Big(\frac{1+z}{1+x}\Big)^{\alpha_-} p_{t_1}^*(x,y)  p_{s}^*(y,z)  \mA(dz)\mA(dy)\\
= &{e^{-\alpha_0 (t_1+s)}}\Big(\frac{1+x}2\Big)^{-\alpha_-}\int_U p^*_{t_1}(x,y) w_{s}(y) \mA(dy).
\end{align*}
Replacing $U$ by $[-1,1]$, we get from the above formula
 $$\mu^{\delta_+,\delta_-}_{x}( \Sigma_{t_1+s})= {e^{-\alpha_0 (t_1+s)}}\Big(\frac{1+x}2\Big)^{-\alpha_-}\int_{-1}^1 p^*_{t_1}(x,y) w_{s}(y) \mA(dy).$$
Combining the above two displayed formulas, we conclude that $\ind_{[-1,1]}p_{t_1}^s(x,\cdot)$ is the density of $\pi_{t_1}(\mu^{\delta_+,\delta_-}_x[\cdot|\Sigma_{t_1+s}])$.

A similar argument involving Lemma \ref{stationary-quasi} gives
$$\mu^{\delta_+,\delta_-}_{\R} (\pi_{t_2}^{-1}(U)\cap \Sigma_{t_2+s})=\mu^{\delta_+,\delta_-}_{\R} \circ \pi_{t_2,t_2+s}^{-1}(U\times [-1,1])$$
$$=\int_U \int_{-1}^1e^{-\alpha_0( t_2+s)} \Big(\frac{1+z}2\Big)^{\alpha_-} p_\infty^*(y)    p_{s}^*(y,z)\mA(dz) \mA(dy)=e^{-\alpha_0 (t_2+s)}\int_U p_\infty^*(y) w_{s}(y) \mA(dy).$$
Taking $U=[-1,1]$, we get
$\mu^{\delta_+,\delta_-}_{\R}( \Sigma_{t_2+s})= e^{-\alpha_0 (t_2+s)}\int_{-1}^1 p_\infty^*(y) w_{s}(y) \mA(dy)$, which together with the above formula implies  that $\ind_{[-1,1]}p_{\infty}^s$ is the density of $\pi_{t_2}(\mu^{\delta_+,\delta_-}_{\R}[\cdot|\Sigma_{t_2+s}])$.
\end{proof}

\begin{Lemma}
There are constants $C,L_0\in (0,\infty)$ such that for any $x\in[-1,1]$, $t_1>L_0$, $t_2\in\R$, and $t_0>0$,
there is a coupling $\ha Z^j$, $j=1,2$, of the probability measures $\mu_1:= \mu^{\delta_+,\delta_-}_x[\cdot|\Sigma_{t_1+t_0}]$ and $ \mu_2:= \mu^{\delta_+,\delta_-}_{\R}[\cdot|\Sigma^{\R}_{t_2+t_0}]$ such that $\PP[(\ha Z^1_{t_1+t})_{t\ge 0}\ne (\ha Z^2_{t_2+t})_{t\ge 0}]< 8C e^{-\frac{\delta_++\delta_-^*}4 t_1}$.
\label{coupling}
\end{Lemma}
\begin{proof}
  By (\ref{transition-rate-i}), there are constants $C,L_0\in(0,\infty)$ such that if $t_1>L_0$, then for any $x,y\in[-1,1]$,
$|\frac{p_{t_1}^*(x,y)}{p_\infty^*(y)}-1|< C e^{-\frac{\delta_++\delta_-^*}4 t_1}<\frac 18$,
which implies that
$$\Big|\log\Big( \frac{p_{t_1}^*(x,y)w_{t_0}(y)}{p_\infty^*(y)w_{t_0}(y)}\Big)\Big|< 2C e^{-\frac{\delta_++\delta_-^*}4 t_1} .$$
From this we further have
$$\Big|\log\Big( \frac{\int_{-1}^1p_{t_1}^*(x,z)w_{t_0}(z)\mA(dz)}{\int_{-1}^1 p_\infty^*(z)w_{t_0}(z)\mA(dz)}\Big)\Big|< 2C e^{-\frac{\delta_++\delta_-^*}4 t_1} .$$
Let $p^{t_0}_{t_1}(x,y),p^{t_0}_\infty(y)$ be defined by (\ref{pinfsxy}).
By the above two displayed formulas,   if $t_1>L_0$, then for any $x,y\in [-1,1]$,
$|\log( \frac{p_{t_1}^{t_0}(x,y)}{p_\infty^{t_0}(y)})|< 4C e^{-\frac{\delta_++\delta_-^*}4 t_1}<\frac 12$,
which implies that
$$ |p_{t_1}^{t_0}(x,y)-p_\infty^{t_0}(y)|< 8C e^{-\frac{\delta_++\delta_-^*}4 t_1} p_\infty^{t_0}(y),\quad \mbox{if }t_1>L_0.$$ 
By the previous lemma,
we get a coupling $X_1$ and $X_2$ of $\pi_{t_1}(\mu_1)$ and $\pi_{t_2}^{\R}(\mu_2)$ such that $\PP[X_1\ne X_2]<8C e^{-\frac{\delta_++\delta_-^*}4 t_1}$, which together with the Markov properties of $\mu_1$ and $\mu_2$ implies the existence of the coupling   satisfying the desired inequality.
\end{proof}

\section{Rooted SLE$_\kappa(\rho)$ Bubbles} \label{Section-rooted-loop}
Let $\kappa>0$ and $\rho>-2$.
We will construct rooted SLE$_\kappa(\rho)$ bubble measures. In the case that $\rho\le \frac\kappa 2-4$ (and so $\kappa>4$), the bubble measure is a probability measure. In the case that $\rho>(-2)\vee (\frac\kappa 2-4)$, the bubble measure is a $\sigma$-finite infinite measure. 

We use $\til\mu^{D}_{(w;v)\to x}$ to denote the law of a single-force-point (MTC) SLE$_\kappa(\rho)$ at $D:(w;v)\to x$.  The DMP of SLE$_\kappa(\rho)$ could be expressed symbolically  by the following formula: for any TCI stopping time $\tau$,
\BGE \ind_{\til\Sigma_\tau} \til \mu^D_{(w;v)\to x}(d[\eta]) \oplus_\tau \til \mu^{D([\eta]|\tau^+)}_{([\eta](\tau^+); [v]^{[\eta]}_{\tau^+;D})\to x}
= \ind_{\til\Sigma_\tau} \til \mu^D_{(w;v)\to x}.\label{DMP-kappa-rho}\EDE

\subsection{The construction}

Suppose $\gamma$ is an AP Loewner curve in $D$, and $s_0\in[0,\ha T_\gamma)$.
Recall that $[\gamma(0)^+]^{\gamma}_{s_0^+;D}$ equals $\ha b^{\gamma}_{s_0^+;D}$ or $(\ha b^{\gamma}_{s_0^+;D})^+$ depending on whether $\gamma(s_0^+)\ne\ha b^{\gamma}_{s_0^+;D}$; and  $[\gamma(0)^-]^{\gamma}_{s_0^+;D}$ equals $\ha a^{\gamma}_{s_0^+;D}$ or $(\ha a^{\gamma}_{s_0^+;D})^-$ depending on whether $\gamma(s_0^+)\ne\ha a^{\gamma}_{s_0^+;D}$.
When $\gamma(s_0^+)\ne \ha a^{\gamma}_{s_0^+;D}$, and so  $[\gamma(0)^-]^{\gamma}_{s_0^+;D}=\ha a^{\gamma}_{s_0^+;D}$, we define
\BGE \til\mu^D_{\gamma|s_0^+}=\til \mu^{D(\gamma|s_0^+)}_{(\gamma(s_0^+);[\gamma(0)^+]^{\gamma}_{s_0^+;D})\to [\gamma(0)^-]^{\gamma}_{s_0^+;D}}%
. \label{til-mu-gamma} \EDE
If $\sigma$ is a TCI time on the MTC Loewner curve $[\gamma]$, which is less than $\ha T$, such that $[\gamma](\sigma^+)\ne \ha a^{[\gamma]}_{\sigma^+;D}$, we define
\BGE \til\mu^D_{[\gamma]|\sigma^+}=\til\mu^D_{\gamma|\sigma(\gamma)^+}
.\label{til-mu-gamma-[]} \EDE

\begin{Theorem}
Let $\kappa>4$ and $\rho\in (-2,\frac\kappa 2-4]$. Then there is a unique probability measure $\til\mu^{\HH}_{0\ccw} $ on the space of  Loewner bubbles in $\HH$ rooted at $0$, such that for any strictly positive TCI $\F$-stopping time $\tau$,
\BGE \ind_{\til\Sigma_\tau} \til\mu^{\HH}_{0\ccw} =\ind_{\til\Sigma_\tau} \til\mu^{\HH}_{0\ccw}(d[\gamma]) \oplus_\tau \til\mu^{\HH}_{[\gamma]|\tau^+} .\label{DMP-kappa-rho-proob}\EDE
Moreover,  $\til\mu^{\HH}_{0\ccw} $  is invariant under any conformal automorphism of $\HH$ fixing $0$,  is supported by boundary-filling curves (visiting every point on $\ha\R$), and ${\cal P}(\til\mu^{\HH}_{0\ccw})$ is the law of a CP SLE$_\kappa(\rho,\kappa-6-\rho)$ at $(0;0^+,0^-)$. Recall that $\cal P$ is defined right before (\ref{radxPeta}).
  \label{Thm-loop-probability}
\end{Theorem}

We call $\til\mu^{\HH}_{0\ccw} $ the ccw SLE$_\kappa(\rho)$ loop measure in $\HH$ rooted at $0$. Later in Theorem \ref{Thm-loop-sigma-finite} we will define the ccw SLE$_\kappa(\rho)$ loop measure $\til\mu^{\HH}_{0\ccw} $ for $\rho>\frac\kappa 2-4$. Let $\Phi(z)=-\lin z$. Then $\til\mu^{\HH}_{0\cw}:=\Phi(\til\mu^{\HH}_{0\ccw})$ is called a cw SLE$_\kappa(\rho)$ bubble measure in $\HH$ rooted at $0$. Let $A_w(z)=w+z$ for $w\in\R$. Then $\til\mu^{\HH}_{w\ccw}:=A_w(\til\mu^{\HH}_{0\ccw})$ and $\til\mu^{\HH}_{w\cw}:=A_w(\til\mu^{\HH}_{0\cw})$ are SLE$_\kappa(\rho)$ loop measures rooted at $w$.

\begin{Remark}
  In plain words, (\ref{DMP-kappa-rho-proob}) means that if a random MTC curve $[\eta]$ follows the law $\til\mu^{\HH}_{0\ccw} $, then for any positive TCI stopping time $\tau$, conditionally on $\til\F_\tau$ and the event $\til\Sigma_\tau$, the part of $[\eta]$ after $\tau$ is an SLE$_\kappa(\rho)$ curve in a connected component of $\HH\sem [\eta]([0,\tau])$ aimed at (the left of) $0$ with the force point being the ccw most point on $\ha \R$ visited by $[\eta]$ up to the time $\tau$. See Figure \ref{figure} for illustration.
\end{Remark}

\begin{figure}
	\includegraphics[width=1\textwidth]{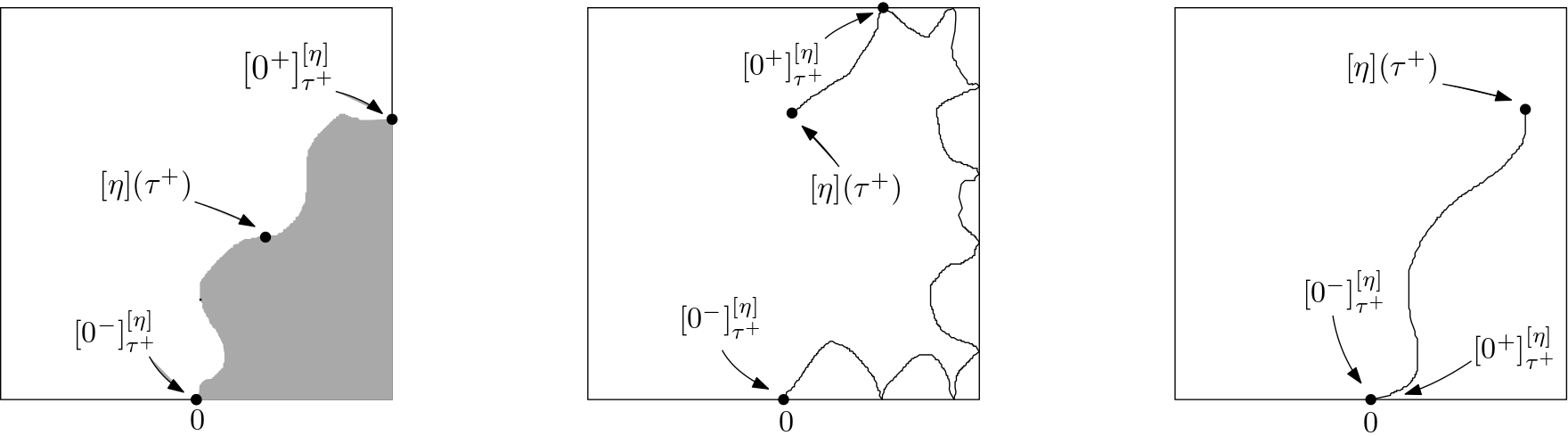}%
\caption{The three figures illustrate  ccw SLE$_\kappa(\rho)$ loops $[\eta]$ rooted at $0$  up to some TCI stopping time $\tau$. The rest of $[\eta]$ (not drawn) is an SLE$_\kappa(\rho)$ curve in the remaining domain from  $[\eta](\tau^+)$ to $[0^-]^{[\eta]}_{\tau^+}$ with the force point being $[0^+]^{[\eta]}_{\tau^+}$, where the three prime ends are labeled and marked with dots.
The $\rho$ for the $[\eta]$ on the left lies in $(-2, \frac\kappa 2-4)$ as in Theorem \ref{Thm-loop-probability}, and the $[\eta]$ is drawn as a (grey) space-filling curve ($\kappa\ge 8$). The $\rho$ for the  $[\eta]$   in the middle lies in $(\frac\kappa 2-4,\frac\kappa 2-2)$ as in Theorem \ref{Thm-loop-sigma-finite}, and $[\eta]$  intersects the domain boundary at a fractal set, which posses Minkowski content measure by Theorem \ref{Thm-decomposition-loop}. The $\rho$ for the  $[\eta]$  on the right lies in $[\frac\kappa 2-2,\infty)$ as in Theorem \ref{Thm-loop-sigma-finite}, and $[\eta]$  intersects the domain boundary only at its root $0$.
} \label{figure}
\end{figure}
\begin{proof}
Let $\rho_+=\rho$ and $\rho_-=\kappa-6-\rho$. Then $\rho_+>-2$ and $\rho_-\ge \frac\kappa 2-2$. Let $\mu$ be the law of a CP SLE$_\kappa(\rho_+,\rho_-)$  at $(0;0^+,0^-)$. Let $\til\mu=\pi(\mu)$. For any $r>0$, let $\sigma_r$ denote the hitting time at $\{|z|\ge r\}$.
Using the symbol in  (\ref{til-mu-gamma},\ref{til-mu-gamma-[]}), we define
 \BGE \til\mu_r= \mu(d\gamma) \oplus_{\sigma_r}  \til\mu^{\HH}_{\gamma|\sigma_r^+}= \til\mu(d[\gamma]) \oplus_{\sigma_r}  \til\mu^{\HH}_{[\gamma]|\sigma_r^+}.\label{DMP-kappa-rho-proob-sigma}\EDE
 Intuitively speaking, a curve with law $\til\mu_r$ is constructed by the following procedure: first run a CP SLE$_\kappa(\rho_+,\rho_-)$ curve $\gamma$ at $(0;0^+,0^-)$, then stop $\gamma$ at the first time $\sigma_r$ that it reaches $\{|z|=r\}$, and finally continue it with an SLE$_\kappa(\rho)$ curve in the remaining domain starting from the tip of $\gamma$ at $\sigma_r$, such that the target and the force point are respectively the smallest and the biggest  real number that is visited by $\gamma$ at $\sigma_r$. Since $\rho_-\ge \frac\kappa 2-2$, by the boundary behavior of a single-force-point SLE$_\kappa(\rho_-)$ curve and Girsanov theorem,   $\gamma$ does not visit $(-\infty,0)$, and so the target for the conditional SLE$_\kappa(\rho)$ curve is a prime end that determines $0$.
 By Lemma \ref{continue-Loewner}, $\til\mu_r$ is supported by Loewner bubbles in $\HH$ rooted at $0$.

 We now show that $\til\mu_r$ does not depend on $r$.
 Suppose $0<s<t$. Let $\sigma=\sigma_s$ and $\tau=\sigma_t$. Then $\sigma<\tau$.
Let $\tau_{[\gamma];\sigma}$ be as defined in Lemma \ref{tau-[g]'} with  $[g]=[\gamma]$, which is a TCI stopping time by Lemmas \ref{tau-g'} and \ref{tau-[g]'}. We will use the following facts: if $[\gamma]$ follows the law $\til\mu$ and $[\eta]$ follows the conditional law $\til\mu^{\HH}_{[\gamma]|\sigma^+}$, then
$$\HH([\gamma]|\sigma^+)([\eta]|\tau_{[\gamma];\sigma}^+)=\HH([\gamma]\oplus_\sigma[\eta]|\tau^+),\quad [\eta](\tau_{[\gamma];\sigma}^+)=([\gamma]\oplus_\sigma [\eta])(\tau^+),$$
$$[[0^+]^{[\gamma]}_{\sigma^+;\HH}]^{[\eta]}_{\tau_{[\gamma];\sigma}^+;\HH([\gamma]|\sigma^+)} =[0^+]^{[\gamma]\oplus_\sigma [\eta]}_{\tau^+;\HH},\quad [0^-]^{[\gamma]}_{\sigma^+;\HH}=[0^-]^{[\gamma]\oplus_\sigma [\eta]}_{\tau^+;\HH},$$
where the third equality follows from Corollary \ref{[[v]]-cor}, and the fourth equality means that the prime end $[0^-]^{[\gamma]}_{\sigma^+;\HH}$ of $\HH([\gamma]|\sigma^+)$ and the prime end $[0^-]^{[\gamma]\oplus_\sigma [\eta]}_{\tau^+;\HH}$ of $\HH([\gamma]\oplus_\sigma [\eta]|\tau^+)$ are identified.

By 
the DMP formula (\ref{DMP-kappa-rho}) for the SLE$_\kappa(\rho)$ law $\til\mu^{\HH}_{[\gamma]|\sigma^+}$ and the above displayed formulas, we get
\BGE \ind_{\til\Sigma_{\tau_{[\gamma];\sigma} }} \til\mu^{\HH}_{[\gamma]|\sigma^+}(d[\eta])\oplus_{\tau_{[\gamma];\sigma}} \til\mu^{\HH}_{[\gamma]\oplus_\sigma [\eta]|\tau^+}= \ind_{\til\Sigma_{\tau_{[\gamma];\sigma} }}  \til\mu^{\HH}_{[\gamma]|\sigma^+}.
\label{compos}
\EDE
Since $\rho\le \frac\kappa 2-4$, by the boundary behavior of SLE$_\kappa(\rho)$, a curve following the law $\til\mu^{\HH}_{[\gamma]|\sigma^+}$ visits every point on the prime end interval $([0^+]^{[\gamma]}_{\sigma^+;\HH},[0^-]^{[\gamma]}_{\sigma^+;\HH})_{\ha\pa \HH([\gamma]|\sigma^+)}$, which includes the point $t$. Thus, $\til\mu^{\HH}_{[\gamma]|\sigma^+}$ is supported by $\til\Sigma_{\sigma_t}$. Since $\sigma(\gamma)<\tau(\gamma)$ and $\til\mu^{\HH}_{[\gamma]|\sigma^+}$ is supported by curves started from $\gamma(\sigma)$, on which $\tau_{[\gamma];\sigma}$ agrees with $\tau=\sigma_t$, we see that  $\til\mu^{\HH}_{[\gamma]|\sigma^+}$ is supported by $\til\Sigma_{\tau_{[\gamma];\sigma^+}}$. Thus, (\ref{compos}) simplifies to
\BGE  \til\mu^{\HH}_{[\gamma]|\sigma^+}(d[\eta])\oplus_{\tau_{[\gamma];\sigma^+}} \til\mu^{\HH}_{[\gamma]\oplus_\sigma [\eta]|\tau^+}=  \til\mu^{\HH}_{[\gamma]|\sigma^+}.
\label{compos2}
\EDE

 By the DMP for SLE$_\kappa(\rho_+,\rho_-)$, if $[\gamma]$ follows the law $\til\mu$, then conditionally on $\F_{\sigma}$, the part of $[\gamma]$ after $\sigma$   is an SLE$_\kappa(\rho_+,\rho_-)$ at
$\HH([\gamma]|\sigma^+): ([\gamma](\sigma^+);[0^+]^{[\gamma]}_{\sigma^+;\HH},[0^-]^{[\gamma]}_{\sigma^+;\HH})\to \infty$.
Since $\rho_-=\kappa-6-\rho_+$, and $[\gamma]([0,\tau])$ does not separate $[0^-]^{[\gamma]}_{\sigma^+;\HH}$ from $\infty$, by   \cite{SW} (SLE coordinate changes), the part of $[\gamma]$ from $\sigma$ up to $\tau$ has the conditional law $\K_\tau(\til\mu^{\HH}_{[\gamma]|\sigma^+})$.
 Thus,
 \begin{align}
 \til\mu_t&\stackrel{\phantom{(\ref{Markov-mu-R})}}{=}(\til\mu(d[\gamma]) \oplus_{\sigma}   \til\mu^{\HH}_{[\gamma]|\sigma^+}(d[\eta]))\oplus_{\tau}  \til\mu^{\HH}_{[\gamma]\oplus_\sigma [\eta]|\tau^+}\nonumber \\
 &\stackrel{{(\ref{fgh})}}{=} \til\mu (d[\gamma]) \oplus_{\sigma}(\til\mu^{\HH}_{[\gamma]|\sigma^+}(d[\eta])\oplus_{\tau_{[\gamma];\sigma}}\til\mu^{\HH}_{[\gamma]\oplus_\sigma [\eta]|\tau^+})\nonumber\\
  &\stackrel{{(\ref{compos2})}}{=} \til\mu (d[\gamma]) \oplus_{\sigma} \til\mu^{\HH}_{[\gamma]|\sigma^+}=\til\mu_s.
 \label{muts}
 \end{align}
 Hence $\til\mu_r$ does not depend on $r$, and we may and will write $\til\mu^{\HH}_{0\ccw}$ for $\til\mu_r$. The boundary-filling property of $\til\mu^{\HH}_{0\ccw}$ follows from its DMP and the boundary behavior of SLE$_\kappa(\rho)$ for   $\rho\le \frac\kappa 2-4$. For any $r>0$, since $\til\mu^{\HH}_{0\ccw}=\til\mu_r$, by the construction of $\til\mu_r$, we have $\K_{\sigma_r}({\cal P}(\til\mu^{\HH}_{0\ccw}))=\K_{\sigma_r}(\mu)$. Since this holds for any $r>0$, we get ${\cal P}(\til\mu^{\HH}_{0\ccw})=\mu$.

Let $\tau$ be a strictly positive TCI stopping time. First suppose $\tau\ge \sigma_r$ for some $r>0$.  Applying (\ref{compos}) with $\sigma:=\sigma_r$, using  the definition (\ref{DMP-kappa-rho-proob-sigma}) of $\til\mu_r$, which equals $\til\mu^{\HH}_{0\ccw}$,  we get
\begin{align}
  &\ind_{\til\Sigma_\tau} \til\mu^{\HH}_{0\ccw}\stackrel{ {(\ref{DMP-kappa-rho-proob-sigma})}}{=}\til\mu(d[\gamma]) \oplus_\sigma\ind_{\til\Sigma_{\tau_{[\gamma];\sigma}}} \til\mu^{\HH}_{[\gamma]|\sigma^+}\nonumber\\
   &\stackrel{ {(\ref{compos})}}{=}\til\mu(d[\gamma]) \oplus_\sigma (\ind_{\til\Sigma_{\tau_{[\gamma];\sigma}}} \til\mu^{\HH}_{[\gamma]|\sigma^+} (d[\eta]) \oplus_{\tau_{[\gamma];\sigma}} \til\mu^{\HH}_{[\gamma]\oplus_\sigma[\eta]|\tau^+})\nonumber \\
   &\stackrel{{(\ref{fgh})}}{=}\ind_{[\gamma]\oplus_\sigma[\eta]\in \til\Sigma_\tau} (\til\mu(d[\gamma])  \oplus_\sigma   \til\mu^{\HH}_{[\gamma]|\sigma^+} (d[\eta]) )\oplus_\tau \til\mu^{\HH}_{[\gamma]\oplus_\sigma[\eta]|\tau^+}\nonumber \\
  &\stackrel{ {(\ref{DMP-kappa-rho-proob-sigma})}}{=}\ind_{\til\Sigma_\tau}\til\mu(d[\zeta]) \oplus_\tau \til\mu^{\HH}_{ [\zeta]|\tau^+},
   \label{tau>sigmar}
\end{align}
where in the last step we wrote $[\zeta]$ for $[\gamma]\oplus_\sigma[\eta]$. Thus, we get (\ref{DMP-kappa-rho-proob}) in the case that $\tau\ge \sigma_r$ for some $r>0$. This means that, for each $n\in\N$, (\ref{DMP-kappa-rho-proob}) holds for $\tau\vee \sigma_{1/n}$. Let $E_n=\{\sigma_{1/n}<\tau\}$.  Since  $\tau\vee \sigma_{1/n}=\tau$ on $E_n$, we see that (\ref{DMP-kappa-rho-proob}) holds with an additional factor $\ind_{E_n}$ on both sides. Since $\tau$ is strictly  positive, $\bigcup_n E_n$ supports $\til\mu^{\HH}_{0\ccw}$. Thus, we get  (\ref{DMP-kappa-rho-proob}).

Suppose $\til \mu'$ satisfies all properties of $\til\mu^{\HH}_{0\ccw}$.  For any $r\in (0,\infty)$, taking $\tau=\sigma_r$ in (\ref{DMP-kappa-rho-proob}) and using the boundary-filling properties of SLE$_\kappa(\rho)$ for $\rho\le \frac\kappa 2-4$, we see that $\til \mu' (\til\Sigma_{\sigma_s}\sem \til\Sigma_{\sigma_r})=0$ for any $s>r>0$, which together with the fact that $\til\mu'$ is   supported by $\bigcup_{n=1}^\infty \til\Sigma_{\sigma_{1/n}}$ implies that it is supported by $\til\Sigma_{\sigma_r }$ for any $r\in(0,\infty)$. Since $\til\mu'$ and $\til\mu^{\HH}_{0\ccw}$ are both supported by $\bigcup_{n=1}^\infty \til\Sigma_{\sigma_{1/n}}$, and satisfy (\ref{DMP-kappa-rho-proob}),  in order to prove that they are equal,   it suffices to show that ${\cal P}(\til\mu')=\mu$.

Let $[\eta]$ follow the law $\til\mu'$, and $\eta_0={\cal P}([\eta])$.
Fix a deterministic time $t_0\in (0,\infty)$. Applying (\ref{DMP-kappa-rho-proob}) and using the coordinate change results in \cite{SW}, we find that, conditionally on $\F_{t_0}$ and the event $\Sigma_{t_0}$, the part of $\eta_0$ after $t_0$ is the $(g^{\eta_0}_{t_0})^{-1}$-image of a CP SLE$_\kappa(\rho_+,\rho_-)$ at
 $(W^{\eta_0}_{t_0};D^{\eta_0}_{t_0},C^{\eta_0}_{t_0})$.   Since this holds for any $t_0>0$, and $W^{\eta_0},D^{\eta_0},C^{\eta_0}$ all start from $0$, $\eta_0$ is a CP SLE$_\kappa(\rho_+,\rho_-)$ at $(0;0^+,0^-)$. So we get ${\cal P}(\til \mu')=\mu$, as desired. 

Suppose $f$ is a conformal automorphism of $\HH$ fixing $0$. Then $f(\til\mu^{\HH}_{0\ccw})$ satisfies the same properties as $\til\mu^{\HH}_{0\ccw}$. By the uniqueness of $\til\mu^{\HH}_{0\ccw}$, we get $f(\til\mu^{\HH}_{0\ccw})=\til\mu^{\HH}_{0\ccw}$.
\end{proof}

 Some argument in the above proof such as those in (\ref{muts}) and (\ref{tau>sigmar}) will be used again later in the proof of Theorem \ref{Thm-loop-sigma-finite}.

To deal with the case $\rho\in(-2,\frac\kappa 2-4)$, we need the following lemmas, whose proofs are straightforward, and so will be omitted.

\begin{Lemma}
  Let $\kappa>0$, $\rho_+,\rho_->-2$. Let $x\le w\le v$ and $x<v$. Let $z=\frac{2w-v-x}{v-x}\in [-1,1]$. Let
$ \delta_\pm=\frac 4\kappa(\rho_\pm+2)>0$. Then we have the following correspondence.
  \begin{enumerate}
    \item [(i)]   Let $W$ be the driving function of a CP SLE$_\kappa(\rho_+,\rho_-)$ at $(w;v,x)$, and let $V_t$ and $X_t$ be the corresponding force point processes. When $v=w$ (resp.\ $x=w$), the force point $v$ (resp.\ $x$) is understood as $w^+$ (resp.\ $w^-$).   Let
      $$  Z_t=\frac{2W_t-V_t-X_t}{V_t-X_t},\quad  p(t)=\frac\kappa 2 \log(\frac{V_t-X_t}{v-x}), \quad \ha Z_s=Z_{p^{-1}(s)}.$$
  Then $(\ha Z_s)$ follows the law $\nu^{\delta_+,\delta_-}_z$ (defined after (\ref{Bessel-SDE})).
  \item [(ii)] Let $(\ha Z_s)$ follow the law $\nu^{\delta_+,\delta_-}_z$. Define
    \begin{align*}
    \ha V_{s}&=v+(v-x)\int_{0}^s \frac 1\kappa (1+\ha Z_r) e^{\frac 2\kappa r} dr,\quad
        \ha X_{s} =x-(v-x)\int_{0}^s \frac 1\kappa (1-\ha Z_r)  e^{\frac 2\kappa r} dr, \\
    \ha W_{s}&= \frac{1+\ha Z_s}2 \ha V_{s}+\frac{1-\ha Z_s}2 \ha X_s,\quad
    q(s)=(v-x)^2\int_{0}^s \frac 1{4\kappa} (1-(\ha Z_r)^2) e^{\frac 4\kappa r} dr.
  \end{align*}
  Then $W_t:=\ha W_{q^{-1}(t)}$ is the driving function of a CP SLE$_\kappa(\rho_+,\rho_-)$ at $(w;v,x)$, and the corresponding force point processes are  $V_t:= \ha V_{q^{-1}(t)}$ and $X_t:= \ha X_{q^{-1}(t)}$. Moreover, $\ha V_s-\ha X_s=(v-x)e^{\frac 2\kappa s}$ for all $s\in [0, \ha T_{\ha Z})$, and so the $p(t)$ in (i) is the inverse of the $q(s)$ in (ii).
  \end{enumerate}
\label{correspondance-mu}
\end{Lemma}

\begin{Lemma}
Lemma \ref{correspondance-mu} still holds if we make the following  changes simultaneously:
  \begin{itemize}
    \item We still assume $\rho_+>-2$, but  assume $\rho_-<\frac\kappa 2 -2$ instead of $\rho_->-2$.
    \item Instead of $x\le w$, we require that $x<w$, which ensures that $z\in (-1,1]$.
    \item We kill the processes $W,V,X$  at the time $\tau_x$, i.e., the first $t$ such that $X_t=W_t$.
    \item The $(\ha Z_s)$ follows the law  $\mu^{\delta_+,\delta_-}_z$ (defined right before Lemma \ref{Girsanov-lem}) instead of  $\nu^{\delta_+,\delta_-}_z$.
  \end{itemize}
  \label{correspondance-nu}
\end{Lemma}

Using Lemma \ref{correspondance-mu} and the Markov property (\ref{Markov-mu-R}) of stationary radial Bessel processes, we obtain the following lemma.

\begin{Lemma}
  Let $\kappa>0$, $\rho_+,\rho_->-2$. Let
$ \delta_\pm=\frac 4\kappa(\rho_\pm+2)>0$. Then we have the following correspondence.
  \begin{enumerate}
    \item [(i)]   Let $W$ be the driving function of a CP SLE$_\kappa(\rho_+,\rho_-)$ at $(0;0^+,0^-)$. Let $V$ and $X$ be respectively the force point processes started from $0^+$ and $0^-$. Let
   $$Z_t=\frac{2W_t-V_t-X_t}{V_t-X_t},\quad  p(t)=\frac\kappa 2 \log( {V_t-X_t} ), \quad  \ha Z_s=Z_{p^{-1}(s)}.$$
  Then $(\ha Z_s)_{s\in\R}$   follows the law $\nu_{\R}^{\delta_+,\delta_-}$ (defined right before (\ref{sigma-R})).
  \item [(ii)] Let  $(\ha Z_s)_{s\in\R}$ follow the law $\nu_{\R}^{\delta_+,\delta_-}$. Define
    \begin{align}
    \ha V_{s}&=\int_{-\infty}^s \frac 1\kappa (1+\ha Z_r) e^{\frac 2\kappa r} dr,\quad
        \ha X_{s} =-\int_{-\infty}^s \frac 1\kappa (1-\ha Z_r) e^{\frac 2\kappa r} dr,\nonumber \\
    \ha W_{s}&= \frac{1+\ha Z_s}2 \ha V_{s}+\frac{1-\ha Z_s}2 \ha X_s,\quad
    q(t)=\int_{-\infty}^s \frac 1{4\kappa} (1-(\ha Z_r)^2) e^{\frac 4\kappa r} dr.\nonumber  
  \end{align}
  Then $W_t:=\ha W_{q^{-1}(t)}$ is the driving function of a CP SLE$_\kappa(\rho_+,\rho_-)$ at $(0;0^+,0^-)$, and the corresponding force point processes  are $V_t:=\ha V_{q^{-1}(t)}$ and $X_t=\ha X_{q^{-1}(t)}$. Moreover, $\ha V_s-\ha X_s=e^{\frac 2\kappa s}$ for all $s\in(-\infty,\ha T_{\ha Z})$, and so the $p $ in (i) is the inverse of the $q $ in (ii).
  \end{enumerate}
\label{correspondance-mu-infty}
\end{Lemma}

Combining the above three lemmas with  Lemma \ref{stationary-quasi},
we get the following lemma easily.

\begin{Lemma}
   Let $\kappa>0$, $\rho_+>-2$, $\rho_-<\frac\kappa 2-2$, and $\rho_-^*= \kappa -4-\rho_->\frac \kappa 2-2$.  For $W\in\Sigma$   and $t\in (0,\ha T_W)$ such that $C^W_t<W_t$ we write  $\mu^W_t$ for  the law of the driving function of a CP  SLE$_\kappa(\rho_+,\rho_-)$ at $(W_t;D^W_t,C^W_t)$, which is killed at $\tau_{C^W_t}$, where in the case $D^W_t=W_t$,  the force point $D^W_t$ is understood as $W_t^+$. Let $\nu_0^*$ be the law of the driving function of a CP  SLE$_\kappa(\rho_+,\rho_-^*)$ at $(0;0^+,0^-)$.  Then there is a nonzero $\sigma$-finite measure $\mu_0$ on $\Sigma$ such that the following hold.
    \begin{itemize}
           \item [(i)] For any strictly positive stopping time $\tau$,
         $${\ind}_{\Sigma_\tau} \mu_0={\ind}_{\Sigma_\tau} \mu_0(dW) \oplus_\tau \mu^W_\tau.$$
         \item [(ii)] For any strictly  positive stopping time $\tau$,  $\mu_0\ll \nu_0^*$ on $\F_\tau\cap\Sigma_\tau$ with the RN derivative being
         \BGE Q^W_t:=|D^W_t-C^W_t|^{\frac{\rho_+(\rho_--\rho^*_-)}{2\kappa}} |W_t-C^W_t|^{\frac{2(\rho_--\rho^*_-)}{2\kappa}} .\label{QWt}\EDE
         \item [(iii)] For any  deterministic $t_0>0$, $\mu_0(\{W:\sup_{0\le t<\ha T_W} |D^W_t-C^W_t|>t_0\})<\infty$.
         \end{itemize}
  Moreover, if another measure $\mu'_0$ satisfies (iii) and (i) for $\tau$ being a deterministic time, then it equals some constant times $\mu_0$.
  \label{stationary-quasi-rho}
\end{Lemma}
\begin{proof}[Sketch of Proof]
Let $\delta_\pm=\frac 4\kappa(\rho_\pm+2)$. Let $\mu^{\delta_+,\delta_-}_{\R}$ be as in Lemma \ref{stationary-quasi}. Note that, if we define $q,W,V,X$ from $\ha Z$ using formulas in Lemma \ref{correspondance-mu-infty} (ii), then $V=D^W$, $X=C^W$, and so it is straightforward to check that $ Q^W_t=M^{\ha Z}_{q^{-1}(t)}$, where  $M^{\ha Z}$ was defined in (\ref{MZt}). Thus, the pushforward of $\mu^{\delta_+,\delta_-}_{\R}$ under the map $\ha Z\mapsto W$ is the $\mu_0$ that we need.
\end{proof}

For $\kappa>0$ and $\rho>(\frac\kappa 2-4)\vee (-2)$, let $\mu_0$ and $\nu_0^*$ be as in Lemma \ref{stationary-quasi-rho} for $\rho_+=\rho$, $\rho_-=\kappa-6-\rho$ and $\rho_-^*=\rho+2$. Let $\mu=\Lo(\mu_0)$ and $\nu^*=\Lo(\nu_0)$.  Note that $\nu^*$ is the law of a CP SLE$_\kappa(\rho,\rho+2)$  at $(0;0^+,0^-)$. By Lemma \ref{stationary-quasi-rho}, for any positive $\F$-stopping time $\tau$, $\mu\ll\nu^*$ on $\F_\tau\cap\Sigma_\tau$ with the RN derivative being
        \BGE Q^\eta_{t}:=|D^\eta_{t}-C^\eta_{t}|^{\frac{\rho(\kappa-8-2\rho)}{2\kappa}} |W^\eta_{t}-C^\eta_{t}|^{\frac{2(\kappa-8-2\rho)}{2\kappa}}.\label{Q-eta-t}\EDE

\begin{Theorem}
Let $\kappa>0$ and $\rho>(\frac\kappa 2-4)\vee (-2)$.
Then there is a non-zero $\sigma$-finite measure $\til\mu^{\HH}_{0\ccw} $ on the space of  Loewner bubbles in $\HH$ rooted at $0$, such that the following holds.
\begin{enumerate}
  \item [(i)] For any strictly positive TCI $\F$-stopping time $\tau$, (\ref{DMP-kappa-rho-proob}) holds.
  \item [(ii)]    ${\cal P}(\til\mu^{\HH}_{0\ccw}) $ is the law of a CP SLE$_\kappa(\rho,\kappa-6-\rho)$ at $(0;0^+,0^-)$.
  \item [(iii)] For any $r>0$, the restriction of $ \til\mu^{\HH}_{0\ccw}$ to  $\{[\eta]:\rad_0([\eta])>r\}$  is a finite measure.
\end{enumerate}
Moreover, $\til\mu^{\HH}_{0\ccw}$ satisfies conformal covariance with exponent \BGE \alpha=\alpha(\kappa,\rho):=\frac{(\rho+2)(2\rho+8-\kappa)}{2\kappa}>0,\label{alpha}\EDE i.e.,  if $f$ is a conformal automorphism of $\HH$ that fixes $0$, then $f(\til\mu^{\HH}_{0\ccw})=f'(0)^\alpha \til\mu^{\HH}_{0\ccw}$.
Any measure on Loewner bubbles in $\HH$ rooted at $0$ satisfying (i) and (iii) equals  some constant times $\til\mu^{\HH}_{0\ccw}$. If $\rho\ge \frac\kappa 2-2$, $\til\mu^{\HH}_{0\ccw}$ is supported by loops that intersect $\ha\R$ only at $0$; and if $\rho<\frac\kappa 2-2$, $\til\mu^{\HH}_{0\ccw}$ is supported by loops whose intersection with $\ha\R$ is a compact subset of $\R$, of which $0$ is an accumulation point.
  \label{Thm-loop-sigma-finite}
\end{Theorem}
\begin{proof}
  By (\ref{Koebe0}) and Lemma \ref{stationary-quasi-rho} (iii), $\mu$ satisfies (iii). Suppose $\eta$ follows the ``law'' $\mu$. By Lemma \ref{stationary-quasi-rho} (i), for any stopping time $\tau$, conditionally on $\F_\tau$ and the event $\Sigma_\tau$, the part of $\eta$ after $\tau$ modulo a time-change is an SLE$_\kappa(\rho,\kappa-6-\rho)$ at
$\HH(\eta|\tau^+):(\eta(\tau^+); [0^+]^{\eta}_{\tau^+;\HH},[0^-]^\eta_{\tau^+;\HH})\to \infty$.  Thus, by the coordinate change in \cite{SW}, the part of the MTC curve $[\eta]$ after $\tau$ follows the law $\K_{\sigma_{\R_-}}(\til\mu^{\HH}_{\eta|\tau^+})$, where $\sigma_{\R_-}$ is the first hitting time at $\R_-:=(-\infty,0)$.
Let $\til\mu=\pi(\mu)$, where $\pi$ is the projection $\eta\mapsto [\eta]$. Then the above property could be described by the formula
\BGE \ind_{\til\Sigma_\tau} \til\mu= \ind_{\til\Sigma_\tau} \til\mu(d[\eta]) \oplus_\tau \K_{\sigma_{\R_-}}(\til\mu^{\HH}_{[\eta]|\tau^+}).\label{hamu-DMP}\EDE

If $\rho\ge \frac\kappa 2-2$, then by the boundary behavior of SLE$_\kappa(\rho)$, a curve following the conditional law $\til\mu^{\HH}_{[\eta]|\tau^+}$ a.s.\ does not hit $([0^+]^{[\eta]}_{\tau^+;\HH},[0^-]^{[\eta]}_{\tau^+;\HH})_{\HH([\eta]|\tau^+)}= (b^{[\eta]}_{\tau^+;\HH},\infty)\cup \{\infty\}\cup \R_-$. Thus,  $\K_{\sigma_{\R_-}}(\til\mu^{\HH}_{[\eta]|\tau^+})=\til\mu^{\HH}_{[\eta]|\tau^+}$, which means that $\til\mu $ satisfies  (\ref{DMP-kappa-rho-proob}). Since $\mu$ satisfies (iii), so does $\til\mu$. Thus, $\til\mu^{\HH}_{0\ccw}:=\til\mu$ satisfies (i,ii.iii). Letting $\tau=\sigma_r$, where $\sigma_r$ is the hitting time at $\{|z|=r\}$ and sending $r\to 0^+$, we see that $\til\mu^{\HH}_{0\ccw} $ is supported by MTC Loewner curves that intersect $\ha\R$ only at $0$.

Suppose now $\rho<\frac\kappa 2-2$. Then a curve following the conditional law $\til\mu^{\HH}_{[\eta]|\tau^+}$ a.s.\ hits $\R_-$, and so $\til\mu$ does not satisfy (\ref{DMP-kappa-rho-proob}). For every $r\in(0,\infty)$,  define
\BGE \til\mu_r=\ind_{\til\Sigma_{\sigma_r }}\mu (d\gamma) \oplus_{\sigma_r}  \til\mu^{\HH}_{\gamma|\sigma_r^+}=\ind_{\til\Sigma_{\sigma_r }}\til\mu (d[\gamma]) \oplus_{\sigma_r}  \til\mu^{\HH}_{[\gamma]|\sigma_r^+}.\label{DMP-kappa-rho-proob-sigma'}\EDE
This is similar to (\ref{DMP-kappa-rho-proob-sigma}) except that   the $\til\mu$ in (\ref{DMP-kappa-rho-proob-sigma}) is  supported by $ \til\Sigma_{\sigma_r }$, and so $\ind_{\til\Sigma_{\sigma_r}}$ is not needed in (\ref{DMP-kappa-rho-proob-sigma}).   By Lemma \ref{continue-Loewner}, $\til\mu_r$ is supported by the space of  Loewner bubbles in $\HH$ rooted at $0$.

 Suppose $0<s<t$. Let $\sigma=\sigma_s$ and $\tau=\sigma_t$.
Let $\tau_{[\gamma];\sigma}$ be as defined in Lemma \ref{tau-[g]'} with  $[g]=[\gamma]$, which is a TCI stopping time by Lemmas \ref{tau-g'} and \ref{tau-[g]'}. If $[\gamma]$ follows the ``law'' $\ind_{\til\Sigma_\sigma}\til\mu$ and $[\eta]$ follows the conditional law $\til\mu^{\HH}_{[\gamma]|\sigma^+}$, then (\ref{compos}) holds for the setting here. Thus, the argument in (\ref{muts}) works here with slight modification. Now we have
 \begin{align}
 \til\mu_t &\stackrel{\phantom{(\ref{Markov-mu-R})}}{=}  \ind_{\til\Sigma_\tau}\til\mu (d[\gamma]) \oplus_{\tau}  \til\mu^{\HH}_{[\gamma]|\tau^+}\nonumber\\
 &\stackrel{ {(\ref{hamu-DMP})}}{=} \ind_{\til\Sigma_{\sigma }}(\til\mu(d[\gamma]) \oplus_{\sigma}  \ind_{\til\Sigma_{\tau_{[\gamma];\sigma}}} \til\mu^{\HH}_{[\gamma]|\sigma^+}(d[\eta]))\oplus_{\tau}  \til\mu^{\HH}_{[\gamma]\oplus_\sigma [\eta]|\tau^+}\nonumber \\
 &\stackrel{ {(\ref{fgh})}}{=} \ind_{\til\Sigma_{\sigma}} \til\mu (d[\gamma]) \oplus_{\sigma}( \ind_{\til\Sigma_{\tau_{[\gamma];\sigma}}} \til\mu^{\HH}_{[\gamma]|\sigma^+}(d[\eta])\oplus_{\tau_{[\gamma];\sigma}} \til\mu^{\HH}_{[\gamma]\oplus_\sigma [\eta]|\tau^+})\nonumber\\
  &\stackrel{ {(\ref{compos})}}{=}\ind_{\til\Sigma_{\sigma}} \til\mu (d[\gamma]) \oplus_{\sigma}  \ind_{\til\Sigma_{\tau_{[\gamma];\sigma}}} \til\mu^{\HH}_{[\gamma]|\sigma^+}= \ind_{\til\Sigma_{\tau}}\til\mu_s,
 \nonumber
 \end{align}
where in the last ``$=$'' we used the fact that $\tau_{[\gamma];\sigma}(\eta)<\ha T_\eta$ iff $\tau(\gamma\oplus_{\sigma(\gamma)}\eta)<\ha T_{\gamma\oplus_{\sigma(\gamma)}\eta}$. We now get $\til\mu_s|_{\til\Sigma_{\sigma_t}}=\til\mu_t$ when $0<s<t$. Thus, there is a $\sigma$-finite measure $\til\mu^{\HH}_{0\ccw}$ on the space of ccw Loewner bubbles in $\HH$ rooted at $0$ such that  $\til\mu^{\HH}_{0\ccw}|_{\til\Sigma_{\sigma_t}}=\til\mu_t$ for every $t\in(0,\infty)$.

Let $\tau$ be a positive TCI stopping time. First suppose $\tau\ge \sigma_s$ for some $s>0$. Let $\sigma=\sigma_s$. The argument in (\ref{tau>sigmar}) works here as follows. Since $\til\Sigma_\tau\subset \til\Sigma_\sigma$, we have
\begin{align}
  &\ind_{\til\Sigma_\tau} \til\mu^{\HH}_{0\ccw}=\ind_{\til\Sigma_\tau} \til\mu_s= \til\mu(d[\gamma]) \oplus_\sigma\ind_{\til\Sigma_{\tau_{[\gamma];\sigma}}} \til\mu^{\HH}_{[\gamma]|\sigma^+}\nonumber\\
  &=\til\mu(d[\gamma]) \oplus_\sigma (\ind_{\til\Sigma_{\tau_{[\gamma];\sigma}}} \til\mu^{\HH}_{[\gamma]|\sigma^+} (d[\eta]) \oplus_{\tau_{[\gamma];\sigma}} \til\mu^{\HH}_{[\gamma]\oplus_\sigma[\eta]|\tau^+})\nonumber \\
  &=\ind_{[\gamma]\oplus_\sigma[\eta]\in \til\Sigma_{\tau}} ( \til\mu(d[\gamma])  \oplus_\sigma   \til\mu^{\HH}_{[\gamma]|\sigma^+} (d[\eta]) )\oplus_\tau \til\mu^{\HH}_{[\gamma]\oplus_\sigma[\eta]|\tau^+}\nonumber \\
  &= \ind_{ \til\Sigma_{\tau}}\til\mu (d[\zeta]) \oplus_\tau \til\mu^{\HH}_{ [\zeta]|\tau^+}.
   \nonumber
\end{align}
Thus, we get (\ref{DMP-kappa-rho-proob}) in the case that $\tau\ge \sigma_s$ for some $s>0$.

For a general positive TCI stopping time $\tau$, we first apply the above result to $\tau\wedge \sigma_{1/n}$ to conclude that  (\ref{DMP-kappa-rho-proob}) holds with an additional factor $\ind_{\{\sigma_{1/n}<\tau\}}$ on both sides, and then take a union over $n\in\N$ to get (\ref{DMP-kappa-rho-proob}). Thus, $\til\mu^{\HH}_{0\ccw}$ satisfies (i). For any $r>0$, from the construction of $\til\mu_r$ and that $\ind_{\til\Sigma_{\sigma_r}}\til\mu^{\HH}_{0\ccw}=\til\mu_r$, we know that $\ind_{\til\Sigma_{\sigma_r}}{\cal P}(\til\mu^{\HH}_{0\ccw})=\ind_{\til\Sigma_{\sigma_r}} (\mu)$, which implies that ${\cal P}(\til\mu^{\HH}_{0\ccw})=\mu$, i.e., (ii) holds. Since $\mu$ satisfies (iii), by (\ref{radxPeta}), $\til\mu^{\HH}_{0\ccw}$ satisfies (iii) as well.

 By the boundary behavior of SLE$_\kappa(\rho)$ for $\rho\in((-2)\vee(\frac\kappa 2-4),\frac\kappa 2-2)$, for any $r>0$, if $\gamma$ follows the law $\ind_{\til\Sigma_{\sigma_r}}\mu$, and $[\eta]$ follows the conditional law $\til\mu^{\HH}_{\gamma|\sigma_r^+}$, then $[\eta]$ ends at $0$, $0$ is an accumulation point of $[\eta]\cap \R_-$, and $\infty\not\in [\eta]$. By (\ref{DMP-kappa-rho-proob-sigma'}), $\til\mu_r$ is supported by curves in $\lin\HH$, whose intersection with $\R$ is a compact set, and $0$ is an accumulation point of the intersection. Since this holds for any $r>0$, $\til\mu^{\HH}_{0\ccw}$ also satisfies these properties.

 Suppose $\til\mu'$ also satisfies (i) and (iii). Let  $\mu'={\cal P}(\til\mu')$. Then $\til\mu':=\pi(\mu')$ satisfies (\ref{hamu-DMP}), and there is a measure $\mu_0'$ supported by $\Sigma^{\Lo}$ such that $\mu'=\Lo(\mu_0')$. Since  $\til\mu'$ satisfies (iii), so does $\mu'$.  By (\ref{Koebe0}), $\mu_0'$ satisfies Lemma \ref{stationary-quasi-rho} (iii). Let $\eta$ follow the ``law'' $\mu'$. Since $\til\mu'$ satisfies (\ref{hamu-DMP}), by SLE coordinate changes, for any stopping time $\tau$, conditionally on $\F_\tau$ and $\til\Sigma_{\tau}$, the part of $\eta$ after $\tau$ is the $(g^\eta_\tau)^{-1}$-image of a CP SLE$_\kappa(\rho,\kappa-6-\rho)$ at $(W^\eta_\tau;D^\eta_\tau,C^\eta_\tau)$. Thus, $\mu_0'$ satisfies Lemma \ref{stationary-quasi-rho} (i) for $\rho_+=\rho$ and $\rho_-=\kappa-6-\rho$. By the uniqueness part of that lemma, $\mu_0'=C\mu_0$ for some $C\ge 0$, which implies that $\mu' =C\mu$. Since both $\til\mu^{\HH}_{0\ccw}$ and $\til\mu'$ satisfy (\ref{DMP-kappa-rho-proob}), from ${\cal P}(\til\mu')=C  {\cal P}(\til\mu^{\HH}_{0\ccw})$ we then  get $\til\mu'=C\til\mu^{\HH}_{0\ccw}$.

 Let $f$ be a conformal automorphism of $\HH$ that fixes $0$. Then $f(\til\mu^{\HH}_{0\ccw})$ also satisfies (i) and (iii), and so equals some constant  times $\til\mu^{\HH}_{0\ccw}$.
 First, suppose $f(z)=cz$ for some $c>0$. Then $f'(0)=c$.
Let $S_c$ denote the Brownian scaling map: $X\mapsto c X(\cdot/c^2)$.
We are going to use the following facts: (a) for any MTC Loewner curve $[\eta]$ in $\HH$ started from a real number, ${\cal P}([f\circ \eta])=S_c({\cal P}([\eta]))$; (b) for any $W\in\Sigma^\Lo$,   $S_c(\Lo(W))=\Lo(S_c(W))$, $C^{S_c(W)}=S_c( C^W)$, and $D^{S_c(W)}=S_c( D^W)$, which together imply that $Q^{S_c(W)}_t=c^{-\alpha} Q^W_{t/c^2}$; and (c) $S_c(\nu_0^*)=\nu_0^*$. Since ${\cal P}(\til\mu^{\HH}_{0\ccw})=\mu=\Lo(\mu_0)$, we get ${\cal P}(f(\til\mu^{\HH}_{0\ccw}))=S_c(\mu)=\Lo( S_c(\mu_0))$. By the relation between $\mu_0$ and $\nu_0^*$ described in Lemma \ref{stationary-quasi-rho} (ii),  for any $t_0\in(0,\infty)$, $S_c(\mu_0)\ll S_c(\nu_0^*)=\nu_0^*$ on $\F_{t_0 }\cap \til\Sigma_{t_0}$ with the RN derivative being $Q^{S_c^{-1}(W)}_{t_0/c^2}= c^\alpha Q^W_{t_0}$. This implies that $S_c(\mu_0)=c^\alpha \mu_0$. Thus, ${\cal P}(f(\til\mu^{\HH}_{0\ccw}))=c^\alpha {\cal P}(\til\mu^{\HH}_{0\ccw})$, which further implies that $f(\til\mu^{\HH}_{0\ccw})=c^\alpha \til\mu^{\HH}_{0\ccw}$ by the DMP.

 It remains to show that if $f'(0)=1$, then $f(\til\mu^{\HH}_{0\ccw})=\til\mu^{\HH}_{0\ccw}$. Let $J(z)=-1/z$ and $\til\mu _\infty=J(\til\mu^{\HH}_{0\ccw})$. The statement is equivalent to that, for any $c\in\R$, $f_{+c}(\til\mu_\infty)=\til\mu_\infty$, where $f_{+c}(z):=z+c$. First, we know that $f_{+c}(\til\mu_\infty)=C\til\mu_\infty$ for some $C>0$. Let $M_2(z)=2z$. By the scaling property of $\til\mu^{\HH}_{0\ccw}$, $M_2(\til\mu_\infty)=2^{-\alpha} \til\mu_\infty$.  Since $f_{+2c}=f_{+c}\circ f_{+c}$, we have $f_{+2c}(\til\mu_\infty)=C^2\til\mu_\infty$. On the other hand, from $f_{+2c}=M_2\circ f_c\circ M_2^{-1}$, we get
 $f_{+2c}(\til\mu_\infty)=2^{-\alpha}\times C \times 2^\alpha\times \til\mu_\infty=C\til\mu_\infty$. Thus, $C^2=C$, which implies that $C=1$. So we get $f_{+c}(\til\mu_\infty)=\til\mu_\infty$, as desired.
\end{proof}

By the conformal invariance and covariance of the $\til\mu^{\HH}_{0\ccw}$ in Theorems \ref{Thm-loop-probability} and \ref{Thm-loop-sigma-finite}, we get the following theorem.

\begin{Theorem} Let $\kappa>0$ and $\rho>-2$.
   Let $J(z)=-\frac 1z$ and $\til\mu^{\HH}_{\infty\ccw}=J(\til\mu^{\HH}_{0\ccw})$. Let $f(z)=cz+d$, where $c,d\in\R$ and $c>0$. Then  $f(\til\mu^{\HH}_{\infty\ccw})= \til\mu^{\HH}_{\infty\ccw}$ if $\rho\le \frac\kappa 2-4$; and $=c^{-\alpha} \til\mu^{\HH}_{\infty\ccw}$ if $\rho>\frac\kappa 2-4$, where $\alpha$ is as in (\ref{alpha}).
   \label{scaling-invariance}
\end{Theorem}

\begin{Remark}
The above construction is more straightforward than the construction of SLE$_\kappa$ loop measures in \cite{loop} because it does not rely on the Minkowski content measure. We may use the method here to improve the construction of rooted SLE$_\kappa$ loop in $\ha\C$. We first make the following observation. Let $r>0$ and $\D:=\{z\in\C:|z|>1\}\cup\{\infty\}$. Let $v<w<v+2\pi$ and $[\eta]$ be a chordal SLE$_\kappa$ curve in $r\D^*$ from $re^{iw}$ to $re^{iv}$. We stop $[\eta]$ at the time $\sigma$ when it reaches $\infty$ or disconnects $r e^{iv}$ from $\infty$, and parametrize the truncated $[\eta]$ such that, if $g_t$ maps $\D^*_t$, the connected component of $\D^*\sem \eta[0,t]$ that contains $\infty$, conformally onto $\D^*$, fixes $\infty$, and satisfies $g_t'(\infty):=\lim_{z\to\infty} z/g_t(z)>0$, then $g_t'(\infty)/g_0'(\infty)=e^t$. Let $W\in \Sigma$ be such that $W_0=w$ and $e^{iW_t}=g_t(\eta(t^+))$. Then $g_t$ satisfies the radial Loewner equation driven by $W$. Let $V\in \Sigma$ be such that $V_0=v$ and $e^{iV_t}=g_t(re^{iv})$. By SLE coordinate changes, $W,V$ satisfy the SDE/ODE:
\begin{align*}
  dW_t&=\sqrt\kappa dB_t+\frac{\kappa-6}2\cot_2(W_t-V_t)dt,\\
  dV_t&=\cot_2(V_t-W_t)dt,
\end{align*}
where $B$ is a standard Brownian motion and $\cot_2 :=\cot(\cdot/2)$. Let $X_s=\cos_2(W_{\frac 4\kappa s}-V_{\frac 4\kappa s}) $, where $\cos_2:=\cos(\cdot/2)$. Then $X$ satisfies the radial Bessel SDE of dimension $(\delta,\delta)$ up to the time it hits $\{1,-1\}$, where $\delta:=3-\frac8\kappa$. On the other hand, if $X$ is a radial Bessel process of dimension $(\delta,\delta)$, which is started from $w-v$ and stopped when it hits $\{1,-1\}$, then we can recover  $W,V$   from $X$ by
\BGE V_t=v-\int_0^t \cot_2\circ \cos_2^{-1}(X_{\frac\kappa 4 s})ds,\quad W_t=V_t+\cos_2^{-1}(X_{\frac\kappa 4 t}). \label{WVX}\EDE

If $\kappa\ge 8$, then $\delta\ge 2$ and $X$ does not hit $1$ or $-1$ at a finite time, and so has lifetime $\infty$. Let $\nu^\delta$ denote the law of $X$. Recall the existence of  the stationary radial Bessel process of dimension $(\delta,\delta)$, whose law we denote by $\nu^\delta_{\R}$. Suppose $X^{\R}$ follows the law $\nu^\delta_{\R}$. Define random processes $W^{\R},V^{\R}\in\Sigma_{\R}$ such that (i) $V^{\R}_0$ and $X^{\R}$ are independent; (ii) $V^{\R}_0$ is uniform on $[0,2\pi)$; and (iii)  (\ref{WVX}) holds  for $t\in\R$ with $W,V,X,v$ replaced by $W^{\R},V^{\R},X^{\R},V^{\R}_0$. Let $\eta(t)$, $t\in\R$, be the whole-plane Loewner curve driven by $W$. Then for any finite TCI stopping time $\tau$, conditionally on $[\eta]|_{(-\infty,\tau]}$, the part of $[\eta]$ after $\tau$ is a chordal SLE$_\kappa$ curve in $\ha\C\sem \eta(-\infty,t_0]$ from $\eta(t_0)$ to $0$ stopped when it hits $\infty$. We continue $[\eta]$ with a chordal SLE$_\kappa$ curve in $\ha\C\sem \eta$ from $\infty$ to $0$. Then the law of the resulted random MTC curve is an SLE$_\kappa$ loop in $\ha\C$ rooted at $0$.

Now suppose $\kappa<8$. Then $\delta<2$, and so $X$ stops at some finite time when it hits $\{1,-1\}$. Let $\mu^\delta$ denote the law of $X$. Let $\delta^*=4-\delta>2$ and $\nu^{\delta^*}$ denote the law of a radial Bessel process of dimension $(\delta^*,\delta^*)$, which has lifetime $\infty$. Then $\mu^\delta\tl\nu^{\delta^*}$, and the RN process is $M^Z_t/M^Z_0$, where $M^Z_t:=e^{\alpha_* t} (1-Z_t^2)^{\alpha_*}$ and $\alpha_*:=\frac\delta 2-1<0$. Let $\nu^{\delta^*}_{\R}$ denote the law of the stationary radial Bessel process of dimension $(\delta^*,\delta^*)$. Mimicking the proof of Lemma \ref{stationary-quasi}, we can prove that there is a $\sigma$-finite measure $\mu^\delta_\R$ on $\Sigma^{\R}$ such that for any $\F^{\R}$-stopping time $\tau>-\infty$, we have the Markov formula $\ind_{\Sigma^{\R}_\tau} \mu^\delta_\R= \ind_{\Sigma^{\R}_\tau} \mu^\delta_\R(dZ)\oplus_\tau \mu^\delta_{Z_\tau}$, and $\mu^\delta_\R\ll \nu^{\delta^*}_\R$ on $\F^{\R}_\tau\cap \Sigma^{\R}_\tau$ with $M^Z_\tau$ being the RN derivative. Suppose $Z$ follows the ``law'' $\mu^\delta_\R$. Let $v$ be a random variable, which is uniform on $[0,2\pi)$ and is independent of $Z$. We use (\ref{WVX}) to define a $\Sigma_\R$-valued random element $W$. Let $\eta$ be the whole-plane Loewner curve driven by $W$. Then $\eta$ ends at a point other than $0$ and $\infty$. Finally, we continue $\eta$ with a chordal SLE$_\kappa$ curve from the endpoint of $\eta$ to $0$ in the (unique) bounded connected component of $\ha\C\sem \eta$ whose boundary contains $0$. Then the ``law'' of the resulted curve modulo time-change is an SLE$_\kappa$ loop in $\ha\C$ rooted at $0$.
%
\end{Remark}

\subsection{Weak convergence}\label{section-weak}
 We now prove that the rooted SLE$_\kappa(\rho)$ bubble measure $ \til\mu^{\HH}_{0\ccw}$ is the weak limit of some rescaled SLE$_\kappa(\rho)$ measures  on the metric space $\til\Sigma^{\ha\C}$ with the distance defined by (\ref{dist-E}). We will use the facts (F1,F2,F3) in Section \ref{Section-MK}.


\begin{Theorem}
\begin{enumerate}
  \item [(i)]  Suppose $\rho\in(-2,\frac\kappa 2-4]$. Then
$\til\mu^{\HH}_{(r;r^+)\to -r}\wto \til\mu^{\HH}_{0\ccw}$ as $r\to 0^+$ in  $\til\Sigma^{\ha\C}$.
\item [(ii)] Suppose $\rho>\frac\kappa 2-4$. Let $\alpha>0$ be as in (\ref{alpha}).
 For $R>0$,  let $E_R= \{[\eta]:\rad_0([\eta])>R\}$. Then
\BGE (2r)^{-\alpha}\ind_{E_R} \til\mu^{\HH}_{(r;r^+)\to -r} \wto \ind_{E_R}\til\mu^{\HH}_{0\ccw},\quad \mbox{as }r\to 0^+.\label{weak-S}\EDE
\end{enumerate}
\label{weak-1}
\end{Theorem}
\begin{proof}  We will use the following symbols.  For a CP Loewner curve $\gamma$  and $s\in [|\gamma(0)|,\infty)$, let $V^{s;\gamma}_t=[s]^\gamma_t$ and $X^{s;\gamma}_t=[-s]^\gamma_t$. Here we allow $\gamma(0)=s$ (resp.\ $\gamma(0)=-s$), in which case $[s]^\gamma_t$ (resp.\ $[-s]^\gamma_t$) is understood as $[\gamma(0)^+]^\gamma_t$ (resp.\ $[\gamma(0)^-]^\gamma_t$).
When $s=r\in(0,\infty)$, define $Z^{r;\gamma} $, $p_r^{\gamma} $ and $\ha Z^{r;\gamma} $ as in Lemma \ref{correspondance-mu} (i) for $v=r$, $x=-r$, $W=W^{\gamma}$, $V=V^{r;\gamma}$, and $X=X^{r;\gamma}$. When $\gamma(0)=s=0$, define  $Z^{0;\gamma} $, $p_{0}^{\gamma} $ and $\ha Z^{0;\gamma} $ as in Lemma \ref{correspondance-mu-infty} (i) for $W=W^{\gamma}$, $V=V^{0;\gamma}$, and $X=X^{0;\gamma}$. For $T>r>0$, let  $t^T_r=\frac\kappa 2\log(T/r)$  and $t^T_{0 }= \frac \kappa 2 \log(2T)$, and for $s\in\{r,0\}$, let $\tau_s^T(\gamma)$ be the biggest $t_0\in (0,\ha T_\gamma]$ such that $V^{s;\gamma }_t-X^{s ;\gamma}_t<  2T$ for $0\le t<t_0$, which is a stopping time.

Let $\til\mu_{0}=\til\mu^{\HH}_{0\ccw}$ and $\mu_{0}={\cal P}(\til\mu_{0})$. For $r\in(0,\infty)$, let $\til\mu_r=\til\mu^{\HH}_{(r;r^+)\to -r}$, $\mu_r={\cal P}(\til\mu_r)$, and $\gamma_r$ follow the law $\mu_r$. Let $\rho_+=\rho$, $\rho_-=\kappa-6-\rho$, and $\delta_\pm=\frac 4\kappa(\rho_\pm+2)$. Then $\delta_+>0$. By SLE coordinate changes, $\gamma_r$ is a CP SLE$_\kappa(\rho_+,\rho_-)$ curve started from $r$ with force points $r^+,-r$. We write $V^r$ for $V^{r;\gamma_r}$, and similarly $X^r $, $Z^r$, $p_r$ and $\ha Z^r$.   By Lemma \ref{correspondance-mu}, $ \ha Z^{r} $ is a radial Bessel process of dimension $(\delta_+,\delta_-)$ started from $1$ and stopped when it hits $-1$.

(i) Suppose $\rho\in(-2,\frac\kappa 2-4]$. Then $\rho_-\ge \frac\kappa 2-2$, $\delta_-\ge 2$, and $\til\mu_{0},\mu_{0}$ are probability measures. Let $\gamma_{0}$ follow the law $\mu_{0}$. 
Then $\gamma_{0}$ is a CP SLE$_\kappa(\rho_+,\rho_-)$ curve started from $0$ with force points $0^+,0^-$. We write $V^{0}$ for $V^{{0};\gamma_{0}}$, and similarly $X^{0}$, $Z^{0}$, $p_{0}$ and $\ha Z^{0}$. By Lemma \ref{correspondance-mu-infty},   $ \ha Z^{{0} } $ is a stationary radial Bessel process of dimension $(\delta_+,\delta_-)$.

Fix $T>r>0$.  We will omit the $T$ in $t_s^T$ and $\tau_s^T$ for $s\in \{r,0\}$. Since $\rho_+>-2$ and $\rho_-\ge \frac\kappa 2-2$, $V^{s}_t-X^{s}_t$ increases continuously to $\infty$ with the initial value $2s$, we conclude that $\tau _s$ is  finite, and $p_j (\tau _s)=t_s $. So $Z^{r}_{\tau_r }=\ha Z^{r}_{t_r }$ has density $p_{t _r}(1,\cdot)$, and $Z^{0}_{\tau_{0} }=\ha Z^{0}_{t _{0}}$ has density $p_\infty$, where $p_t(\cdot,\cdot)$ and $p_\infty(\cdot)$ are the transition density and stationary density, respectively, for the radial Bessel processes of dimension $(\delta_+,\delta_-)$. By (\ref{transition-rate-i}), there are $C,L\in(0,\infty)$ depending only on $\kappa,\rho$, such that if $t_r >L$, i.e., $\frac r T< e^{-\frac 2\kappa L}$,  then
$|p_{t_r }(1,y)-p_\infty(y)|\le C p_\infty(y) (\frac r T )^{\frac{\kappa -2}2}$.
Assume now $\frac r T< e^{-\frac 2\kappa L}$. Then there is a coupling $\gamma_r$ and $\gamma_{0}$ of $\mu_r$ and $\mu_{0}$ such that $\PP[Z^{r}_{\tau_r }\ne Z^{0}_{\tau_r }]\le C (\frac r T )^{\frac{\kappa -2}2}$.
Let
$\til \mu_s^\otimes= \mu_s(d\gamma)\rotimes  \til \mu^{\HH}_{(Z^{s;\gamma}_{\tau_s };1)\to -1}$, $s\in\{r,0\}$.
Then we get a coupling $(\gamma_r,[\eta_r])$ and $(\gamma_{0},[\eta_{0}])$ of $\til\mu_r^\otimes$ and $\til\mu_{0}^\otimes$ such that $\PP[[\eta_r]\ne [\eta_{0}]]\le C (\frac r T )^{\frac{\kappa -2}2}$.
Let $[\zeta_s]=\gamma_s\oplus_{\tau_s }[f_{s}^{\gamma_s}(\eta_s)]$, where
\BGE f_s^\gamma(z) := (g^\gamma_{\tau_s^T})^{-1}(T z+ ({V^{s;\gamma}_{\tau_s^T}+X^{s;\gamma}_{\tau_s^T}})/2),\quad s\in\{r,0\}. \label{fj}\EDE By the DMP of SLE$_\kappa(\rho)$ curve and rooted SLE$_\kappa(\rho)$ bubble,  $[\zeta_r]$ and $[\zeta_{0}]$ form a coupling of $\til\mu_r$ and $\til\mu_{0}$.
By (\ref{Koebe0},\ref{Koebe0'}), $\rad_0(K^{\gamma_s}_{\tau_s})\le 2 T$, which together with (\ref{g-z}) implies that $|(g^{\gamma_s}_{\tau_s})^{-1}(z)-z |\le 6 T$ for any $z\in \HH $. Since $V^{s}_t\ge 0\ge X^{s }_t$ and $V^{s }_{\tau_s}-X^{s }_{\tau_s}=2T$, we have $|V^{s }_{\tau_s}+X^{s}_{\tau_s}|\le 2 T$. Thus, for any $z\in\lin \HH$, $|f_r^{\gamma_r}(z)-f_{0}^{\gamma_{0}}(z)|\le 14 T$.  So, on the event that $[\eta_r]=[\eta_{0}]$, we have $\dist_{\til \Sigma^{\ha\C}}([f_r^{\gamma_r}(\eta_r)],[f_{0}^{\gamma_{0}}(\eta_{0})])\le 14 T$. From  $\rad_0(K^{\gamma_s}_{\tau_s})\le2 T$ we have $\diam(\gamma_s[0,\tau_s])\le 4T$. Since $[\zeta_s]=\gamma_s\oplus_{\tau_s}[f_s^{\gamma_s}(\eta_s)]$, we have $\dist_{\til \Sigma^{\ha\C}}([\zeta_r],[\zeta_{0}])\le 22 T$ when $[\eta_r]=[\eta_{0}]$. Thus,  $\PP[\dist_{\til \Sigma^{\ha\C}}([\zeta_r],[\zeta_{0}])> 22 T]< C (\frac r T )^{\frac{\kappa -2}2}$. For any $\eps>0$, by  choosing $T=\eps/22$, we see that  $\PP[\dist_{\til \Sigma^{\ha\C}}([\zeta_r],[\zeta_0])> \eps]< \eps$ when $r$ is small enough. By  (F1) we get the weak convergence.

 (ii) Suppose $\rho>\frac\kappa 2-4$. Then $\delta_-<2$, and $\til\mu_{0},\mu_{0}$ are infinite measures, and for $r>0$, $\gamma_r$ has a finite lifetime.  Let $\rho_-^*=\rho+2$  and $\delta_-^*=\frac 4\kappa(\rho_-^*+2)$. Then $\rho_->\frac\kappa 2-2$ and $\delta_-^*=4-\delta_->2$.
Let $p_t^*(x,y)$ and $p_\infty^*(y)$ be the transition density and stationary density, respectively, for radial Bessel processes of dimension $(\delta_+,\delta_-^*)$.

Fix $R>0$. For $s\in[0,R)$, let $F_s^R$ denote the set of CP Loewner curves $\gamma$ with $|\gamma(0)|\le s$ such that $\tau_s^R(\gamma) <\ha T(\gamma)$, and let $\til F_s^R={\cal P}^{-1}(F_s^R)$.
We  will prove that
 \BGE (2r)^{-\alpha}\ind_{\til F_r^R} \til\mu_r \wto \ind_{\til F_{0}^R }\til\mu_{0},\quad\mbox{as }r\to 0^+.\label{weak-F}\EDE
By (F2) it suffices to show that, as $r\to 0^+$,
 \BGE (2r)^{-\alpha} \mu_r( F_r^R)\to \mu_{0}( F_{0}^R);\label{weak-F-1}\EDE
 \BGE   \til\mu_r[\cdot|\til F_r^R] \wto  \til\mu_{0}[\cdot|\til F_{0}^R].\label{weak-F-2}\EDE

Let $r>0$.
By Lemma \ref{correspondance-nu}, $ \ha Z^{r} $ follows the law $\mu^{\delta_+,\delta_-}_1$.   By Proposition \ref{Bessel-transition} and Lemma \ref{Girsanov-lem}, $$\PP[\ha T_{Z^r}>\tau_r^R]=\PP[\ha T_{\ha Z^{r}}>t _r^R]=\int_{-1}^1 e^{-\frac{\delta_+}8(2-\delta_-) t_r^R} \Big(\frac{1+y}2 \Big)^{\frac{\delta_-}2-1} p_{t_r^R}^*(1,y) dy$$
$$= \Big(\frac { r}{ R}\Big)^{\alpha}  \int_{-1}^1 \Big( \frac{1+y}2 \Big )^{{\frac{\delta_-}2-1}} p_{t_r^R}^*(1,y) dy.$$
 Thus,
\BGE (2r)^{-\alpha}\mu_r(F^r_R) = (2R)^{-\alpha} \int_{-1}^1\Big ( \frac {1+y}2\Big )^{{\frac{\delta_-}2-1}} p_{t_r^R}^*(1,y) dy.\label{2rtimes}\EDE

By Theorem \ref{Thm-loop-sigma-finite} (ii), $\mu_{0}$ agrees with the $\mu_0$ in Lemma \ref{stationary-quasi-rho}.
 By the proof of Lemma  \ref{stationary-quasi-rho},  the pushforward of $\mu_0$ under the map $\gamma\mapsto (\ha Z^{0;\gamma_{0}}_s)$ is the $\mu^{\delta_+,\delta_-}_{\R}$ in Lemma \ref{stationary-quasi}.
Thus, by (\ref{MZt}),
\BGE \mu_{0}(F_{0}^R)=
\int_{-1}^1 e^{-\frac{\delta_+}8(2-\delta_-)t_{0}^R} \Big( \frac{1+y}2 \Big)^{\frac{\delta_-}2-1}p_\infty^*(y) dy  =(2R)^{-\alpha} \int_{-1}^1  \Big( \frac{1+y}2 \Big)^{{\frac{\delta_-}2-1}} p_{\infty}^*(y) dy,\label{starionary}\EDE
 Then  (\ref{weak-F-1}) follows from  (\ref{2rtimes},\ref{starionary},\ref{transition-rate-i}) and the fact that $\int_{-1}^1 (1+y)^{{\frac{\delta_-}2-1}} p_\infty^*(y) dy <\infty$.

Now we turn to (\ref{weak-F-2}). By (F1) it suffices to construct a coupling $([\zeta_r],[\zeta_{0}])$ of   $\til\mu_r^R:=\til\mu_r[\cdot|\til F_r^R]$ and $\til\mu_{0}^R:=\til\mu_{0}[\cdot|\til F_{0}^R]$  such that when $r$ is small, with high probability, $[\zeta_r]$ and $[\zeta_{0}]$ are close. Let $\mu_s^R={\cal P}(\til\mu_s^R)$, $s\in\{r,0\}$. Note that if $\gamma_s$ follows the law $\mu_s^R$, then $\ha Z^{s;\gamma_s}$ follows the law $\mu^{\delta_+,\delta_-}_s[\cdot|\Sigma^{\sigma(s)}_{t_s^R}] $, where $\sigma(s)$ is empty when $s=r$ and is $\R$ when $s=0$.

Let $T\in (r,R)$.  For $s\in\{r,0\}$, by the DMP of SLE$_\kappa(\rho)$ and rooted SLE$_\kappa(\rho)$ bubble, we have
\BGE \til\mu_s^R=\til\mu_s^{R}(d\gamma)\oplus_{\tau_s^T}f^\gamma_s(\til\mu^{\HH}_{(Z^{s;\gamma}_{\tau_s^T};1)\to -1}[\cdot|\til F^{R/T}_1])  ,\label{DMP-mu-j}\EDE
where $f^\gamma_s$ is as in (\ref{fj}). Applying Lemma \ref{coupling} with $t_1=t_r^T$, $t_2=t_0^T$ and $t_0=\log(R/T)$ and using $Z^{s;\gamma_s}_{\tau_s^T}=\ha Z^{s;\gamma_s}_{t_s^T}$, we see that there are constants $C,L_0\in(0,\infty)$ such that if $t_r^T>L_0$, i.e., $\frac r{T}<e^{-\frac 2\kappa L_0}$, then there is a coupling  $ \gamma_r$ and $\gamma_{0}$ of $\mu_r^R$ and $\mu_{0}^R$ such that $\PP[ Z^{r;\gamma_r}_{\tau_r^T}\ne Z^{0;\gamma_{0}}_{\tau_{0}^T}]<8C (\frac r T )^{\frac{\rho+3}2}$.
 Assume $\frac r{T}<e^{-\frac 2\kappa L_0}$. Then there exists  a coupling $(\gamma_r,[\eta_r])$ and $(\gamma_{0},[\eta_{0}])$ of the measures
$\til\mu_s^\otimes:=\mu^T_s(d\gamma)\rotimes \til\mu^{\HH}_{(Z^{s;\gamma}_{\tau_r^T};1)\to -1}[\cdot|\til F^{R/T}_1](d[\eta])$, $s\in\{r,0\}$,
such that $\PP[[\eta_r]\ne [\eta_{0}]]< 8C (\frac r T )^{\frac{\rho+3}2}$. Let  $[\zeta_s]=\gamma_s\oplus_{\tau_s^T} f^{\gamma_s}_{s}([\eta_s])$, $s\in\{r,0\}$. From (\ref{DMP-mu-j}) we know that $[\zeta_r]$ and $[\zeta_{0}]$ form a coupling of $\til\mu_r^R$ and $\til\mu_{0}^R$. The argument in the last paragraphs of the proof of (i) shows that  $\PP[\dist([\zeta_r],[\zeta_{0}])>22 T]<8C (\frac r T )^{\frac{\rho+3}2}$.
For $\eps>0$, by taking $T<\eps/22$ we see that $\PP[\dist ([\zeta_1],[\zeta_2])> \eps]< \eps$ when $r$ is small enough. By (F1) we get (\ref{weak-F-2}), which together with (\ref{weak-F-1}) implies (\ref{weak-F}).

Now we turn to (\ref{weak-S}). It is clear that $E_R$ is open in $\til\Sigma^{\ha\C}$, and    $\pa E_R\subset  \{[\eta]:\rad_0([\eta])= R\}$. By the scaling property of $\til\mu_0$, $\til\mu_0(\pa E_R)=0$. Let $E^R_s$ be the set of $[\eta]\in E_R$ with $[\eta](0)=s$. Suppose $R>2r$.
By (\ref{Koebe0'}),  for $s\in\{r,0\}$, $E^R_s\subset \til F_s^{R/2} $, which implies that
$\ind_{E_R} \til\mu_s=\ind_{E_R}\ind_{\til F_r^{R/2}} \til\mu_s$ because $\til\mu_s$ is supported by curves started from $s$. Then (\ref{weak-S}) follows from (\ref{weak-F}) and (F3).
\end{proof}

\begin{Corollary}
  Whenever SLE$_\kappa(\rho)$ satisfies reversibility, i.e., ${\cal R}(\til\mu^{\HH}_{(r;r^+)\to -r})= \til\mu^{\HH}_{(-r;(-r)^-)\to r}$ for $r>0$, the SLE$_\kappa(\rho)$ bubble measure satisfies the reversibility: ${\cal R}(\til\mu^{\HH}_{0\ccw})=\til\mu^{\HH}_{0\cw}$. \label{Reversibility-cor}
\end{Corollary}

We have known that, for $\kappa\in(0,8]$ and $\rho\in (-2,\infty)\cap[ \frac\kappa2-4,\infty)$,  SLE$_\kappa(\rho)$ satisfies reversibility  (\cite{SS, intermideate,MS2,MS3}), and so the corresponding  rooted SLE$_\kappa(\rho)$ bubble also satisfies reversibility.

\begin{Remark}
  We may use the argument in this subsection to show that the chordal SLE$_\kappa$ measure in $\ha\C\sem \{|z|\le r\}$ from $r$ to $-r$, suitably rescaled, converges weakly to the SLE$_\kappa$ loop measure in $\ha\C$ rooted at $0$, as $r\to 0^+$, in the sense of Theorem \ref{weak-1} (i) for $\kappa\ge 8$ and of Theorem \ref{weak-1} (ii) for $\kappa<8$ with $\alpha=1-\frac\kappa 8$. The role of (\ref{g-z}) in the argument will be played by a similar inequality: if $g$ maps $\ha\C\sem K$, where $0\in K$, conformally onto  $\ha\C\sem \{|z|\le r\}$, and satisfies $\lim_{z\to\infty} g(z)/z=1$, then $|g(z)-z|\le 5r$ for any $z\in\C\sem K$.
\end{Remark}

\section{Decomposition of SLE$_\kappa(\rho)$ Bubbles} \label{section-decomposition}
In this section we will  derive  decomposition theorems for multi-force-point SLE$_\kappa(\ulin\rho)$ curves in Section \ref{section-decomp} and for rooted SLE$_\kappa(\rho)$ bubbles in Section \ref{section-proof}, which help us to construct unrooted SLE$_\kappa(\rho)$ bubble measures in Section \ref{section-unrooted}. We will also derive a  decomposition theorem for unrooted SLE$_\kappa(\rho)$ bubble measures.

\subsection{Decomposition of SLE$_\kappa(\ulin\rho)$}\label{section-decomp}
Recall that $\PP^{w;\ulin v}_{\kappa;\ulin\rho}$ denotes the law of the driving function of a CP SLE$_\kappa(\ulin\rho)$ curve at $(w;\ulin v)$.

\begin{Lemma}
  Let $\kappa>0$, $\rho_1,\dots,\rho_m,\sigma,\sigma'\in\R$, $m\in\N\cup\{0\}$, and $\ulin\rho=(\rho_1,\dots,\rho_m)$.  Let $w\in\R$, $v_1,\dots,v_m\in \{w^+,w^-\}\cup (\R\sem \{w\})$, $\ulin v=(v_1,\dots,v_m)$, and $u\in \R\sem \{w,v_1,\dots,v_m\}$. Let $\tau$ be a stopping time.
   Then
$$  \frac{ \PP_{\kappa;\ulin\rho,\sigma'}^{w;\ulin v,u}(dW)|\F_{\tau}\cap \{\tau<\ha T\wedge \tau^W_u\}}{ \PP_{\kappa;\ulin\rho,\sigma}^{w;\ulin v,u}(dW)|\F_{\tau}\cap \{\tau<\ha T\wedge \tau^W_u\}}=\frac{M^W_\tau}{M^W_0}, $$ 
  where
\BGE M^W_t :=(g^W_t)'(u)^{\frac{ (\sigma'-\sigma)(\sigma'+\sigma+4-\kappa)}{4\kappa}}  |W_t-\la u\ra^W_t |^{\frac{2(\sigma'-\sigma)}{2\kappa}} \prod_{j=1}^m |\la u\ra^W_t-\la v_j\ra^W_t|^{\frac{\rho_j(\sigma'-\sigma)}{2\kappa}}.\label{MWt} \EDE
  Recall the $\la u\ra^W_t$ and $\la v_j\ra^W_t$ defined in Section \ref{section-SLE}, which are the force point processes started from $u$ and $v_j$.  In particular,
  $\K_{\tau_u}(\PP^{w;\ulin v,u}_{\kappa;\ulin\rho,\sigma'})\tl \K_{\tau_u}(\PP^{w;\ulin v,u}_{\kappa;\ulin\rho,\sigma})$ with the RN process being $ M^\cdot_t/M^\cdot_0$.
  \label{P/P-kappa-rho}
\end{Lemma}
\begin{proof}
This is a standard application of It\^o's formula and Girsanov Theorem to the Loewner equations. One easily check that $M^W_t$ is a positive local martingale under $\PP_{\kappa;\ulin\rho,\sigma}^{w;\ulin v,u}$, and weighting $\PP_{\kappa;\ulin\rho,\sigma}^{w;\ulin v,u}$ by $M^W_\tau/M^W_0$ for any $\F$-stopping time $\tau<\ha T\wedge \tau^W_u$ such that $M^W_{\cdot\wedge \tau}$ is uniformly bounded results in the measure $\PP_{\kappa;\ulin\rho,\sigma'}^{w;\ulin v,u}$ restricted to $\F_\tau$.
\end{proof}

Let $\kappa>0$, $m\in\N$, $\ulin\rho=(\rho_1,\dots,\rho_m)\in\R^m$.  Let $\eta$ be a CP SLE$_\kappa(\ulin\rho)$  at $(w;\ulin v)$, where $w\in\R$, and $\ulin v=(v_1,\dots,v_m)\in ((\R\sem \{w\})\cup \{w^+,w^-\})^m$. Let $W$ be the driving function and $V^j$, $1\le j\le m$, be the force point processes. We focus on the case that all force points lie on the same side of $w$. By symmetry we assume that $w^+\le v_1\le\cdots\le v_m $. Then $W\le V^1\le\cdots\le V^m$. Assume that $\sum_{j=1}^k \rho_j>-2$, $1\le k\le m$, and $\rho_\Sigma:=\sum_{j=1}^m\rho_j\in (\frac\kappa 2-4,\frac\kappa 2-2)$.  That $\sum_{j=1}^k \rho_j>-2$, $1\le k\le m$,  implies that  $\ha T_\eta=\infty$, and $\lim_{t\to \infty}\eta(t)=\infty$. Let $I_\infty=(v_m,\infty)$. The assumption on $\rho_\Sigma $ implies that, for any $u\in I_\infty$,  a.s.\ $u$ is not visited by $\eta$ and $\tau^W_u<\infty$, which further implies that  $\eta\cap I_\infty$ has zero Lebesgue measure and  is unbounded.

Fix $u\in I_\infty$.
At the time $\tau^W_u$, we have $V^j_{\tau^W_u}=W_{\tau^W_u}$, $1\le j\le m$. By DMP of SLE$_\kappa(\ulin\rho)$, conditionally on $\F_{\tau^W_u}$, $(W_{\tau^W_u+t}-W_{\tau^W_u})$ has the law $\PP^{0;0^+}_{\kappa;\rho_\Sigma}$. Thus,
\BGE \PP_{\kappa;\ulin\rho}^{w;\ulin v}=\K_{\tau_u}(\PP_{\kappa;\ulin\rho}^{w;\ulin v} )\oplus_{+} \PP_{\kappa;\rho_\Sigma}^{0;0^+}.\label{Pkapparho+} \EDE

Let  $\sigma'=\kappa-8-2\rho_\Sigma$.
Now we consider a random CP Loewner curve $\eta'$ following the law $\Lo(\PP^{w;\ulin v,u}_{\kappa;\ulin\rho,\sigma'})$. Let $V^j$ and $U$ be the force point processes. The assumption that $\rho_\Sigma\in (\frac\kappa 2-4,\frac\kappa 2-2)$ implies that $\sum_{j=1}^m \rho_j +\sigma'=\kappa-8-\rho_\Sigma<\frac\kappa 2-4$, which further implies that a.s.\ $\eta'$ reaches $u$ at the finite time $\tau^W_u$, at which $W_t=U_t=V^j_t$ for  all $j\in\{1,\dots,m\}$. We are not interested in the part of $\eta'$ after $\tau^W_u$ if it is not empty.
Instead, we are interested in a new probability measure defined by
\BGE \PP_{\kappa;\ulin\rho,\oplus}^{w;\ulin v,u}:=\K_{\tau _u}(\PP^{w;\ulin v,u}_{\kappa;\ulin\rho,\sigma'} )\oplus_{+} \PP_{\kappa;\rho_\Sigma}^{0;0^+}= \K_{\tau _u}(\PP^{w;\ulin v,u}_{\kappa;\ulin\rho,\kappa-8-2\rho_\Sigma} )\oplus_{+} \PP_{\kappa;\rho_\Sigma}^{0;0^+}. \label{dP-kappa-rho'-U}\EDE
This measure is supported by $\Sigma^{\Lo}$. A curve $\eta_\oplus$ following the law $\Lo(\PP_{\kappa;\ulin\rho,\oplus}^{w;\ulin v,u})$ is called a CP  SLE$_\kappa(\ulin\rho)$ curve in $\HH$ started from $w$, with force points $\ulin v$, and conditioned to pass through $u$ or simply a conditional CP SLE$_\kappa(\ulin\rho)$ at $(w;\ulin v)\to u\to\infty$.  The part of $\eta_\oplus$ before $\tau_u$ is a CP SLE$_\kappa(\ulin\rho,\sigma')$ at   $(w;\ulin v,u)$, which ends at $u$; and  the part of $\eta_\oplus$ after $\tau^W_u$  is an SLE$_\kappa(\rho_\Sigma)$  in $\HH(\eta_\oplus|(\tau^W_u)^+)$, started from $\eta_\oplus((\tau^W_u)^+)$, aimed at $\infty$, and with the force point $\eta_\oplus((\tau^W_u)^+)^+$.
If $f$ maps $\HH$ conformally onto a simply connected domain $D$, then  $[f\circ \eta_\oplus]$ is called a conditional (MTC) SLE$_\kappa(\ulin\rho)$ at $D:(f(w);{f(v_1),\dots,f(v_m)})\to f(u)\to f(\infty)$.

Let $\alpha=\alpha(\kappa,\rho_\Sigma)$ be as in (\ref{alpha}). From $\rho_\Sigma \in (\frac\kappa 2-4,\frac\kappa 2-2)$ we know that $\alpha\in(0,1)$.
By \cite[Theorem 4.1]{Green-kappa-rho}, for any $u\in I_\infty$,
  \BGE \lim_{r\to 0^+} r^{-\alpha} \PP_{w;\ulin v}^{\kappa;\ulin\rho}[\eta\cap [u-r,u+r]\ne \emptyset]=G_C(w,\ulin v;\cdot)=C_{\kappa;\rho_\Sigma} G(w,\ulin v;u),\label{C-kappa-rho}\EDE
  where $G_C(w,\ulin v;\cdot)$ is called  the Green's function (with exponent $\alpha$) for SLE$_\kappa(\ulin\rho)$ at $\HH:(w;\ulin v)\to \infty$,   $C_{\kappa;\rho_\Sigma}\in (0,\infty)$ is a constant depending only on $\kappa$ and $\rho_\Sigma$, and
   \BGE G(w,\ulin v;u):=|w-u|^{\frac{2(\kappa-8-2\rho_\Sigma)}{2\kappa}}  \prod_{j=1}^m |v_j-u|^{\frac{\rho_j(\kappa-8-2\rho_\Sigma)}{2\kappa}} . \label{Gwvu}\EDE
We write $M^{W;u}_t$ for the  $M^{W}_t$ in (\ref{MWt}) for   $\sigma=0$ and $\sigma'=\kappa-8-2\rho_\Sigma$. Then it can be expressed by
\BGE  M^{W;u}_t= (g^W_t)'(u)^\alpha G(W_t,\ulin V^W_t;g^W_t(u)),\label{Mu-express}\EDE
where $\ulin V^W_t:=(\la v_1\ra^W_t,\dots,\la v_m\ra^W_t)$.

Let $d=1-\alpha\in(0,1)$.
By  \cite[Theorem 6.17]{Green-kappa-rho}, when $W$ follows the law $ \PP_{w;\ulin v}^{\kappa;\ulin\rho}$, almost surely the Minkowski content measure  ${\cal M}_{\eta^W\cap \lin I_\infty}$  exists and is atomless, and the minimal closed support of ${\cal M}_{\eta^W\cap \lin I_\infty}$ is ${\eta^W\cap \lin I_\infty}$. The fact that   ${\cal M}_{\eta^W\cap \lin I_\infty}$ is atomless together with the definition of Minkowski content measure implies that, for any compact interval $I\subset \lin I_\infty$, $\Cont_d(\eta^W\cap I)$ exists and equals  ${\cal M}_{\eta^W\cap \lin I_\infty}(\eta^W\cap I)$. 

By \cite[Theorem 6.17 (i)]{Green-kappa-rho}, for any measurable set $S\subset I_\infty$,
\BGE  \EE_{\kappa;\rho}^{w;\ulin v}[{\cal M}_{\eta\cap \lin I_\infty}(S)]=\int_S G_C(w,\ulin v; u)\mA(du) .\label{Minkowski-Green}\EDE
 We are going to prove the following theorem.

\begin{Theorem}
We have the equality
 \BGE \Lo( \PP^{w;\ulin v}_{\kappa;\ulin\rho})(d\eta) \rotimes {\cal M}_{\eta\cap \lin I_\infty}(du)= \Lo(\PP_{\kappa;\rho,\oplus}^{w;\ulin v,u})(d\eta)\lotimes  {\ind}_{I_\infty}(u) G_C(w,\ulin v; u)\cdot \mA (du).\label{decomposition-formula}\EDE
  \label{Thm-decomposition}
\end{Theorem}

\begin{Remark}
In plain words, the theorem says that we have two ways to sample a random pair $(\eta,u)$ with the same (infinite) measure on $C([0,\infty),\lin\HH)\times I_\infty$:
  \begin{itemize}
    \item [(i)] First sample a CP SLE$_\kappa(\ulin\rho)$  curve $\eta$ at $(w;\ulin v)$, and then sample a point $u\in \eta\cap I_\infty$ according to the Minkowski content measure ${\cal M}_{\eta\cap \lin I_\infty}$.
    \item [(ii)] First sample a random point $u\in I_\infty$ with density being the Green's function for a CP SLE$_\kappa(\ulin\rho)$ at $(w;\ulin v)$, and then sample a conditional CP SLE$_\kappa(\ulin\rho)$ curve $\eta$ at $(w;\ulin v)\to u\to\infty$.
  \end{itemize}
  We call the theorem a decomposition of the CP SLE$_\kappa(\ulin\rho)$ curve $\eta$ at $(w;\ulin v)$ according to the $d$-dimensional Minkowski content measure on  $\eta\cap \lin I_\infty$.
\end{Remark}

Fix $L>0$. Let  $S=[v_m,v_{m}+L]$.
Define $\Theta^W_S(t)=\Cont_d(\eta^W[0,t]\cap S)$. From the existence of  ${\cal M}_{\eta\cap\lin I_\infty}$ under the law $\PP^{w;\ulin v}_{\kappa;\ulin\rho}$, we know that $\Theta^W_S$ is $ \PP^{w;\ulin v}_{\kappa;\ulin\rho}$-a.s.\ well defined, continuous, increasing, and adapted to  the $ \PP^{w;\ulin v}_{\kappa;\ulin\rho}$-completion of the natural filtration.
Thus,  the measure $d\Theta^W_S$ on $[0,\infty)$ determined by $\Theta^W_S$ is an $ \PP^{w;\ulin v}_{\kappa;\ulin\rho}$-kernel, and is atomless.  By (\ref{Minkowski-Green}) and (\ref{C-kappa-rho}),
$$ \EE_{\kappa;\rho}^{w;\ulin v} [\Theta^W_S(\infty)]=\EE_{\kappa;\rho}^{w;\ulin v}[{\cal M}_{\eta^W\cap \lin I_\infty}(S)]=C_{\kappa;\rho_\Sigma} \int_S G(w,\ulin v;u)\mA(du)<\infty,$$
which together with Proposition \ref{trancation-remark} implies that $\K_{d\Theta^W_S}(\PP_{\kappa;\ulin\rho}^{w;\ulin v}(dW))\tl \PP_{\kappa;\ulin\rho}^{w;\ulin v}(dW)$, and the RN process is
\BGE M^{W;S}_t:= \EE_{\kappa;\rho}^{w;\ulin v}[{\cal M}_{\eta^W\cap \lin I_\infty}(S)|\F_t]-\Theta^W_S(t).\label{MWS1}\EDE
By \cite[Theorem 6.21]{Green-kappa-rho} and (\ref{Mu-express}), we have
\begin{align}
M^{W;S}_t&=\int_S {\ind}_{[0,\tau^W_u)}(t) (g^W_t)'(u)^\alpha   G_C(W(t),\ulin V^W(t);g^W_t(u)) \mA(du) \nonumber\\
&=C_{\kappa;\rho_\Sigma}\int_S {\ind}_{[0,\tau^W_u)}(t)  M^{W;u}_t \mA(du).\label{MWS2}
\end{align}
By Lemma \ref{P/P-kappa-rho}, for each $u\in I_\infty$, $\K_{\tau^W_u}(\PP^{w;\ulin v,u}_{\kappa;\ulin\rho,\sigma'}(dW))\tl \PP^{w;\ulin v}_{\kappa;\ulin\rho}(dW)$, and the RN process is $( {\ind}_{[0,\tau^W_u)}(t)  \frac{M^{W;u}_t}{M^{W;u}_0})$. By Proposition \ref{Prop-fubini},
 $\int_S \K_{\tau^W_u}(\PP^{w;\ulin v,u}_{\kappa;\ulin\rho,\sigma'}(dW))  G_C(w,\ulin v;u)\mA(du)\tl  \PP^{w;\ulin v}_{\kappa;\ulin\rho}$,
and the RN process is
$$  \int_S {\ind}_{[0,\tau^W_u)}(t) \frac{M^{W;u}_t}{M^{W;u}_0} G_C(w,\ulin v;u)\mA(du) =C_{\kappa;\rho_\Sigma}\int_S  {\ind}_{[0,\tau^W_u)}(t)   M^{W;u}_t \mA(du)=M^{W;S}_t,\quad t\ge 0.$$
Since  $\K_{d\Theta^\cdot_S}(\PP_{\kappa;\ulin\rho}^{w;\ulin v} )$ and $\int_S \K_{\tau_u}( \PP^{w;\ulin v,u}_{\kappa;\ulin\rho,\sigma'})  G_C(w,\ulin v;u)\mA(du)$ are both locally absolutely continuous w.r.t.\ $ \PP^{w;\ulin v}_{\kappa;\ulin\rho}(dW)$, and the RN processes are the same: $(M^{\cdot;S}_t)_{t\ge 0}$, we get
$$\K_{d\Theta^\cdot_S}(\PP_{\kappa;\ulin\rho}^{w;\ulin v}  )=\int_S  \K_{\tau _u}(\PP^{w;\ulin v,u}_{\kappa;\ulin\rho,\sigma'} )  G_C(w,\ulin v;u)\mA(du).$$
Acting $\ha\oplus_{+} \PP_{\kappa;\rho_\Sigma}^{0;0^+}$ on both sides of the above equality from their right and using (\ref{dP-kappa-rho'-U}), we get
\BGE \K_{d\Theta^\cdot_S}(\PP_{\kappa;\ulin\rho}^{w;\ulin v})  \ha\oplus_{+} \PP_{\kappa;\rho_\Sigma}^{0;0^+} =  \int_S  \PP^{w;\ulin
v,u}_{\kappa;\ulin\rho,\oplus}    G_C(w,\ulin v;u)\mA(du).\label{KPOK=}\EDE
The following proposition describes the LHS of (\ref{KPOK=}).

\begin{Proposition}
  \BGE \K_{d\Theta^\cdot_S}(\PP_{\kappa;\ulin\rho}^{w;\ulin v} )\ha\oplus_{+} \PP_{\kappa;\rho_\Sigma}^{0;0^+}= \PP_{\kappa;\ulin\rho}^{w;\ulin v} (dW) \rotimes d\Theta^W_S.\label{KPOP=PO}\EDE
  \label{KPOK-Prop}
\end{Proposition}
\begin{proof}
For $x\in\R$, let $C_x([0,\infty))=\{f\in C([0,\infty)):f(0)=x\}$. Then $\PP_{\kappa;\ulin\rho}^{w;\ulin v}$ is supported by $C_w([0,\infty))$. Define two probability spaces $(\Omega_2,\PP_2)$ and $(\Omega_3,\PP_3)$ and a map $h_3:\Omega_3\to\Omega_2$ by
$$\Omega_2=C_w([0,\infty))\times [0,\infty),\quad \PP_2= \PP_{\kappa;\ulin\rho}^{w;\ulin v}(dW)\rotimes d\Theta^W_S;$$
$$\Omega_3=\Omega_2 \times C_0([0,\infty)),\quad \PP_3=\PP_2\otimes \PP^{0;0^+}_{\kappa;\rho_\Sigma};$$
$$h_3(W,t,Z)=\K_t(W)\ha \oplus_{+} Z=(\K_t(W) \oplus_{+} Z, t).$$ Then   (\ref{KPOP=PO}) becomes $h_3(\PP_3)=\PP_2$ because  $h_3(\PP_3)= \K_{d\Theta^\cdot_S}(\PP_{\kappa;\ulin\rho}^{w;\ulin v} )\ha\oplus_{+} \PP_{\kappa;\rho_\Sigma}^{0;0^+}$.

For $n\in\N$, let $x_{n;k}= v_m+\frac k{2^n} L$, $0\le k\le 2^n$. Note that $x_{n;0}=v_m$ and $x_{n;2^n}=v_m+L$.
Define $T_n:C_w([0,\infty))\times [0,\infty)\to [0,\infty)$, $h^n_2:\Omega_2\to \Omega_2$ and $h^n_3:\Omega_3\to \Omega_2$ by
$$T_n(W,t)=\left\{\begin{array}{lll} t, & \mbox{if } t<\tau^W_{x_{n;0}}\mbox{ or } t\ge  \tau^W_{x_{n;2^n}}, &\\
\tau^W_{x_{n;k-1}}, & \mbox{if } \tau^W_{x_{n;k-1}}\le t<\tau^W_{x_{n;k}}, & 1\le k\le 2^n,
\end{array}
\right.$$
$$h_2^n(W,t)=(W,T_n(W,t)),\quad h_3^n(W,t,Z)=h_3(W,T_n(W,t),Z)
. $$
Let $W\in C_w([0,\infty))$ and $t_0\in[0,\infty)$, Then the limit $t_0':=\lim T_n(W,t_0)$ exists in $[0,t_0]$. We claim that $s\mapsto b^W_s$ is constant on $[t_0',t_0)$ if $ t_0'<t_0$. From $t_0'<t_0$ we get $\tau^W_{v_m}\le t_0'<t_0<\tau^W_{v_m+L}$, and so $v_m\le b^W_{t_0'}\le b^W_{t_0}< v_m+L$.  If the claim is not true, then there is $s_0\in(t_0',t_0)$ such that $b^W_{s_0}>b^W_{t_0'}$. There is $n_0\in\N$ such that for some $1\le k\le 2^{n_0}$, $b^W_{t_0'}<x_{n_0;k-1}\le b^W_{s_0}<x_{n_0;k}$. Then $t_0'<\tau^W_{x_{n_0;k-1}}\le s_0<\tau^W_{x_{n_0;k}}$. So we have $T_{n_0}(W,t_0)\ge T_{n_0}(W,s_0)=\tau^W_{x_{n_0;k-1}}>t_0'$, which contradicts that $t_0'$ is the increasing limit of $T_n(W,t_0)$. So the claim is proved. If $W$ generates a CP Loewner curve $\eta$, then no matter whether $t_0'=t_0$ or $t_0'<t_0$, the arc $\eta(t_0',t_0)$ does not intersect $(b^W_{t_0'},\infty)$, and so $\Theta^W_S$ is constant on $[t_0',t_0]$.

Suppose $W$ follows the law $\PP_{\kappa;\ulin\rho}^{w;\ulin v}$ and let $\eta=\Lo(W)$.
Let $F_W$ be the minimal closed support of $d\Theta^W_S$, and let $\Xi_W$ be the collection of connected components of $(w,\infty)\sem F_W$. If $t_0'<t_0$, then $(t_0',t_0)\cap F_W=\emptyset$ because $d\Theta^W_S$ has no mass on $(t_0',t_0)$, which implies that $t_0\in \bigcup_{I\in \Xi_W} \lin I$. Since $d\Theta^W_S$ has no point mass and $\Xi_W$ is countable, we conclude that the set  of $t_0$ such that $t_0'=\lim T_n(W,t_0)<t_0$ has zero $d\Theta^W_S$ mass. Hence for $\PP_2=\PP_{\kappa;\ulin\rho}^{w;\ulin v}\rotimes d\Theta^W_S$-a.s.\ every $(W,t)$, $T_n(W,t)\to t$. This implies that $h^n_2(\PP_2)\wto \PP_2$. Moreover for $\PP_2$-a.s.\ every $(W,t)$ and every $Z\in C_0([0,\infty))$, $h_3^n(W,t,Z) =h_3(h_2^n(W,t),Z)  \to  h_3(W,t,Z)$.
Thus, $\PP_3$-a.s.\ $h_3^n\to h_3$, which implies that  $h^n_3(\PP_3)\wto h_3(\PP_3)$.

Since $h^n_2(\PP_2)\wto \PP_2$ and $h^n_3(\PP_3)\wto h_3(\PP_3)$, in order to prove that  $h_3(\PP_3)=\PP_2$, it suffices to show that  $h_3^n(\PP_3)=h_2^n(\PP_2)$ for every $n\in\N$.
Fix $n\in\N$. From (\ref{Pkapparho+}), for any $1\le k\le 2^n$,
$$ \PP_{\kappa;\ulin\rho}^{w;\ulin v}(dW) \rotimes  \delta_{\tau^W_{x_{n;k-1}}}=  \K_{ \delta_ {\tau _{x_{n;k-1}}}}(\PP_{\kappa;\ulin\rho}^{w;\ulin v} )\ha\oplus_{+} \PP_{\kappa;\rho_\Sigma}^{0;0^+}.$$ 
Since $\Theta^\cdot_S(\tau_{x_{n;k}})-\Theta^\cdot_S(\tau_{x_{n;k-1}})=\Cont_d(\eta\cap [x_{n;k-1},x_{n;k}])$ and $\Theta^\cdot_S$ is constant on $[0,\tau_{x_{n;0}}]$ and $[\tau_{x_{n;2^n}},\infty)$, we have
\begin{align*}
h_2^n(\PP_2)&= \sum_{k=1}^{2^n} \PP_{\kappa;\ulin\rho}^{w;\ulin v}(dW) \rotimes \Cont_d(\eta^W\cap [x_{n;k-1},x_{n;k}])\delta_{\tau^W_{x_{n;k-1}}};\\
h_3^n(\PP_3)&=\sum_{k=1}^{2^n} \Cont_d(\eta^W\cap [x_{n;k-1},x_{n;k}])  \K_{  \delta_ {\tau^W_{x_{n;k-1}}}}(\PP_{\kappa;\ulin\rho}^{w;\ulin v}(dW))\ha\oplus_{+} \PP_{\kappa;\rho_\Sigma}^{0;0^+} .
\end{align*}
The above three displayed formulas together imply that  $h_3^n(\PP_3)=h_2^n(\PP_2)$.
\end{proof}

\begin{proof} [Proof of Theorem \ref{Thm-decomposition}]
To prove (\ref{decomposition-formula}), it suffices to show that for any $S=[v_m,v_m+L]$, $L>0$,
\BGE \Lo(\PP_{\kappa;\ulin\rho}^{w;\ulin v})(d\eta)\rotimes \ind_S{\cal M}_{\eta\cap \lin I_\infty} (du) = \Lo(\PP_{\kappa;\ulin\rho,\oplus}^{w;\ulin v,u})(d\eta)\lotimes {\ind}_S G_C(w,\ulin v;u) \mA(du). \label{decomposition-formula-S}\EDE
By (\ref{KPOK=},\ref{KPOP=PO}), we have
\BGE \PP_{\kappa;\ulin\rho}^{w;\ulin v}(dW) \rotimes d\Theta^W_S= \int_S   \PP^{w;\ulin
v,u}_{\kappa;\ulin\rho,\oplus}   G_C(w,\ulin v;u)\mA(du).\label{==}\EDE
Act   $\ha \Lo:(W,t)\mapsto (\Lo(W),\Lo(W)(t))$ on both sides of (\ref{==}) from their left. The lefthand side becomes
$$\ha\Lo(\PP_{\kappa;\ulin\rho}^{w;\ulin v}(dW) \rotimes d\Theta^W_S)=\Lo(\PP_{\kappa;\ulin\rho}^{w;\ulin v})(d\eta)\rotimes{\cal M}_{\eta\cap S}(du) ,$$
where we used the fact that,  $\eta^W(d\Theta^W_S)={\cal M}_{\eta^W\cap S}$.
The righthand side of (\ref{==}) becomes
\begin{align*}
  \int_S \ha\Lo(\PP_{\kappa;\ulin\rho,\oplus}^{w;\ulin v,u})  G_C(w,\ulin v;u)\mA(du)
=&\int_S (\Lo(\PP_{\kappa;\ulin\rho,\oplus}^{w;\ulin v,u})(d\eta)\otimes \delta_u ) G_C(w,\ulin v;u)\mA(du) \\
=&\Lo(\PP_{\kappa;\ulin\rho,\oplus}^{w;\ulin v,u})(d\eta)\lotimes {\ind}_S G_C(w,\ulin v;u) \mA(du),
\end{align*}
where we used the  fact that under $\Lo(\PP_{\kappa;\ulin\rho,\oplus}^{w;\ulin v,u})$, $\eta$ hits $u$ at $\tau_u$.
Since $  {\cal M}_{\eta\cap S}=\ind_S{\cal M}_{\eta\cap \lin I_\infty} $, from (\ref{==}) and the above two formulas,  we get the desired (\ref{decomposition-formula-S}). 
\end{proof}

\subsection{Two-marked-point SLE$_\kappa(\rho)$ loops} \label{section-two-marked}
In this subsection, let $\kappa\in(0,8)$ and $\rho> (-2)\vee (\frac\kappa 2-4)$.
 We are going to use the following symbols.
Let $\til\mu^D_{(w_0;v)\to w_\infty}$ denote the law of an  SLE$_\kappa(\rho)$ at $D:(w_0;v)\to \infty$.   Let $\til\nu^D_{(w_0;v)\to u\to w_\infty}$ denote the law of a conditional  SLE$_\kappa(\rho)$ at  $D:(w_0;v)\to u\to w_\infty$.
Let $\til\nu^D_{(w_0;v_1,v_2)\to w_\infty}$ denote  the law of an SLE$_\kappa(\rho,\rho+2)$ at $D:(w_0;v_1,v_2)\to w_\infty$. We write  $\til\mu^D_{w_0\to w_\infty}$ for $\til\mu^D_{(w_0;w_0^+)\to w_\infty}$, $\til\nu^D_{w_0\to u\to w_\infty}$ for $\til\nu^D_{(w_0;w_0^+)\to u\to w_\infty}$, and $\til\nu^D_{w_0\to w_\infty}$ for $\til\nu^D_{(w_0;w_0^+,w_0^-)\to w_\infty}$.

Recall that if $\eta$ is a conditional CP SLE$_\kappa(\rho)$ at $(w;v)\to u\to\infty $, then $\eta|_{[0,\tau_u)}$ is a CP SLE$_\kappa(\rho,\kappa-8-2\rho)$ at $(w;v,u)$ stopped at the time $\tau_u$, at which $\eta$ hits $u$, and conditionally on $\eta|_{[0,\tau_u)}$, the part of $\eta$ after $\tau_u$, modulo time-change, has the law of $\til \mu^{\HH(\eta[0,\tau_u);\infty)}_{u\to \infty}$, where the symbol $u$ in $\til \mu^{\HH(\eta[0,\tau_u);\infty)}_{u\to \infty}$ denotes the prime end of $\HH(\eta[0,\tau_u);\infty)$ that is identified with the point $u\in\R$. By SLE coordinate changes, $[\eta|_{[0,\tau_u)}]$ has the law of $\til \nu^{\HH}_{(w;v,\infty)\to u}$. So we have
$$\til\nu^{\HH}_{(w;v)\to u\to \infty} =\til \nu^{\HH}_{(w;v,\infty)\to u}(d[\eta]) \oplus \til \mu^{\HH([\eta];\infty)}_{u \to \infty}.$$
Thus, if $D$ is a simply connected domain in $\ha\C$, and $w_0,w_\infty$ are distinct prime ends of $D$, $u\in (w_0,w_\infty)_{\ha\pa D}$, and   $v\in [w_0^+,u)_{\ha\pa D}$, then
\BGE \til\nu^D_{(w_0;v)\to u\to w_\infty} =\til \nu^D_{(w_0;v,w_\infty)\to u}(d[\eta]) \oplus \til \mu^{D([\eta] ;w_\infty)}_{u\to w_\infty}.
\label{DMP-conditional}\EDE

By the results of \cite{MS1} and \cite{GFF-level},  there is a coupling of two random CP Loewner curves $\eta_+,\eta_-$  in $\lin{\HH}^\#$ from $0$ to $\infty$ such that
\begin{itemize}
  \item [(I)] $\eta_\pm$ is a CP SLE$_\kappa(\rho,\rho+2)$ at $(0;0^\pm,0^\mp)$, which does not intersect $\R_\mp$;
\item[(II)] given $\eta_\pm$, $[\eta_\mp]$ is an (MTC) SLE$_\kappa (\rho)$ at $\HH(\eta_\pm;\R_\mp):(0;0^\mp)\to \infty$, where $0$ and $\infty$ are identified with two prime ends of the Jordan domain $\HH(\eta_\pm;\R_\mp)$.
\end{itemize}
Here when $\kappa\ne 4$, the $\eta_+$ and $\eta_-$ could be constructed as flow lines or counter-flow lines of a Gaussian free field (GFF) with some boundary values (\cite{MS1}); and when $\kappa=4$,   $\eta_+$ and $\eta_-$ could be constructed as level lines of a GFF with some boundary values (\cite{GFF-level}).

Moreover, when $\rho\ge \frac\kappa 2$, a.s.\ $(\eta_+\cup\eta_-)\cap\R=\{0\}$, and when $\rho\in ((-2)\vee (\frac\kappa 2-4),\frac\kappa 2-2)$, a.s.\ $\eta_\pm\cap\R_\pm$ is nonempty and unbounded. In both cases, for any deterministic point $x_0\in\R\sem\{0\}$, a.s.\ $x_0\not\in \eta_+\cup\eta_-$. Both $\eta_+$ and $\eta_-$ satisfy the scaling property: for any $a>0$, $[a\eta_\pm]$ has the same law as $[\eta]$,
and so the set $a\cdot(\eta_\pm\cap \R_\pm)$ has the same law as $\eta_\pm\cap\R_\pm$.

Recall the time-reversal map $\cal R$ defined in Section \ref{section-lifetime}. Let $[\eta_\pm^R]={\cal R}([\eta_\pm])$.
 By (I,II) and the reversibility of SLE$_\kappa(\rho)$ and SLE$_\kappa(\rho_+,\rho_-)$   (\cite[Theorem 1.1]{MS2}, \cite[Theorem 1.2]{MS3}, \cite[Theorem 1.1.6]{GFF-level}),  we know that
 \begin{itemize}
   \item [(III)] $[\eta_-^R]$ is an SLE$_\kappa(\rho,\rho+2)$ at $\HH:(\infty ; \infty^+,\infty^-)\to 0$;
  \item [(IV)]  given $[\eta_+]$, $ [\eta_-^R]$ is an SLE$_\kappa(\rho)$ at $\HH([\eta_+];\R_{-}):(\infty;\infty^+)\to 0$.
\end{itemize}

 We call the loop $[\eta_+]\oplus   [\eta_-^R]$ (resp.\ $[\eta_-]\oplus  [\eta_+^R]$)   a ccw (resp.\ cw)   conditional rooted SLE$_\kappa(\rho)$   loop at $\HH:0\lr\infty$, and denote its law by $\til\nu^{\HH}_{0\lr \infty}$ (resp.\ $\til\nu^{\HH}_{0\rl\infty}$).  
 Using conformal maps, we then define ccw (resp.\ cw) conditional rooted SLE$_\kappa(\rho)$   loop measure $\til\nu^{D}_{w_0\lr w_\infty}$ (resp.\ $\til\nu^{D}_{w_0\lr w_\infty}$) for any simply connected domain $D$ and its two distinct prime ends $w_0,w_\infty$. 
By (I,IV) we have
 \BGE \til\nu^{\HH}_{0\lr \infty} = \til\nu^{\HH}_{0\to \infty}(d[\eta_+])\oplus \til\mu^{\HH(\eta_+;\R_-)}_{\infty\to 0} .\label{DMP-two-loop-0}\EDE

 By (III,II) , $[\eta_-^R]\oplus[\eta_+]$ is  a ccw  conditional rooted SLE$_\kappa(\rho)$   loop at $\HH:\infty\lr 0$. Note that $[\eta_+]\oplus [\eta_-^R]$ and $[\eta_-^R]\oplus[\eta_+]$ are translation equivalent. So we have $\til\pi( \til\nu^{\HH}_{0\lr\infty})= \til\pi (\til\nu^{\HH}_{\infty\lr 0})$. Recall that $\til\pi:[\eta]\mapsto [[\eta]]$ is the projection from rooted MTC loops to unrooted MTC loops. Thus, for any  simply connected domain $D$ with  $w_0\ne w_\infty\in\ha \pa D$, we have
 \BGE \til\pi(\til\nu^D_{w_0\lr w_\infty})=\til\pi(\til\nu^D_{w_\infty\lr w_0}).\label{unrooted=}\EDE

For $u\in I^{\gamma}_{s_0^+;D}$ (resp.\ $u\in I^{[\gamma]}_{\sigma^+;D}$), where $s_0\ge 0$ and $\sigma$ is a TCI time, we define
 \BGE \til\nu^{D;u}_{\gamma|s_0^+}=\til\nu^{D(\gamma|s_0^+)}_{(\gamma(s_0^+);[\gamma(0)^+]^{\gamma}_{s_0^+;D})\to u\to  [\gamma(0)^-]^{\gamma}_{s_0^+;D}},\quad \til\nu^{D;u}_{[\gamma]|\sigma^+}=\til\nu^{D;u}_{\gamma|\sigma(\gamma)^+}
 .\label{til-mu-gamma-u}\EDE

\begin{Lemma}
Let $\sigma$ be a positive TCI stopping time. Let $\sigma_\infty$ be the first hitting time at $\infty$. Then
  \BGE \ind_{\{\sigma< \sigma_\infty\}}  \til\nu^{\HH}_{0\lr \infty}= \ind_{ \til\Sigma_{\sigma}}  \til\nu^{\HH}_{0\to \infty}(d[\gamma])\oplus_\sigma  \til \nu^{\HH;\infty}_{[\gamma]|\sigma^+}.\label{DMP5'}\EDE
  \label{lemma-4.5}
\end{Lemma}
\begin{proof}  By the DMP of the SLE$_\kappa(\rho,\rho+2)$, we have
\BGE \ind_{\til\Sigma_{\sigma}} \til \nu^{\HH}_{0\to \infty} =  \ind_{\til\Sigma_{\sigma}}  \til \nu^{\HH}_{0\to \infty}(d[\gamma])\oplus_\sigma \til\nu_{[\gamma];\sigma}.\label{DMP1'}\EDE
where $\til\nu_{[\gamma];\sigma}:=\til \nu^{\HH([\gamma]|\sigma^+)}_{([\gamma](\sigma^+);[0^+]^{[\gamma]}_{\sigma^+;\HH},[0^-]^{[\gamma]}_{\sigma^+;\HH})\to \infty}$.
Fix an MTC curve $[\gamma]$ following the law  $\til \nu^{\HH}_{0\to \infty}$.
 Since $\sigma>0$ and $\rho+2\ge \frac\kappa 2-2$, we have $[\gamma](\sigma^+) \ne [0^-]^{[\gamma]}_{\sigma^+;\HH}$, $[0^-]^{[\gamma]}_{\sigma^+;\HH}=\ha a^{[\gamma]}_{\sigma^+;\HH}$, and $a^{[\gamma]}_{\sigma^+;\HH}=\gamma(0)=0$.
By  (\ref{DMP-conditional}),
\BGE \til\nu_{[\gamma];\sigma} (d[\eta]) \oplus \til\mu^{\HH([\gamma]|\sigma^+)([\eta ];[0^-]^{[\eta]}_{\sigma^+;\HH})}_{\infty\to [0^-]^{[\gamma]}_{\sigma^+;\HH}}  = \til \nu^{\HH;\infty}_{[\gamma]|\sigma^+} .\label{DMP3'}\EDE
Suppose $[\eta]$ follows the conditional law $\til\nu_{[\gamma];\sigma}$. Then $[\gamma]\oplus_\sigma[\eta]$ follows the law $\til\nu^{\HH}_{0\to \infty}$, and so is disjoint from $\R_-$. Moreover, $0$ is identified with a prime end of $\HH([\gamma]\oplus_\sigma[\eta];\R_-)$.
Recall that  the prime end $[0^-]^{[\gamma]}_{\sigma^+;\HH}=\ha a^{[\gamma]}_{\sigma^+;\HH}$   of $\HH([\gamma]|\sigma^+)$   determines the point  $ a^{[\gamma]}_{\sigma^+;\HH}=0$. Since the prime end $[0^-]^{[\gamma]}_{\sigma^+;\HH}$ is shared by $\HH([\gamma]|\sigma^+)$ and $\HH([\gamma]|\sigma^+)([\eta]; [0^-]^{[\gamma]}_{\sigma^+;\HH})$, $\R_-$ is also shared by the two domains.
Thus, $\HH([\gamma]|\sigma^+)([\eta];[0^-]^{[\gamma]}_{\sigma^+;\HH}) 
=\HH([\gamma]\oplus_\sigma[\eta];\R_-)$.

Since $\HH([\gamma]|\sigma^+)$ and $ \HH([\gamma]\oplus_\sigma[\eta];\R_-)$ share the prime end $[0^-]^{[\gamma]}_{\sigma^+;\HH}$, this prime end  also determines $0$. We already know that $0$ is determined by a prime end of  $\HH([\gamma]\oplus_\sigma[\eta];\R_-)$. Thus, such prime end must be $[0^-]^{[\gamma]}_{\sigma^+;\HH}$. Hence, we may rewrite (\ref{DMP3'}) as
 \BGE \til\nu_{[\gamma];\sigma} (d[\eta]) \oplus \til\mu^{\HH([\gamma]\oplus_\sigma [\eta];\R_-)}_{\infty\to 0} = \til \nu^{\HH;\infty}_{[\gamma]|\sigma^+} .\label{DMP4}\EDE
Acting $\til\nu^{\HH}_{0\to\infty}(d[\gamma])\oplus_\sigma$ on both sides of (\ref{DMP4}) and using  (\ref{DMP-two-loop-0},\ref{DMP1'}),   we get (\ref{DMP5'}).
 \end{proof}

\begin{Corollary}
$ \til\nu^{\HH}_{0\lr \infty}$   is supported by the space of Loewner bubbles in $\HH$ rooted at $0$.
\end{Corollary}
\begin{proof}
  This follows immediately from  Lemma \ref{lemma-4.5} applied to $\sigma=\sigma_1$, the hitting time at $\{|z|=1\}$, and    Lemma \ref{continue-Loewner} applied to $s_0=\sigma_1(\gamma)$.
\end{proof}

Let $u\in\R\sem \{0\}$ and $\sigma_u$ be the first hitting time at $u$. By the above results and conformal invariance,  $\til\nu^{\HH}_{0\lr u}$ is supported by the space of Loewner bubbles in $\HH$ rooted at $0$ and passing through $u$, and for any positive TCI stopping time $\sigma$,
  \BGE \ind_{\{\sigma< \sigma_u\}}  \til\nu^{\HH}_{0\lr u}= \ind_{\til\Sigma_{\sigma}}  \til\nu^{\HH}_{0\to u}(d[\gamma])\oplus_\sigma \til \nu^{\HH;u}_{[\gamma]|\tau^+}.\label{DMP-u}\EDE

\subsection{Decomposition theorems} \label{section-proof}
From now on, let $\kappa\in (0,8)$ and $\rho\in((-2)\vee(\frac\kappa 2-4),\frac\kappa 2-2)$.  Let $\alpha $ be as in (\ref{alpha}) and $d=1-\alpha$. Then $\alpha,d\in(0,1)$.
Let   $C_{\kappa;\rho}>0$ be the constant in (\ref{C-kappa-rho}) with $\rho$ in place of $\rho_\Sigma$. 

\begin{Theorem} Let $\til\mu^\HH_{0\ccw}$ be the rooted ccw SLE$_\kappa(\rho)$ bubble measure  in Theorem \ref{Thm-loop-sigma-finite}. Then
\begin{itemize}
  \item [(i)] $\til\mu^\HH_{0\ccw}$ is supported by the set of Loewner bubbles $[\eta]$ in $\HH$ rooted at $0$ such that
  \begin{itemize}
    \item [(a)] the $d$-dimensional Minkowski content measure  ${\cal M}_{[\eta]\cap\R}$ exists;
    \item [(b)] the minimal closed support of ${\cal M}_{[\eta]\cap\R}$ is $[\eta]\cap\R$, which is compact; and
    \item [(c)]  $\Cont_d([\eta]\cap\R) $ exists and lies in $(0,\infty)$.
  \end{itemize}
   \item [(ii)] The following equalities hold:
\BGE \til\mu^{\HH}_{0\ccw }(d[\eta]) \rotimes {\cal M}_{[\eta] \cap \R}(d u) = \til \nu^{\HH}_{0\lr u} (d[\eta]) \lotimes {\ind}_{\R\sem\{0\}}  C_{\kappa;\rho} |u|^{-2\alpha}\cdot \mA(d  u) ,\label{mu-decomp-R}\EDE
\BGE \til\mu ^{\HH}_{0\ccw} =\Cont_d(\cdot\cap\R)^{-1}\cdot \int_{\R\sem\{0\}} \til\nu^{\HH}_{0\lr u}  C_{\kappa;\rho} |u|^{-2\alpha} \mA(du).\label{express-rooted-loop}\EDE
\end{itemize}
\label{Thm-decomposition-loop}
\end{Theorem}
\begin{proof}
Let $\nu^{\HH}_{0\to \infty}$ be  the law of a CP SLE$_\kappa(\rho,\rho+2)$ at $(0;0^+,0^-)$. Then $\pi(\nu^{\HH}_{0\to \infty})=\til \nu^{\HH}_{0\to \infty}$. For $u\in\R\sem\{0\}$, let $\nu^{\HH}_{0\to u}$ denote the law of a CP SLE$_\kappa(\rho,\rho+2,\kappa-8-2\rho)$ at $(0;0^+,0^-,u)$   killed at the time $\tau_u$. By SLE coordinate changes,  $\pi(\nu^{\HH}_{0\to u})=\K_{\tau_u}(\til \nu^{\HH}_{0\to u})$.

Fix a CP Loewner curve $\gamma$ sampled from $\nu^{\HH}_{0\to \infty}$. Then $\ha T_\gamma=\infty$, $\gamma\cap \R_-=\emptyset$,  $a^\gamma_t =0$ for all $t\ge 0$, and $C^\gamma_t<W^\gamma_t$ for all $t>0$. Let $\iota=\frac{2\rho+8-\kappa}{2\kappa}>0$. Then $\alpha=(\rho+2)\iota$.
For  $u\in\R\sem\{0\}$ and $0<t< \tau_u$, we define
\begin{align*}
   R^u_{t}(\gamma)= & |u|^{2\alpha}  ( g^\gamma_{t})'(u) ^{\alpha}|g^\gamma_{t}(u)  -C^\gamma_t |^{-\alpha}  |g^\gamma_{t}(u)  -D^\gamma_t |^{-\rho\iota} |g^\gamma_{t}(u) -W^\gamma_t |^{{-2\iota}},
   \\
    G^\gamma_{t}(u)
 =& C_{\kappa;\rho}\Big(\frac{(g^\gamma_{t})'(u)}{|g^\gamma_{t}(u)-C^\gamma_t|}\Big)^\alpha  \Big(\frac{|D^\gamma_t-C^\gamma_t|}{|g^\gamma_{t}(u)-D^\gamma_t|}\Big)^{ \rho\iota} \Big( \frac{|W^\gamma_t-C^\gamma_t|} {|g^\gamma_{t}(u)-W^\gamma_t|}\Big)^{2\iota}
\end{align*}
Let $Q^\gamma_t$ be as defined by (\ref{Q-eta-t}) for $t>0$.
Then we have
\BGE Q^\gamma_t G^\gamma_{t}(u) =  C_{\kappa;\rho} |u|^{-2\alpha} R^u_{t}( \gamma ).\label{QGGR}\EDE

Let $\tau$ be a positive finite random time.
Let $F_{\gamma;\tau}(z)=\frac{z-W^\gamma_\tau}{z-C^\gamma_\tau}$ be a M\"obius automorphism of $\HH$, which sends $W^\gamma_\tau$, $C^\gamma_\tau$ and $D^\gamma_\tau$ respectively to $0$, $\infty$ and $X^\gamma_\tau:=\frac{D^\gamma_\tau-W^\gamma_\tau}{D^\gamma_\tau-C^\gamma_\tau}\in[0,1)$.
Let $G_C(w,v;u)$ be as defined in (\ref{C-kappa-rho}) for $m=1$, $v_1=v$ and $\rho_1=\rho$, and $\check G^\gamma_{\tau}(\check u)=G_C(0, X^\gamma_\tau;\check u)$. Let $\check X^\gamma_\tau=0^+\vee X^\gamma_\tau$.
By  \cite[Theorem 6.17]{Green-kappa-rho}, if $[\check\eta]$ follows the conditional law $\til \mu^{\HH}_{(0;\check X^\gamma_\tau)\to \infty}$, then a.s.\  ${\cal M}_{[\check\eta]\cap [X^\gamma_\tau,\infty)}$  exists and is atomless, and its minimal closed support is $[\check\eta]\cap [X^\gamma_\tau,\infty)$.
Moreover, by Theorem \ref{Thm-decomposition}, we have the decomposition formula
\BGE \til \mu^{\HH}_{(0;\check X^\gamma_\tau)\to \infty}(d[\check\eta]) \rotimes {\cal M}_{[\check\eta]\cap [X^\gamma_\tau,\infty)}(d\check u)= \til \nu^{\HH}_{(0;\check X^\gamma_\tau)\to \check u\to \infty}(d[\check\eta]) \lotimes {\ind}_{(X^\gamma_\tau,\infty)} \check G^\gamma_{\tau}(\check u)\cdot \mA(d\check u). \label{decomposition-formula-eta1}\EDE

Let $h^\gamma_{\tau}=(g^\gamma_\tau)^{-1}\circ F_{\gamma;\tau}^{-1}$.
Since $h^\gamma_{\tau}$ maps $\HH$ conformally onto $H^\gamma_\tau$, and sends $0,\check X^\gamma_\tau,1,\infty$ respectively to the prime ends $\gamma(\tau^+),[0^+]^\gamma_\tau,\infty,[0^-]^\gamma_\tau$ of $H^\gamma_\tau$, if $[\check\eta]$ follows the law $\til \mu^{\HH}_{(0;\check X^\gamma_\tau)\to \infty}$, then $ [h^\gamma_{\tau}\circ \check\eta]$  follows the law  $\til \mu^{\HH}_{\gamma|\tau^+}$ (defined in (\ref{til-mu-gamma},\ref{til-mu-gamma-[]})),  and by  the  existence  of ${\cal M}_{[\check\eta]\cap [X^\gamma_\tau,\infty)}$ and Proposition \ref{Prop-Minkowski} (iii,iv), a.s.\ ${\cal M}_{[h^\gamma_{\tau}\circ \check\eta]\cap (\R\sem [0,b^\gamma_\tau])}$ exists and satisfies
$$h^\gamma_{\tau}((h^\gamma_{\tau})'(\check u)^d \cdot {\cal M}_{[\check \eta] \cap [X^\gamma_\tau,\infty) \sem\{1\}}(d\check u))={\cal M}_{[h^\gamma_{\tau}\circ \check\eta]\cap (\R\sem [0,b^\gamma_t])}(du) .$$
Thus, if $[\eta]$ follows the law  $\til \mu^{\HH}_{\gamma|\tau^+}$, then ${\cal M}_{[\eta]\cap (\R\sem [0,b^\gamma_\tau])}$ a.s.\ exists and is atomless, and its minimal closed support is $\lin{[\eta]\cap (\R\sem [0,b^\gamma_\tau])}$.
It is straightforward to calculate
$ (h^\gamma_{\tau})'(\check u)^\alpha G^\gamma_{\tau}(h^\gamma_{\tau}(\check u))=\check G^\gamma_{\tau}(\check u)$,
which together with $1-\alpha=d$ implies that
$$h^{\gamma}_{\tau}((h^\gamma_{\tau})'(\check u)^d {\ind}_{(X^\gamma_\tau,\infty)\sem\{1\}}    \check G^\gamma_{\tau}(\check u) \cdot \mA(d\check u)) ={\ind}_{\R\sem [0,b^\gamma_\tau]}  G^\gamma_{\tau}(u)\cdot \mA(du),$$
where we used a change of variable: $u=h^\gamma_{\tau}(\check u)$.
Weighting the measures on both sides of (\ref{decomposition-formula-eta1}) by $(h^\gamma_{\tau})'(\check u)^d$,  then applying the map $([\check\eta],\check u)\mapsto ([h^\gamma_{\tau}\circ \check\eta],h^\gamma_{\tau}(\check u))=([\eta],u)$, and using the above two displayed formulas and symbols defined in (\ref{til-mu-gamma},\ref{til-mu-gamma-[]},\ref{til-mu-gamma-u}), we get
$$ \til \mu^{\HH}_{\gamma|\tau^+} (d[\eta]) \rotimes {\cal M}_{[\eta]\cap (\R\sem [0,b^\gamma_\tau])}(d u)   =  \til \nu^{\HH;u}_{\gamma|\tau^+}(d[\eta]) \lotimes {\ind}_{\R\sem [0,b^\gamma_\tau]} G^\gamma_{\tau}( u)\cdot \mA(d  u). $$
Multiplying both sides of the above  formula by $Q^\gamma_\tau$ and using   (\ref{QGGR}), we get
 \BGE Q^\gamma_\tau \til \mu^{\HH}_{\gamma|\tau^+} (d[\eta]) \rotimes {\cal M}_{[\eta]\cap (\R\sem [0,b^\gamma_\tau])}(d u)  =  R^u_{\tau}(\gamma) \til \nu^{\HH;u}_{\gamma|\tau^+}(d[\eta]) \lotimes {\ind}_{\R\sem [0,b^\gamma_\tau]} C_{\kappa;\rho} |u|^{-2\alpha}\cdot \mA(d  u). \label{decomposition-formula-eta3}\EDE

Let $u\in \R\sem\{0\}$. Let $\sigma_u$ be the first hitting time at  $u$.
Since $\pi(\nu^{\HH}_{0\to u})=\K_{\sigma _u}(\til \nu^{\HH}_{0\to u})$,  by (\ref{DMP-u}), for any positive stopping time $\sigma$,
 \BGE \ind_{\{\sigma< \sigma_u\}}  \til\nu^{\HH}_{0\lr u}= \ind_{ \{\sigma<\tau_u\}}  \nu^{\HH}_{0\to u}(d\gamma)\oplus_\sigma  \til \nu^{\HH;u}_{\gamma|\sigma^+}.\label{DMP-u-CP}\EDE
Let $K$ be any compact subset of $\R\sem \{0\}$.
 Let $J$ be a crosscut of $\HH$, which encloses a neighborhood of $0$ in $\HH$, and disconnects $K$ from $0$.  Let $\sigma_J$ be the hitting time at $J$.
 Then for any $u\in K$, $ \til\nu^{\HH}_{0\lr u}$ is supported by $\{\sigma_J<\sigma_u\}$, and $ \nu^{\HH}_{0\to u}$ is supported by $\{\sigma_J<\tau_u\}$. From (\ref{DMP-u-CP}) we get
  \BGE   \til\nu^{\HH}_{0\lr u}=  \nu^{\HH}_{0\to u}(d\gamma)\oplus_{\sigma_J}  \til \nu^{\HH;u}_{\gamma|\sigma_J^+},\quad u\in K.\label{DMP-u-CP-KJ}\EDE

 By Lemma \ref{P/P-kappa-rho}, for any $u\in K$, $\nu^{\HH}_{0\to u}(dW) = R^u_{\sigma_J}(W)\cdot\nu^{\HH}_{0\to \infty}(dW)$ on $\F_{\sigma_J} $, where we used the fact that both measures are supported by $\{\sigma_J<\ha T\wedge \tau_u\}$. Restricting (\ref{decomposition-formula-eta3}) to $u\in K$ and then acting $\nu^{\HH}_{0\to \infty}(d\gamma)\oplus_{\sigma_J}$ on both sides, we get
\begin{align*}
&(Q^\gamma_{\sigma_J}\cdot \nu^{\HH}_{0\to \infty}(d\gamma)\oplus_{\sigma_J}  \til \mu^{\HH}_{\gamma|\sigma_J^+}  (d[\eta]) )\rotimes {\cal M}_{(\gamma\oplus_{\sigma_J} [\eta])\cap K}(d u)\\
 \stackrel{\phantom{(\ref{Markov-mu-R})}}{=}& ( \nu^{\HH}_{0\to u}(d\gamma)\oplus_{\sigma_J} \til \nu^{\HH;u}_{\gamma|\sigma_J^+} (d[\eta])) \lotimes {\ind}_{K}(u) C_{\kappa;\rho} |u|^{-2\alpha}\cdot \mA(d  u)\\
 \stackrel{{(\ref{DMP-u-CP-KJ})}}{=}& \til \nu^{\HH}_{0\lr u} (d(\gamma\oplus_{\sigma_J}[\eta])) \lotimes {\ind}_{K}(u) C_{\kappa;\rho} |u|^{-2\alpha}\cdot \mA(d  u),
\end{align*}
where in the first equality we used the fact  $\K_{\sigma_J}(\gamma)\cap K=\emptyset$ when $\gamma$ follows the law $\nu^{\HH}_{0\to\infty}$. 

Note that the $ \nu^{\HH}_{0\to \infty}$ here agrees with the $\nu^*$ defined before Theorem \ref{Thm-loop-sigma-finite}.
By  Theorem \ref{Thm-loop-sigma-finite} (i,ii),
$ \til\mu^{\HH}_{0\ccw} =Q^\gamma_{\sigma_J}\cdot \nu^{\HH}_{0\to \infty}(d\gamma)\oplus_{\sigma_J}  \til \mu^{\HH}_{\gamma|\sigma_J^+}$.
Thus,
we may rewrite the last displayed formula as
\BGE \til\mu^{\HH}_{0\ccw} (d[\eta]) \rotimes {\cal M}_{[\eta]\cap K}(d u) =  \til \nu^{\HH}_{0\lr u} (d[\eta]) \lotimes {\ind}_{K} C_{\kappa;\rho} |u|^{-2\alpha}\cdot \mA(d  u).\label{decomposition-loop-K}\EDE

Since $\til\mu^{\HH}_{0\ccw}$-a.s.\  ${\cal M}_{[\eta]\cap K}$ exists for any compact set $K\subset\R\sem\{0\}$, by Lemma \ref{Mink-exist-S12} we have the existence of  ${\cal M}_{[\eta]\cap(\R\sem\{0\})}$, whose restriction to $K$ is ${\cal M}_{[\eta]\cap K}$ by Proposition \ref{Prop-Minkowski} (iii). Since (\ref{decomposition-loop-K}) holds for any compact set $K\subset\R\sem\{0\}$, we get
\BGE \til\mu^{\HH}_{0\ccw }(d[\eta]) \rotimes {\cal M}_{\eta \cap (\R\sem\{0\})}(d u) = \til \nu^{\HH}_{0\lr u} (d[\eta]) \lotimes {\ind}_{\R\sem\{0\}} C_{\kappa;\rho} |u|^{-2\alpha}\cdot \mA(d  u). \label{mu-decomp-R+}\EDE
This is slightly different from  (\ref{mu-decomp-R}) because we have not proved the existence of  ${\cal M}_{[\eta]\cap\R}$.

Assume $[\eta]$ follows the ``law'' $\til\mu^{\HH}_{0\ccw}$. Then $[\eta]\cap\R$ is a compact set, and $0$ is its accumulation point by Theorem \ref{Thm-loop-sigma-finite}. Since for each compact $K\subset\R\sem\{0\}$,   ${\cal M}_{[\eta]\cap K}$ is atomless with  minimal closed support ${[\eta]\cap K}$,   ${\cal M}_{[\eta]\cap(\R\sem\{0\})}$ is also atomless, and its minimal closed support is $\lin{[\eta]\cap(\R\sem\{0\})}$, which equals the compact set $[\eta]\cap\R$.

Let $\Gamma$ denote the space of MTC curves $[\eta]$ such $[\eta]\cap\R$ is a nonempty compact set, and ${\cal M}_{[\eta]\cap\R}$ exists, is atomless, and its minimal closed support is $[\eta]\cap\R$. If $[\eta]\in\Gamma$, then by the definition of Minkowski content measure and the fact that $[\eta]\cap\R$ is compact, $\Cont_d([\eta]\cap\R)$ exists and equals the total mass of ${\cal M}_{[\eta]\cap\R}$, which is finite. Since the minimal closed support of  ${\cal M}_{[\eta]\cap\R}$ is $[\eta]\cap\R$, which is nonempty, ${\cal M}_{[\eta]\cap\R}$ is strictly positive. Thus, $\Cont_d([\eta]\cap\R)\in (0,\infty)$.

For $x\in\R$, let $\Gamma_x$ denote the space of MTC curves $[\eta]$ such that ${\cal M}_{[\eta]\cap(\R\sem\{x\})}$ exists and is atomless, and the minimal closed support of ${\cal M}_{[\eta]\cap(\R\sem\{x\})}$ is the compact set $[\eta]\cap\R$. By Proposition \ref{Prop-Minkowski} (iii),  $\Gamma\subset\Gamma_x$ for every $x\in\R$.
By Lemma \ref{Mink-exist-S12}, for any $x\ne y\in\R$, $\Gamma_x\cap\Gamma_y\subset\Gamma$.

We have known that $\til\mu^{\HH}_{0\ccw}$ is supported by $\Gamma_0$. By (\ref{mu-decomp-R+}), for $\mA$-a.s.\ every $u\in\R\sem\{0\}$, $ \til \nu^{\HH}_{0\lr u}$ is supported by $\Gamma_0$. By scaling, this holds for every $u\in\R\sem\{0\}$. By translation, for every $u\in\R\sem\{0\}$, $ \til \nu^{\HH}_{u\lr 0}$ is supported by $\Gamma_u$. By (\ref{unrooted=}), every $ \til \nu^{\HH}_{0\lr u}$ is supported by both $\Gamma_0$ and $\Gamma_u$, and so is supported by $\Gamma$ since $\Gamma_0\cap\Gamma_u\subset\Gamma$.  Then we use (\ref{mu-decomp-R+}) again to conclude that $\til\mu^{\HH}_{0\ccw}$ is also supported by $\Gamma$. Moreover, we have $\til\mu^{\HH}_{0\ccw}$-a.s.\ ${\cal M}_{[\eta]\cap\R}={\cal M}_{[\eta]\cap(\R\sem\{0\})}$. Thus, we get  (\ref{mu-decomp-R}) from (\ref{mu-decomp-R+}). By observing the marginal measures on curves in (\ref{mu-decomp-R}) we get
$$\Cont_d([\eta]\cap\R)\cdot \til\mu^{\HH}_{0\ccw}(d[\eta])= \int_{\R\sem\{0\}} \til\nu^{\HH}_{0\lr u}  C_{\kappa;\rho} |u|^{-2\alpha} \mA(du),$$
which   implies (\ref{express-rooted-loop}) because $\Cont_d([\eta]\cap\R)\in (0,\infty)$ for $ \til\mu^{\HH}_{0\ccw}$-a.s.\ every $[\eta]$.
\end{proof}

\begin{Theorem} Let  $\til\mu^{\HH}_{\infty\ccw} $ be as in Theorem \ref{scaling-invariance}. Then
  \begin{itemize}
  \item [(i)] $\til\mu^\HH_{\infty\ccw}$ is supported by the set of Loewner bubbles $[\eta]$ in $\HH$ rooted at $\infty$ such that
  \begin{itemize}
    \item [(a)]  the $d$-dimensional Minkowski content measure  ${\cal M}_{[\eta]\cap\R}$ exists;
    \item [(b)] the minimal closed support of ${\cal M}_{[\eta]\cap\R}$ is $[\eta]\cap\R$;
  \end{itemize}
   \item [(ii)] The following equality holds:
\BGE \til\mu^{\HH}_{\infty\ccw }(d[\eta]) \rotimes {\cal M}_{[\eta] \cap \R}(d u) = \til \nu^{\HH}_{\infty\lr u} (d[\eta]) \lotimes   C_{\kappa;\rho}  \mA(d  u).\label{mu-decomp-R-infty}\EDE
\end{itemize}
\label{Thm-decomposition-loop-infty}
\end{Theorem}
\begin{proof}
We have $\til\mu^{\HH}_{\infty\ccw}=J(\til\mu^{\HH}_{\infty\ccw}) $, where $J(z):=-1/z$. Part (i) follows Theorem \ref{Thm-decomposition-loop} (i) and Proposition \ref{Prop-Minkowski} (iv).
Weighting both sides of (\ref{mu-decomp-R}) by $|J'(u)|^d=|u|^{-2d}$ and then applying the map $([\eta],u)\mapsto ([J\circ \eta],J(u))$, we get (\ref{mu-decomp-R-infty}) using  $\alpha+d=1$.
 \end{proof}

\begin{Remark}
  We could use (\ref{express-rooted-loop}) to define the rooted SLE$_\kappa(\rho)$ bubble measure for $\kappa\in (0,8)$ and $\rho\in((-2)\vee(\frac\kappa 2-4),\frac\kappa 2-2)$. This is similar to the definition of  SLE$_\kappa$ loops in \cite{loop}.
\end{Remark}

\begin{Remark}
In the case $\kappa\ge 8$ and $\rho\in( \frac\kappa 2-4,\frac\kappa 2-2)$, we could define $\til\nu^{\HH}_{0\lr \infty}$ using (\ref{DMP-two-loop-0}), which does not satisfy reversibility. 
Theorem \ref{Thm-decomposition-loop} holds in this case  with slight changes: we have the existence of ${\cal M}_{[\eta]\cap(\R\sem \{0\})}$ but may not have the existence of ${\cal M}_{[\eta]\cap\R}$. We also may not have (\ref{express-rooted-loop}).
 Theorem \ref{Thm-decomposition-loop-infty} holds in the case $\kappa\ge 8$ without changes.
\end{Remark}

\subsection{Unrooted SLE$_\kappa(\rho)$ bubbles} \label{section-unrooted}
We still assume that $\kappa\in (0,8)$ and $\rho\in((-2)\vee(\frac\kappa 2-4),\frac\kappa 2-2)$.
Recall the family $\Gamma$ in the proof of Theorem \ref{Thm-decomposition-loop} and the facts that $\til\mu^{\HH}_{0\ccw}$ and $\til\nu^{\HH}_{0\lr u}$ are supported by $\Gamma$ for any $u\in\R\sem\{0\}$. 
By translation, for any $w\ne u\in\R$, $\til\mu^{\HH}_w$ and $\til\nu^{\HH}_{w\lr u}$ are supported by $\Gamma$, and so are supported by MTC curves $[\eta]$ such that $\Cont_d([\eta]\cap\R)$ exists and lies in $(0,\infty)$.

Recall that $\til\pi$ is the projection from rooted MTC loops to unrooted MTC loops.  Let
$$\ha\mu^{\HH}_{w\ccw}=\til\pi(\til \mu^{\HH}_{w\ccw }),\quad\ha  \nu^{\HH}_{w\lr u}= \til\pi(\til \nu^{\HH}_{w\lr u} ),\quad
\ha\mu^{\HH}_{\ccw}=\Cont_d(\cdot \cap\R)^{-1}\cdot \int_{\R}\ha\mu^{\HH}_{w\ccw} \mA(dw).$$
where the definition of $\ha\mu^{\HH}_{\ccw}$ makes sense because for every $w\in\R$, $\ha\mu^{\HH}_{w\ccw}$ is supported by unrooted MTC loops $[[\eta]]$ such that  $\Cont_d([[\eta]]\cap\R)$ exists and lies in $(0,\infty)$.
By (\ref{express-rooted-loop}) we have
\BGE \ha\mu^{\HH}_{\ccw}=\Cont_d(\cdot \cap\R)^{-2}\cdot \int_{\R}\int_{\R\sem \{w\}}  \ha\nu^{\HH}_{w\lr v} C_{\kappa;\rho} |u-w|^{-2\alpha}  \mA(d  u) \mA(dw).\label{double}\EDE
We call $\ha\mu^{\HH}_{\ccw}$ the unrooted SLE$_\kappa(\rho)$ bubble measure in $\HH$. Then (i) $\ha\mu^{\HH}_{\ccw}$ is supported by $\til\pi(\Gamma)$; and (ii) for any  $x_0\in\R$, $\ha\mu^{\HH}_{\ccw}$ is supported by those $[[\eta]]$ such that $x_0\not\in [[\eta]]$ since for every $w\in\R\sem\{x_0\}$, $\ha\mu^{\HH}_{w\ccw}$ satisfies this property, which follows from the DMP formula  (\ref{DMP-kappa-rho-proob})  and the boundary behavior of SLE$_\kappa(\rho)$ for $\rho>\frac\kappa 2-4$.

\begin{Theorem}
  We have the following two equalities:
   \BGE \ha\mu^{\HH}_{\ccw} (d[[\eta]]) \rotimes {\cal M}_{[[\eta]] \cap \R}(d u)= \ha\mu^{\HH}_{u\ccw}(d[[\eta]]) \lotimes  \mA(du); \label{unrooted-decomp1}\EDE
 \BGE \ha\mu^{\HH}_{\ccw} (d[[\eta]]) \rotimes {\cal M}_{[[\eta]] \cap \R}^2(dw\otimes du)= \ha\nu^{\HH}_{w\lr u}(d[[\eta]])\lotimes C_{\kappa;\rho}\ind_{\{w\ne u\}}  |w-u|^{-2\alpha}\cdot \mA^2(dw\otimes du).\label{unrooted-decomp2}\EDE
Moreover, if $f$ is a conformal automorphism of $\HH$, then $f(\ha\mu^{\HH}_{\ccw} )=\ha\mu^{\HH}_{\ccw}$.
\label{Thm-unrooted}
\end{Theorem}
\begin{proof}
Acting $\til\pi$ to (\ref{mu-decomp-R}) and using translation, we find that for any $w\in\R$,
\BGE \ha\mu^{\HH}_{w\ccw }(d[[\eta]]) \rotimes {\cal M}_{[[\eta]] \cap \R}(d u) = \ha \nu^{\HH}_{w\lr u} (d[[\eta]]) \lotimes {\ind}_{\R\sem\{w\}}  C_{\kappa;\rho} |u-w|^{-2\alpha}\cdot \mA(d  u) .\label{mu-decomp-R-w}\EDE
Integrate both sides of the above formula against $\mA(dw)$. The integral of the LHS is
$$\Big(\int_{\R} \ha\mu^{\HH}_{w\ccw }(d[[\eta]]) \mA(dw)\Big) \rotimes {\cal M}_{[[\eta]] \cap \R}(d u) =\Cont_d([[\eta]])\cdot \ha\mu^{\HH}_{\ccw} (d[[\eta]]) \rotimes {\cal M}_{[[\eta]] \cap \R}(d u).$$
By (\ref{unrooted=},\ref{express-rooted-loop}) and translation, the integral of the RHS is
$$\Big(\int_{\R\sem\{u\}} \ha \nu^{\HH}_{u\lr w}  C_{\kappa;\rho} |w-u|^{-2\alpha}\mA(dw)\Big)\lotimes  \mA(du)=\Cont_d([[\eta]])\cdot \ha\mu^{\HH}_{u\ccw}(d[[\eta]]) \lotimes  \mA(du). $$
Comparing the two displayed formulas and using the facts that $\ha\mu^{\HH}_{\ccw}$ and $\ha\mu^{\HH}_{w\ccw}$ are supported by unrooted loops $[[\eta]]$ such that $\Cont_d([[\eta]]\cap\R)\in(0,\infty)$, we get (\ref{unrooted-decomp1}).

Act $\rotimes {\cal M}_{[[\eta]]\cap\R}(dw)$ on both sides of (\ref{unrooted-decomp1}) from their right and then switch the order of the second and third components, i.e., $du$ and $dw$. On the LHS, we get
$$\ha\mu^{\HH}_{\ccw} (d[[\eta]]) \rotimes {\cal M}_{[[\eta]] \cap \R}^2(dw\otimes du).$$
On the RHS, by (\ref{mu-decomp-R-w}) we get
\begin{align*}
&(\ha\mu^{\HH}_{u\ccw}(d[[\eta]]) \rotimes {\cal M}_{[[\eta]] \cap \R}(dw))  \lotimes  \mA(du)\\
= &( \ha \nu^{\HH}_{u\lr w} (d[[\eta]]) \lotimes {\ind}_{\R\sem\{u\}}  C_{\kappa;\rho} |w-u|^{-2\alpha}\cdot \mA(d w))   \lotimes  \mA(du)\\
=& \ha\nu^{\HH}_{w\lr u}(d[[\eta]])\lotimes C_{\kappa;\rho} \ind_{\{w\ne u\}}   |w-u|^{-2\alpha}\cdot \mA^2(dw\otimes du).
\end{align*}
Comparing the above two displayed formulas, we get (\ref{unrooted-decomp2}).

Let $f$ be a conformal automorphism of $\HH$. From (\ref{unrooted-decomp2}) and the formulas $|f(w)-f(u)|=|f'(w)|^{1/2}|f'(u)|^{1/2}|w-u|$ and $\alpha+d=1$ we get
$$ \ha\mu^{\HH}_{\ccw} (d[[\eta]]) \rotimes \ind_{\{w,u\ne f^{-1}(\infty)\}}  |f'(w)|^d |f'(u)|^d\cdot {\cal M}_{[[\eta]] \cap \R}^2(dw\otimes du) $$
$$= \ha\nu^{\HH}_{w\lr u}(d[[\eta]])\lotimes C_{\kappa;\rho}  \ind_{\{w,u\ne f^{-1}(\infty)\}}  |f(w)-f(u)|^{-2\alpha} |f'(w)| |f'(u)| \cdot  \mA^2(dw\otimes du)$$
Let $\check \mu^{\HH}_{\ccw}=f(\ha\mu^{\HH}_{\ccw})$. Since $\ha\mu^{\HH}_{\ccw}$-a.s.\ $[[\eta]]\cap\{ \infty, f^{-1}(\infty)\}= \emptyset$,  $\check \mu^{\HH}_{\ccw}$ is supported by $\til\pi(\Gamma)$.
Applying the map $([[\eta]],w,u)\mapsto ([[\check \eta]],\check w,\check u)=([[f\circ \eta]],f(w),f(u))$ to the above displayed formula and using Proposition \ref{Prop-Minkowski} and the conformal invariance of $\til\nu^{\HH}_{\cdot \lr\cdot}$, we get the equality
$$\check\mu^{\HH}_{\ccw}(d[[\check\eta]])\rotimes {\cal M}_{[[\check \eta]]\cap(\R\sem \{f(\infty)\})}^2 (d\check w\otimes d\check u)$$ $$= \ha\nu^{\HH}_{\check w\lr \check u}(d[[\check \eta]])\lotimes C_{\kappa;\rho} \ind_{\{\check w,\check u\ne f(\infty)\}}  |\check w-\check u|^{-2\alpha}\cdot \mA^2(d\check w\otimes d\check u). $$
Since $\check\mu^{\HH}_{\ccw}$-a.s.\ ${\cal M}_{[[\check \eta]]\cap\R}$ exists and is atomless, the equality holds with $\R$ in place of $\R\sem \{f(\infty)\}$. Computing the marginal measures of both sides on unrooted loops and using (\ref{double}), we then get $\check\mu^{\HH}_{\ccw}=\ha\mu^{\HH}_{\ccw}$, i.e.,   $f(\ha\mu^{\HH}_{\ccw})=\ha\mu^{\HH}_{\ccw}$.
\end{proof}

\begin{Remark}
  We may similarly define clockwise unrooted SLE$_\kappa(\rho)$ bubble measure $\ha\mu^{\HH}_{\cw}$. By Corollary \ref{Reversibility-cor} we get the reversibility equality: ${\cal R}(\ha\mu^{\HH}_{\ccw})=\ha\mu^{\HH}_{\cw}$. \label{Remark-R-unrooted}
\end{Remark}

\begin{Remark}
   Since $\ha\mu^{\HH}_{\ccw}$ and $\ha\mu^{\HH}_{\cw}$ are invariant under M\"obius automorphism of $\HH$, for any simply connected domain $D$ we may define the ccw (resp.\ cw) unrooted SLE$_\kappa(\rho)$ bubble measure in $D$ by $\ha\mu^{D}_{\ccw}=f(\ha\mu^{\HH}_{\ccw})$ (resp.\ $\ha\mu^{D}_{\cw}=f(\ha\mu^{\HH}_{\cw})$)  using any conformal map $f$ from $\HH$ onto $D$. \label{Remark-unrooted-general}
\end{Remark}

\end{document}